\documentclass[hidelinks,journal]{IEEEtran}

\IEEEoverridecommandlockouts
\usepackage{etex}
\usepackage{wrapfig}
\usepackage{multicol}

\usepackage{comment}
\usepackage{siunitx}
\usepackage{relsize}
\usepackage{ifthen}

\usepackage[caption=false]{subfig}

\usepackage{graphics} 
\usepackage{rotating}
\usepackage{color}
\usepackage{enumerate}
\usepackage[T1]{fontenc}
\usepackage{psfrag}
\usepackage{epsfig} 
\usepackage{booktabs}
\usepackage{graphicx,url}
\usepackage{multirow}
\usepackage{array}
\usepackage{latexsym}
\usepackage{amsfonts}
\usepackage{amsmath}
\usepackage{amssymb}
\usepackage{xstring}
\usepackage{multirow}
\usepackage{xcolor}
\usepackage{prettyref}
\usepackage{flexisym}
\usepackage{bigdelim}
\usepackage{breqn} 
\usepackage{listings}
\let\labelindent\relax
\usepackage{enumitem}
\usepackage{xspace}
\usepackage{bm}
\graphicspath{{./figures/}}
\usepackage{tikz}
\usetikzlibrary{matrix,calc}

\usepackage{lineno}

\usepackage{mdwlist}

\let\itemize\undefined
\makecompactlist{itemize}{stditemize}

\newrefformat{prob}{Problem\,\ref{#1}}
\newrefformat{def}{Definition\,\ref{#1}}
\newrefformat{sec}{Section\,\ref{#1}}
\newrefformat{sub}{Section\,\ref{#1}}
\newrefformat{prop}{Proposition\,\ref{#1}}
\newrefformat{app}{Appendix\,\ref{#1}}
\newrefformat{alg}{Algorithm\,\ref{#1}}
\newrefformat{cor}{Corollary\,\ref{#1}}
\newrefformat{thm}{Theorem\,\ref{#1}}
\newrefformat{lem}{Lemma\,\ref{#1}}
\newrefformat{fig}{Fig.\,\ref{#1}}
\newrefformat{tab}{Table\,\ref{#1}}

\newcommand{\bdmath}{\begin{dmath}}
\newcommand{\edmath}{\end{dmath}}
\newcommand{\beq}{\begin{equation}}
\newcommand{\eeq}{\end{equation}}
\newcommand{\bdm}{\begin{displaymath}}
\newcommand{\edm}{\end{displaymath}}
\newcommand{\bea}{\begin{eqnarray}}
\newcommand{\eea}{\end{eqnarray}}
\newcommand{\beal}{\beq \begin{array}{lll}}
\newcommand{\eeal}{\end{array} \eeq}
\newcommand{\beas}{\begin{eqnarray*}}
\newcommand{\eeas}{\end{eqnarray*}}
\newcommand{\ba}{\begin{array}}
\newcommand{\ea}{\end{array}}
\newcommand{\bit}{\begin{itemize}}
\newcommand{\eit}{\end{itemize}}
\newcommand{\ben}{\begin{enumerate}}
\newcommand{\een}{\end{enumerate}}

\newcommand{\algS}{\widehat{\calS}}

\newcommand{\calB}{{\cal B}}

\newcommand{\calG}{{\cal G}}
\newcommand{\calH}{{\cal H}}

\newcommand{\calS}{{\cal S}}

\newcommand{\calV}{{\cal V}}

\definecolor{myblue}{RGB}{65 105 225}

\newcommand{\hide}[1]{}

\newcommand{\hiddenText}{{\color{gray} hidden text.}}
\newcommand{\hideWithText}[1]{\hiddenText}

\DeclareMathOperator*{\argmin}{arg\,min}

\newcommand{\tran}{^{\mathsf{T}}}

\newcommand{\diag}[1]{\mathrm{diag}\left(#1\right)}
\newcommand{\trace}[1]{\mathrm{tr}\left(#1\right)}

\newcommand{\inv}{^{-1}}

\newcommand{\until}[1]{\{1, 2\dots, #1\}}

\newcommand{\eye}{{\mathbf I}}

\newcommand{\Real}[1]{ { {\mathbb R}^{#1} } }

\newcommand{\att}{^{(t)}}
\newcommand{\at}[1]{^{(#1)}}

\newcommand{\omitted}[1]{}

\newcommand{\LQG}{LQG\xspace}
\newcommand{\myParagraph}[1]{{\bf #1.}\xspace}

\newcommand{\logdet}{\log\det}
\newcommand{\myFigure}[1]{Fig.~\ref{#1}}
\newcommand{\nrRobots}{n}
\newcommand{\sensorBudget}{b}
\newcommand{\sensorCost}{c}
\newcommand{\optS}{\calS^\star}
\newcommand{\algSone}{\algS_1}
\newcommand{\algStwo}{\algS_2}
\newcommand{\algStwoAlpha}{\algS_{2,\alpha}}
\newcommand{\algSAlpha}{\algS_{\alpha}}
\newcommand{\optBudget}{\sensorBudget^\star}

\newcommand{\trandom}{{\tt random$^*$}\xspace}
\newcommand{\tlogdet}{{\tt logdet}\xspace}
\newcommand{\tslqg}{{\tt s-LQG}\xspace}
\newcommand{\toptimal}{{\tt optimal}\xspace}
\newcommand{\tallSensors}{{\tt allSensors}\xspace}

\newcommand{\of}[2]{_{#1}(#2)}
\newcommand{\initialCovariance}{\Sigma\att{1}{1}}
\newcommand{\allT}{t=1,2,\ldots,T}

\newcommand{\validated}[2]{{#2}}

\newcommand{\subjectTo}{\rm{s.t.}}

\newcommand{\elem}{{{v}}}

\newcommand{\optUoneT}{u^\star_{1:T}} 
\newcommand{\hatUoneT}{\widehat{u}_{1:T}} 
\newcommand{\temp}{{{\beta}}}

\usepackage{float}

\usepackage{hyperref}
\usepackage{graphicx}
\usepackage{epstopdf}
\usepackage{epsfig}
\usepackage{pgfplotstable}
\usepackage{pgfplots}

\newcommand{\myparagraph}[1]{{\bf #1.}}

\graphicspath{{./figures/},{../figures/},{../},{./code/}}

\usepackage{xr}
\usepackage{cite}
\usepackage{amsthm}
\newtheoremstyle{mystyle}
  {}
  {}
  {\itshape}
  {}
  {\bfseries}
  {.}
  { }
  {\thmname{#1}\thmnumber{ #2}\thmnote{ (#3)}}
\theoremstyle{mystyle}

\floatname{algorithm}{Algorithm}

\newtheorem{mydef}{Definition}

\newtheorem{mytheorem}{Theorem}
\newtheorem{mylemma}{Lemma}
\newtheorem{myremark}{Remark}

\newtheorem{myproblem}{Problem}

\renewcommand{\at}[1]{_{#1}}
\renewcommand{\att}[2]{_{#1|#2}}
\renewcommand{\trace}[1]{\textrm{tr}\left( #1\right)}

\usepackage{algorithm}
\usepackage{algcompatible}

\title{\huge{\LQG Control and Sensing Co-Design}}

\author{
Vasileios Tzoumas,$^{1}$~\IEEEmembership{Member,~IEEE,} Luca Carlone,$^{1}$~\IEEEmembership{Senior Member,~IEEE,} George J.~Pappas,$^{2}$~\IEEEmembership{Fellow,~IEEE,} \\Ali Jadbabaie,$^{1}$~\IEEEmembership{Fellow,~IEEE}
\thanks{$^{1}$The authors are with the Laboratory for Information \& Decision Systems, Massachusetts Institute of Technology, Cambridge, MA 02139, USA (email: {\fontsize{8}{8}\selectfont\ttfamily\upshape \{vtzoumas, lcarlone, jadbabai\}@mit.edu}).}
\thanks{$^{2}$The author is with the Department of Electrical and Systems Engineering, University of Pennsylvania, Philadelphia, PA 19104, USA (email: {\fontsize{8}{8}\selectfont\ttfamily\upshape pappasg@seas.upenn.edu}).}
\thanks{This work was supported  
by the AFOSR Complex Networks Program, ARL DCIST CRA W911NF-17-2-0181, and ONR RAIDER N00014-18-1-2828.
}
}

\let\OLDthebibliography\thebibliography
\renewcommand\thebibliography[1]{
  \OLDthebibliography{#1}
  \setlength{\parskip}{0pt}
  \setlength{\itemsep}{0pt}
}

\usepackage{tikz}
\begin{document}

\maketitle

%

\begin{abstract}
We investigate a Linear Quadratic Gaussian (LQG) control and sensing co-design problem, where one jointly designs sensing and control policies. 
We~focus on the realistic case where the sensing design 
is selected among a finite set of available sensors, where each sensor is  associated with a different cost {(e.g., power consumption)}.  We consider two dual problem instances: \emph{sensing-constrained \LQG control}, 
where one maximizes control performance subject to a sensor cost budget, 
and \emph{minimum-sensing \LQG control}, where one minimizes sensor cost subject to performance constraints. 
{We prove no polynomial time algorithm guarantees across all problem instances a {constant} approximation factor from the optimal.}
Nonetheless, we~present the first polynomial time {algorithms with {{per-instance}} 
 suboptimality guarantees.}  
To this end,~we {leverage} a separation principle, that partially decouples  the design of sensing and control.  
Then, we frame \LQG co-design as the optimization of approximately supermodular set functions;~we develop novel algorithms to solve the problems;~and we prove original results on the performance of the algorithms, and establish connections between their suboptimality and control-theoretic quantities.  
We conclude the paper by discussing two applications, namely,
\emph{sensing-constrained formation control} and 
\mbox{{\emph{resource-constrained robot navigation}.}
}
\end{abstract}

\begin{tikzpicture}[overlay, remember picture]
\path (current page.north east) ++(-4,-0.2) node[below left] {
	This paper has been accepted for publication in the IEEE Transactions of Automatic Control.
};
\end{tikzpicture}
\begin{tikzpicture}[overlay, remember picture]
\path (current page.north east) ++(-4.8,-0.6) node[below left] {
	Please cite the paper as:
	V.~Tzoumas, L.~Carlone, George J.~Pappas, A.~Jadbabaie
};
\end{tikzpicture}
\begin{tikzpicture}[overlay, remember picture]
\path (current page.north east) ++(-3.9,-1) node[below left] {
	``LQG Control and Sensing Co-Design", IEEE Transactions of Automatic Control (TAC), 2020.
};
\end{tikzpicture}


\section{Introduction}\label{sec:Intro}

Traditional approaches to systems control  assume the choice of
sensors fixed~\cite{bertsekas2005dynamic}.
The  sensors usually result from a preliminary design phase in which an 
expert selects a suitable sensor suite that accommodates estimation requirements, and system constraints (e.g., power consumption).
{However, the control applications of the Internet of Things (IoT) and Battlefield Things (IoBT)~\cite{abdelzaher2018toward},
pose serious limitations to the applicability of this traditional paradigm.  Now, systems are not designed from scratch; instead, existing, standardized systems come together, along with their sensors, to form heterogeneous teams (such as robot teams), 
 tasked with various control goals: from collaborative object manipulation to formation control~\cite{michael2011cooperative,prorok2017impact}.  In such heterogeneous networked systems, where new nodes are continuously 
added and removed from the network, sensor redundancies 
are created, depending on the task at hand. 
At the same time, power, bandwidth, and/or computation constraints limit which sensors can be active~\cite{gupta2006stochastic,carlone2017attention,iwaki2018lqg}   Therefore, to optimize the network's operability and prolong its operation, one needs to decide which sensors are important for the task, and activate only these~\cite{iwaki2018lqg}.
Evidently, in large-scale networks
a manual activation policy is not scalable.}
Thus, ones needs to develop 
automated approaches.
Motivated by this need,
we consider the co-design of \LQG control and sensor selection subject to sensor activation constraints.  {Particularly, we assume that  the sensor constraints are captured by a prescribed budget (e.g., available battery power), and that each sensor is associated with \mbox{an activation cost (e.g., power consumption).}}

\myParagraph{Related work in control} {Traditionally, the control literature  has focused on co-designing control, estimation, actuation (i.e., actuator selection), and sensing (i.e., sensor selection)~\cite{bertsekas2005dynamic,Nair07ieee-rateConstrainedControl,Baillieul07ieee-networkedControl,Elia01tac-limitedInfoControl,Nair04sicon-rateConstrainedControl,Tatikonda04tac-limitedCommControl,Borkar97-limitedCommControl,LeNy14tac-limitedCommControl,
lin2017sparse,lin2011augmented,lin2013design,zare2018proximal,liu2017decentralized,
tanaka2015sdp,tanaka2018lqg,
joshi2009sensor,gupta2006stochastic,leny2011kalman,
jawaid2015submodularity,zhao2016scheduling,chamon2017mean,
tzoumas2015sensor,carlone2017attention,
clark2012leader,clark2017input,liu2017minimal,pequito2015framework,summers2016submodularity,
tzoumas2015minimal,summers2017performance,summers2017performance2,
nozari2017time,taha2017time,clark2016submodularity,
chakrabortty2011optimal,moreno2015actuator,lim1992method,
iwaki2018lqg}.  However, the focus so far has mostly been different from the co-design problem we consider in this paper}: 

\textit{a)}{~\cite{bertsekas2005dynamic,Nair07ieee-rateConstrainedControl,Baillieul07ieee-networkedControl,Elia01tac-limitedInfoControl,Nair04sicon-rateConstrainedControl,Tatikonda04tac-limitedCommControl,Borkar97-limitedCommControl,LeNy14tac-limitedCommControl} assume all sensors given and active (instead of choosing a few sensors to activate).}  They focus on the co-design of control and estimation over band-limited communication channels, and 
investigate trade-offs between communication constraints (e.g., quantization), and control performance (e.g., stability). In more detail, they provide results on the impact of quantization~\cite{Elia01tac-limitedInfoControl}, and of finite data 
rates~\cite{Nair04sicon-rateConstrainedControl,Tatikonda04tac-limitedCommControl}, as well as, on separation principles for \LQG design with 
communication constraints~\cite{Borkar97-limitedCommControl}.  Recent works also focus on privacy constraints~\cite{LeNy14tac-limitedCommControl}.
For a comprehensive review on \LQG control and estimation, \mbox{we refer to~\cite{bertsekas2005dynamic,Nair07ieee-rateConstrainedControl,Baillieul07ieee-networkedControl}.} 

{\textit{b)}~\cite{lin2017sparse,lin2011augmented,lin2013design,zare2018proximal,liu2017decentralized,tanaka2015sdp,tanaka2018lqg} extend the focus of the above works, by focusing on the co-design of control, estimation, and sensing.  Yet, the choice of each sensor can be arbitrary (instead, in our framework, a few sensors are activated from a given finite set of available ones).  For example,~\cite{lin2017sparse,lin2011augmented,lin2013design,zelazo2010graph,zare2018proximal,liu2017decentralized} propose the optimization of steady state \LQG costs, subject to sparsity constraints on the sensor matrices and/or on the feedback control and estimation gains. Finally,~\cite{tanaka2015sdp,tanaka2018lqg} augment the  \LQG cost with an information-theoretic regularizer, and design the sensors matrices
using semi-definite programming.  
}

\textit{c)}~\cite{joshi2009sensor,gupta2006stochastic,leny2011kalman,jawaid2015submodularity,
zhao2016scheduling,chamon2017mean,tzoumas2015sensor,
carlone2017attention,clark2012leader,clark2017input,liu2017minimal,pequito2015framework,summers2016submodularity,tzoumas2015minimal,summers2017performance,summers2017performance2,nozari2017time,taha2017time,clark2016submodularity,
chakrabortty2011optimal,moreno2015actuator,lim1992method,
iwaki2018lqg} focus on sensor selection, but they do not consider control aspects (with the exception of~\cite{moreno2015actuator,lim1992method,iwaki2018lqg}, which we discuss below).  Specifically,~\cite{joshi2009sensor} studies sensor placement to optimize maximum likelihood estimation over static parameters, whereas \cite{gupta2006stochastic,leny2011kalman,jawaid2015submodularity,zhao2016scheduling,chamon2017mean,tzoumas2015sensor,carlone2017attention} focus on optimizing Kalman filtering and batch estimation accuracy over non-static parameters.
\cite{clark2012leader,clark2017input,pequito2015framework,summers2016submodularity,tzoumas2015minimal,summers2017performance,summers2017performance2,nozari2017time,taha2017time} present sensor and actuator selection algorithms to optimize the average observability and controllability of systems;~\cite{liu2017minimal} focuses on actuator placement for stability in uncertain systems.  For additional relevant applications, we refer to~\cite{clark2016submodularity}. {\cite[Chapter 6.1.3]{zelazo2010graph} focuses on selecting a sensor for each edge of a consensus-type system for $\calH_2$ optimization subject to sensor cost constraints, and sensor noise considerations (instead, we consider general systems). \cite{chakrabortty2011optimal} selects the location of a phasor measurement unit (PMU) on a single edge of an electrical network to minimize estimation error (each placement happens independently of the rest). \cite{moreno2015actuator,lim1992method} study sensor placement to optimize a steady state \LQG cost; although the latter case is similar to our framework (we optimize a finite horizon \LQG cost, instead of a steady state), the authors focus only on a small-scale system with a few sensors, where a brute-force selection is viable, and no scalable algorithms are proposed (instead, our focus is on scalable approximation algorithms).
Finally,~\cite{iwaki2018lqg} studies an \LQG control and scheduling co-design problem, where \textit{decoupled} systems share a wireless sensor network, while power consumption constraints must be satisfied. {Instead, we consider coupled systems, a framework that makes our co-design problem inapproximable in polynomial time, in contrast to~\cite{iwaki2018lqg}'s, which is optimally solved in polynomial time.}}

\myParagraph{Related work on set function optimization} 
{In this paper, a few sensors must be activated among a set of available ones.  This is a combinatorial problem, and we prove it inapproximable: across all problem instances, no polynomial time algorithm can guarantee a constant approximation factor from the optimal. Thus, to provide efficient algorithms with per-instance suboptimality bounds instead, we resort to tools from \textit{combinatorial} optimization, which has been a successful paradigm on this front~\cite{clark2016submodularity,das2011submodular,wang2016approximation,sviridenko2013optimal,sviridenko2017optimal,clark2016submodularity}.}
Specifically, the literature on combinatorial optimization includes investigation into (i)~\textit{supermodular} optimization subject to \textit{cardinality} constraints {(where only a prescribed number of sensors can be active)}~\cite{nemhauser78analysis,wolsey1982analysis}; (ii)~\textit{supermodular} optimization subject to \textit{cost} constraints~\cite{khuller1999budgeted,sviridenko2004,krause2005note} {(where only sensor combinations that meet a prescribed budget can be active ---each sensor has a potentially different activation cost)};
and~(iii)~{\textit{approximately} \textit{supermodular} optimization subject to \textit{cardinality} constraints~\cite{das2011submodular,sviridenko2013optimal,sviridenko2017optimal,wang2016approximation}.}
The literature does not cover
\textit{approximately} \textit{submodular} optimization subject to \textit{cost} constraints, which is the setup of interest in this paper; hence we herein develop algorithms and novel suboptimality bounds for this case.\footnote{{The transition from cardinality to cost constraints, in terms of providing efficient algorithms with provable suboptimality bounds, is non-trivial, as it is observed by comparing the widely different proof techniques in~\cite{tzoumas2018sensing}, for the cardinality case, versus those considered in this paper, for the cost case.}}

\myParagraph{Contributions to control theory}
We address {an} \emph{\LQG control and sensing co-design} problem.  The problem 
extends \LQG control to the case where, besides designing an optimal controller and estimator,  {one has to decide which sensors to activate, 
due to sensor cost constraints and a limited budget}.
That is, the sensor choice is restricted to a finite selection from the available sensors, rather than being arbitrary ({for arbitrary sensing design, see~\cite{tanaka2015sdp, lin2017sparse,lin2011augmented,lin2013design,zare2018proximal}}).  And each sensor has a cost that captures the penalty incurred for using the sensor. 
{Since different sensors (e.g., lidars, radars, cameras, lasers) have different power consumption, bandwidth utilization, and/or monetary value, we allow each sensor to have a different cost.}
We formulate two dual instances of the {\LQG co-design} problem.
The first, \emph{sensing-constrained \LQG control}, 
involves the joint design of control and sensing 
to minimize the \LQG cost {subject to a sensor cost budget}.  
The second, \emph{minimum-sensing \LQG control},
involves the joint design of control and sensing
to minimize the cost of the activated sensors subject to a desired \LQG cost.  

To solve the proposed \LQG problems, we first leverage a separation 
principle that partially decouples the control and sensor selection.\footnote{{The separation between control and \textit{sensor selection} is proved with the same steps as the separation of control and \textit{estimation} in standard \LQG control theory; e.g., see proof of Lemma~1 in~\cite{tanaka2015sdp}.}}  {As a negative result, we prove  the optimal sensor selection is inapproximable in polynomial time by a constant suboptimality bound across all problem instances.} Therefore, we develop algorithms with per-instance suboptimality bounds instead.  Particularly, we frame the sensor selection  
as the optimization of \textit{approximately} supermodular set functions, {using the notion of supermodularity ratio introduced in 2006 in \cite{lehmann2006combinatorial} (see also~\cite{wang2016approximation})}.\footnote{{The notion has met already increasing interest in the signal processing and control literature; see, for example,~\cite{chamon2016near,chamon2017mean, hashemi2019submodular,guo2019actuator,sviridenko2013optimal,sviridenko2017optimal}.} }
Then, we provide the first polynomial time algorithms, which provably retrieve a close-to-optimal choice of sensors, and
the corresponding {optimal} control policy.  Specifically, the suboptimality gaps of the algorithms depend on 
the supermodularity ratio $\gamma$ of the \LQG cost, and we establish connections between 
$\gamma$ and control-theoretic quantities, providing computable lower bounds for~$\gamma$.

\myParagraph{Contributions to set function optimization}
To prove the aforementioned results, we extend the literature on supermodular optimization.
Particularly, we provide the first efficient algorithm for approximately supermodular optimization (e.g., \LQG cost optimization) subject to {\textit{cost} constraints for subset selection (e.g., sensor selection).
To this end, we use the algorithm in~\cite{krause2005note}, proposed for \textit{exactly} supermodular optimization, and prove it maintains provable suboptimality bounds for even \textit{approximately} supermodular optimization.  Importantly, our bounds improve the previously known bounds for \textit{exactly} supermodular optimization: our bounds become $1-1/e$ for supermodular optimization, tightening the known $1/2(1-1/e)$~\cite{krause2005note}.   Noticeably, $1-1/e$ is the best possible bound in polynomial time for supermodular optimization subject to cardinality constraints ~\cite{feige1998}.  That way, our analysis equates the approximation difficulty of \textit{cost} and \textit{cardinality} constrained optimization for the first time (among all algorithms with at most quadratic running time).\footnote{{Other algorithms, that either achieve the $1-1/e$ bound but are slower ($O(n^5)$ instead of $O(n^2)$ that ours is), or they achieve looser bounds with the same running time, such as the $1-1/\sqrt{e}$, are found in~\cite{sviridenko2004note,krause2005note,nguyen2013budgeted,iyer2013submodular,zhang2016submodular,qian2017subset}.}} That way, our results are relevant beyond sensing in control, such as in cost effective outbreak detection in networks~\cite{leskovec2007cost}.}

{Similarly, we provide the first algorithm for minimal cost subset selection subject to desired  bounds on an approximately supermodular function.  The algorithm relies on a simplification of the algorithm in~\cite{krause2005note}.  
Leveraging our novel bounds, we show the algorithm is the first with provable suboptimality bounds given approximately supermodular functions.  Notably, for exactly supermodular functions the bound recovers the well-known bound for cardinality minimization~\cite{wolsey1982analysis};  that way, similarly to above, our analysis equates the approximation difficulty of cost and cardinality minimization for the first time.}

\myParagraph{Application examples}
We demonstrate the effectiveness of the proposed algorithms in numerical experiments, by considering  
two application scenarios:
\emph{sensing-constrained formation control}, and 
\emph{resource-constrained robot navigation}.    We present a Monte Carlo analysis for both, which demonstrates that 
(i) the proposed sensor selection strategy is near-optimal, and, particularly, the resulting \LQG cost 
matches the optimal selection in all tested instances 
for which the optimal selection could be computed via a brute-force approach;
(ii) a more naive selection which attempts to minimize the state estimation error~\cite{jawaid2015submodularity}
(rather than the \LQG cost) has degraded \LQG performance, often comparable to a random selection;
and (iii) the selection of a small subset of sensors using the proposed algorithms ensures an \LQG cost that is 
close to the one obtained by using all available sensors, hence providing an effective alternative for control under 
sensing constraints.

\myparagraph{Comparison with the preliminary results in~\cite{tzoumas2018sensing} {(which coincides with the preprint~\cite{Tzoumas18acc-scLGQG})}} 
This paper {(which coincides with the preprint \cite{tzoumas2018lqg})} extends the preliminary results in~\cite{tzoumas2018sensing},
and provides comprehensive 
presentation of the \LQG co-design problem, by including both  the
\emph{sensing-constrained \LQG control} (introduced in~\cite{tzoumas2018sensing}) and the \emph{minimum-sensing \LQG control} problem (not previously published). Moreover, we generalize the setup in~\cite{tzoumas2018sensing} to account for any sensor costs 
(in~\cite{tzoumas2018sensing} each sensor has unit cost, whereas herein sensors have different costs).  Also, we
extend the numerical analysis accordingly. {Moreover, we prove the inapproximability of the problem.}
Most of the technical results  (Theorems~\ref{th:LQG_closed}-\ref{th:performance_LQGconstr}, and Algorithms~\ref{alg:greedy_nonUniformCosts}-\ref{algLQGconstr:greedy_nonUniformCosts})
are novel, and have not been~published.

\myparagraph{Organization of the rest of the paper} 
Section~\ref{sec:problemStatement} formulates the  \LQG control and sensing co-design problems.
Section~\ref{sec:codesign} presents a separation principle, the inapproximability theorem, and introduces the algorithms for the co-design problems. 
Section~\ref{sec:guarantees} characterizes the performance of the algorithms, and 
establishes connections between their suboptimality bounds and control-theoretic quantities.
Section~\ref{sec:exp} presents two examples of the co-design problems.
Section~\ref{sec:con} concludes the paper.  
\textbf{{All proofs are given in the appendix.}}
\medskip

\myParagraph{Notation} 
Lowercase letters denote vectors and scalars (e.g., $v$); uppercase letters denote matrices (e.g., $M$). Calligraphic fonts denote sets (e.g., $\calS$).
$\eye$ denotes the identity matrix.


\section{Problem Formulation:\\ \LQG Control and Sensing Co-design} 
\label{sec:problemStatement}

Here we formalize the \LQG control and sensing co-design problem considered in this paper.  Specifically, we present two ``dual'' statements of the problem: the \emph{sensing-constrained \LQG control},
and the \emph{minimum-sensing \LQG control}.

\subsection{System, sensors, and control policies} 
\label{sec:definitions}

We start by introducing the paper's framework:
\setcounter{paragraph}{0}
\paragraph{System} We consider a discrete-time time-varying
linear system with additive Gaussian noise,
\begin{equation}
x\at{t+1}=A_tx_t+B_tu_t + w_t, \quad t = 1, 2, \ldots, T,\label{eq:system}
\end{equation}
where $x_t \in \mathbb{R}^{n}$ is  the system's state at time~$t$, $u_t \in \mathbb{R}^{m_t}$ is the control action, $w_t$ is the process noise, $A_t$ and $B_t$ are known matrices,
and 
$T$ is a finite horizon. Also, $x_1$ is a Gaussian random variable with covariance~$\Sigma\att{1}{1}$, and $w_t$ is a Gaussian random variable with mean zero and covariance $W_t$, such that $w_t$ is independent of $x_1$ and $w_{t'}$ for all $t'=1,2,\ldots,T$, $t'\neq t$.

\paragraph{Sensors} We consider the availability of a (potentially large) set $\calV$ of sensors, which can take noisy linear observations of the system's state.  
Particularly,
\begin{equation}
y_{i,t}=C_{i,t}x_t+v_{i,t}, \quad i \in \calV,\label{eq:sensors}
\end{equation}
where $y_{i,t}\in \mathbb{R}^{p_{i,t}}$ is the measurement 
of sensor $i$ at time~$t$, \validated{$C_t$}{$C_{i,t}$} is a sensing matrix,
and $v_{i,t}$ is the measurement noise.  We assume $v_{i,t}$ to be a Gaussian random variable with mean zero and positive definite covariance $V_{i,t}$, such that~$v_{i,t}$ is independent of $x_1$, and of $w_{t'}$ for any $t'\neq t$, and independent of $v_{i',t'}$ for all $t'\neq t$, and any $i'\in\calV$, $i' \neq i$.

{When only a subset 
$\calS \subseteq \calV$ of the sensors is active for all $t = 1,2,\ldots, T$,
then the measurement model becomes}
\begin{equation}
y_{t}(\calS) = C_t(\calS) x_t + v_{t}(\calS),\label{eq:activeSensors}
\end{equation}
where $
y_{t}(\calS) \triangleq [y_{i_1,t}\tran, y_{i_2,t}\tran, \ldots, y_{i_{|\calS|},t}\tran]\tran\!$,
$C_t(\calS) \triangleq [C\at{i_1,t}\tran, \ldots,$ $ C\at{i_{|\calS|},t}\tran]\tran\!$,
and $v_t(\calS)$ is a mean zero Gaussian noise with cova- riance $V_t(\calS) \triangleq \diag{V\at{i_1,t}, \ldots,V\at{i_{|\calS|},t}}$.

{Each sensor is associated with a (possibly different) cost, which captures, for example,  the sensor's monetary cost, its power consumption, or its bandwidth utilization.}  Specifically, 
we denote the \emph{cost of sensor} $i$ by $\sensorCost(i) \validated{}{\;\geq 0}$; and the \emph{cost of a sensor set} $\calS$ by $\sensorCost(\calS)$,
where we set
	\begin{equation}\label{notation:cost_active_sensor_set}
	\sensorCost(\calS)\;\validated{=}{\triangleq}\sum_{i \in \calS} \sensorCost(i).
	\end{equation}

\paragraph{Control policies}  We consider control policies~$u_t$ informed only by the active sensor set $\calS$:
\begin{equation*}
u_t = u\of{t}{\calS}= u_t\big(y\of{1}{\calS}, y\of{2}{\calS}, \ldots, y\of{t}{\calS}\big), \quad t = 1,2,\ldots,T.
\end{equation*}

\subsection{\LQG co-design problems} 
\label{sec:definitionsLQG}

{We define two versions of the co-design problem:
\emph{sensing-constrained \LQG control} 
and \emph{minimum-sensing \LQG control}. Their unifying goal is to find active sensors $\calS$ and a policy $u\of{1:T}{\calS} \triangleq \{u\of{1}{\calS},$ $u\of{2}{\calS}, \ldots, u\of{T}{\calS}\}$, such that the sensor cost $c(\calS)$ is low and the finite horizon LQG cost $h(\calS,u\of{1:T}{\calS})$ below is optimized:
\begin{equation}\label{eq:finiteLQGcost}
 h(\calS,u\of{1:T}{\calS})\triangleq \sum_{t=1}^{T}\mathbb{E}\left[\|x\of{t+1}{\calS}\|^2_{Q_t} +\|u\of{t}{\calS}\|^2_{R_t}\right],
\end{equation}
where $Q_1,  \ldots, Q_T$ are known positive semi-definite matricies, $R_1,  \ldots, R_T$ are known positive definite matricies, 
and the expectation is taken with respect to $x_1$, $w_{1:T}$, and $v\of{1:T}{\calS}$.  
Particularly, the \emph{sensing-constrained \LQG control} minimizes the \LQG cost subject to a sensor cost budget, and the dual \emph{minimum-sensing \LQG control} minimizes the sensor cost subject to a desired \LQG cost.

\begin{myproblem}[Sensing-constrained \LQG control]\label{prob:LQG}
Given a budget {$b \geq 0$} on the sensor cost, find sensors $\calS $ and a policy $u\of{1:T}{\calS} \triangleq \{u\of{1}{\calS},$ $u\of{2}{\calS}, \ldots, u\of{T}{\calS}\}$ to minimize the finite horizon \LQG cost $ h(\calS,u\of{1:T}{\calS})$:
\begin{equation}
\label{eq:sensingConstrainedLQG}
\!\!
\min_{\scriptsize\begin{array}{c}
\calS \subseteq \calV,\\
 u\of{1:T}{\calS}
\end{array}} 
  h(\calS,u\of{1:T}{\calS})\validated{:}{, 
 \;\; \subjectTo}\;\;\sensorCost(\calS)\leq \sensorBudget.
\end{equation}
\end{myproblem}

Problem~\ref{prob:LQG} models the practical case where we cannot 
activate all sensors (due to power, cost, or bandwidth constraints), and instead need to activate a few sensors to optimize control performance.
If the budget is increased so \validated{}{all 
sensors can be active, then Problem~\ref{prob:LQG} reduces to standard \LQG control.}


\begin{myproblem}[Minimum-sensing \LQG control]
\label{prob:minCostLQG}
Find a minimal cost sensor set $\calS$, and a policy $u\of{1:T}{\calS}$, such that the finite horizon \LQG cost $ h(\calS,u\of{1:T}{\calS})$ is at most $\kappa$, where $\kappa \; \geq 0$ is given:
\begin{equation}
\label{eq:minCostForBoundedLQG}
\!\!
\min_{\scriptsize\begin{array}{c}
\calS \subseteq \calV,\\
 u\of{1:T}{\calS}
\end{array}} \sensorCost(\calS)\validated{:}{, 
 \;\; \subjectTo}\;\; h(\calS,u\of{1:T}{\calS}) \leq \kappa.
\end{equation}
\end{myproblem}

Problem~\ref{prob:minCostLQG}  models the practical case where one wants to design a system with a prescribed performance, while incurring in the smallest \validated{}{sensor} cost.
}

\begin{myremark}
[Case of uniform-cost sensors] \label{rem:uniform_cost} When 
all sensors $i\in \calV$ have the same cost 
\validated{, say
unitary cost $\sensorCost(i)=1$, 
and the sensor budget $\sensorBudget$ is  a natural number such that $\sensorBudget \leq |\calV|$, then:}{, say
$\sensorCost(i)=\bar{c} > 0$, the sensor budget constraint becomes a cardinality constraint:
\bea
\sensorCost(\calS) \leq \sensorBudget  \;\; \Leftrightarrow  \;\;
\sum_{i \in \calS} \sensorCost(i) \leq \sensorBudget \;\; \Leftrightarrow  \;\;
|\calS| \bar{c} \leq \sensorBudget   \;\; \Leftrightarrow  \;\;
|\calS| \leq  \frac{\sensorBudget}{\bar{c}}. 
\eea
}
\end{myremark}


\section{Co-design Principles, {Hardness}, and Algorithms} 
\label{sec:codesign}

{We leverage a separation principle to derive that the optimization of the sensor set $\calS$ and of the control policy $u_{1:T}(\calS)$ can happen in cascade.
However, we show that optimizing for $\calS$ is inapproximable  in polynomial time.  Nonetheless, we then} present polynomial time algorithms for Problem~\ref{prob:LQG} and Problem~\ref{prob:minCostLQG} with provable
per-instance suboptimality bounds.  Particularly, the bounds are presented in Section~\ref{sec:guarantees}.\footnote{{The novelty of the algorithms is also discussed in Section~\ref{sec:guarantees}.}}

\subsection{Separability of optimal sensing and control design}\label{sec:separability}

We characterize the jointly optimal control and sensing solutions \validated{in}{for} Problem~\ref{prob:LQG} and Problem~\ref{prob:minCostLQG}, and prove 
they can be found in two separate steps, where first the sensor set is found, and then the control policy is computed.

\begin{mytheorem}[Separability of optimal sensor set and control policy design]\label{th:LQG_closed}
For any active sensor set $\calS$, let $\hat{x}\of{t}{\calS}$ be the Kalman estimator of the state $x_t$,
and $\Sigma\att{t}{t}(\calS)$ be $\hat{x}\of{t}{\calS}$'s error covariance.
Additionally, let the matrices $\Theta_t$ and $K_t$  be the solution of the \validated{}{following} backward Riccati recursion
\beal\label{eq:control_riccati}
S_t &= Q_t + N_{t+1}, \\
N_t &= A_t\tran (S_t\inv+B_tR_t\inv B_t\tran)\inv A_t, \\ 
M_t &= B_t\tran S_t B_t + R_t, \\ 
K_t &= - M_t\inv B_t\tran S_t A_t, \\ 
\Theta_t &= K_t\tran M_t K_t,
\eeal
with boundary condition $N_{T+1}=0$.

\begin{enumerate}
\item \emph{(Separability in Problem~\ref{prob:LQG})} 
Any optimal solution $(\calS^\star,$ $u_{1:T}^\star)$ to Problem~\ref{prob:LQG} can be computed in cascade:
\begin{align}
\calS^\star &\in \argmin_{\calS \subseteq \calV}\sum_{t=1}^T\text{tr}[\Theta_t\Sigma\att{t}{t}(\calS)]\validated{:}{, 
 \;\; \subjectTo}\;\;\sensorCost(\calS)\leq \sensorBudget,\label{eq:opt_sensors}\\
u_t^\star&=K_t\hat{x}\of{t}{\calS^\star}, \quad t=1,\ldots,T\label{eq:opt_control}.
\end{align}
 
\item \emph{(Separability in Problem~\ref{prob:minCostLQG})} 
Define the constant \validated{}{$\bar{\kappa} \triangleq \kappa-\trace{\initialCovariance N_1}-\sum_{t=1}^T\trace{W_t S_t}$}.
Any optimal solution $(\calS^\star,u_{1:T}^\star)$ to Problem~\ref{prob:minCostLQG}
can be computed in cascade: 
\begin{align}
\calS^\star &\in \argmin_{\scriptsize \calS \subseteq \calV} \sensorCost(\calS)
 \validated{:}{, 
 \;\; \subjectTo}\;\;\validated{}{\sum_{t=1}^T}\text{tr}[\Theta_t\Sigma\att{t}{t}(\validated{\algS}{\calS})] \leq  \bar{\kappa}, \label{eq:opt_sensors_minimum} \\
u_t^\star&=K_t\hat{x}\of{t}{\calS^\star}, \quad t=1,\ldots,T.\label{eq:opt_control_minimum}
\end{align}
\end{enumerate}
\end{mytheorem}

\begin{myremark}[Certainty equivalence principle]
The control gain matrices 
$K_1, K_2, \ldots, K_T$
are the same as the ones that make the controllers $(K_1x_1$, $K_1x_2,\ldots, K_T x_T)$ optimal for the perfect state-information version of Problem~\ref{prob:LQG}, where the state $x_t$ is known to the controllers~\cite[Chapter~4]{bertsekas2005dynamic}.
\end{myremark}

Theorem~\ref{th:LQG_closed} decouples the sensing design from the control policy design.  
Particularly, once an optimal sensor set~$\calS^\star$ is found, then the optimal controllers are equal to $K_t\hat{x}\of{t}{\calS^\star}$, which correspond to the standard \LQG control policy. 

An intuitive interpretation of the sensor design steps in eqs.~\eqref{eq:opt_sensors} and~\eqref{eq:opt_sensors_minimum} follows next.

\begin{myremark}[Control-aware sensor design]\label{rmk:interpretation}
To provide insight on {$\sum_{t=1}^{T}\text{tr}[\Theta_t\Sigma\att{t}{t}(\calS)]$} in eqs.~\eqref{eq:opt_sensors}~and~\eqref{eq:opt_sensors_minimum}, we rewrite it~as
\begin{align}
\!\!\!
\displaystyle\sum_{t=1}^{T}\text{tr}[\Theta_t\Sigma\att{t}{t}(\calS)] &\!=\!
\displaystyle\sum_{t=1}^{T}\mathbb{E}\left(\text{tr}\{[x_t-\hat{x}\of{t}{\calS}]\tran\Theta_t[x_t-\hat{x}\of{t}{\calS}]\}\right) \nonumber \\
&\!=\!
\displaystyle\sum_{t=1}^{T}\mathbb{E}\left( \| K_t x_t-K_t \hat{x}\of{t}{\calS}\|^2_{M_t} \right), \label{eq:interpretationRemark}
\end{align}
since
$\Sigma\att{t}{t}(\calS) = \mathbb{E}\left[(x_t-\hat{x}\of{t}{\calS})(x_t-\hat{x}\of{t}{\calS})\tran\right]$, and $\Theta_t = K_t\tran M_t K_t$.
From eq.~\eqref{eq:interpretationRemark}, each  
$\text{tr}[\Theta_t\Sigma\att{t}{t}(\calS)]$ captures the  mismatch between the 
imperfect state-information controller $u\of{t}{\calS}=K_t\hat{x}\of{t}{\calS}$ (which is only aware of the 
measurements from the active sensors) and the  perfect state-information controller~$K_tx_t$.
That is, while standard sensor selection minimizes the estimation covariance, for instance by 
minimizing
\begin{equation}
\sum_{t=1}^{T}\text{tr}[\Sigma\att{t}{t}(\calS)] \triangleq
\displaystyle\sum_{t=1}^{T}\mathbb{E}\left(\| x_t-\hat{x}\of{t}{\calS}\|^2_2 \right),\label{eq:interpretationRemark_kalman}
\end{equation}
the proposed \LQG cost formulation \validated{attempts to minimize the estimation error of only the informative states to the perfect state-information controller:}{selectively minimizes the estimation error focusing on the  states that are most informative for control purposes. For example,} the mismatch contribution in eq.~\eqref{eq:interpretationRemark} of any $x_t-\hat{x}\of{t}{\calS}$ in the null space of $K_t$ is zero; accordingly, the proposed sensor design approach has 
no incentive in activating sensors to observe states which are irrelevant for control purposes. 
\end{myremark}

\subsection{{Inapproximability of optimal sensing design}}\label{sec:hardness}

{\begin{mytheorem}[Inapproximability]\label{th:hardness}
If \emph{NP}$\neq$\emph{P}, then there is no polynomial time algorithm for Problems~\ref{prob:LQG} and~\ref{prob:minCostLQG} that returns an approximate solution within a constant factor from the optimal.   This remains true, even if all sensors have cost~$1$.
\end{mytheorem}

We prove the theorem by reducing the inapproximable problem in~\cite{ye2018complexity} ---sensor selection with cost constraints for optimal steady state Kalman filtering error---  to eq.~\eqref{eq:opt_sensors}.

Motivated by the inapproximability of Problem~\ref{prob:LQG} and Problem~\ref{prob:minCostLQG}, we next present practical algorithms, which in Section~\ref{sec:guarantees} we prove to enjoy per-instance suboptimality bounds.}

\subsection{Co-design \validated{}{algorithms} for Problem~\ref{prob:LQG}}
\label{sec:designAlgorithm}


\begin{algorithm}[t]
\caption{\mbox{Joint sensing and control design \hspace{-.25mm}for \hspace{-.25mm}Problem~\ref{prob:LQG}.}}
\begin{algorithmic}[1]
\REQUIRE  Horizon $T$; system in eq.~\eqref{eq:system}; covariance $\initialCovariance$; LQG cost matrices $Q_t$ and $R_t$ in eq.~\eqref{eq:finiteLQGcost}; sensors in eq.~\eqref{eq:sensors}; sensor budget $\sensorBudget$; sensor cost~$\sensorCost(i)$, for all  $i \in \calV$. 
\ENSURE Active sensors $\widehat{\mathcal{S}}$, and controls $\hat{u}_1, \hat{u}_2,\ldots,\hat{u}_T$.
\STATE{Compute $\Theta_1,\Theta_2,\ldots,\Theta_T$ using eq.~\eqref{eq:control_riccati}.}
\STATE{Return $\widehat{\calS}$ returned by Algorithm~\ref{alg:greedy_nonUniformCosts}, which finds a solution to the optimization problem in eq.~\eqref{eq:opt_sensors};}
\STATE{Compute  $K_1,K_2,\ldots,K_T$ using eq.~\eqref{eq:control_riccati}.}
\STATE{At each $t=1\ldots,T$, compute the Kalman estimate of~$x_t$:
$$\hat{x}_t\triangleq\mathbb{E}[x_t|y\of{1}{\algS},y\of{2}{\algS},\ldots,y\of{t}{\algS}];$$
}
\STATE{At each $t=1,\ldots,T$, return $\hat{u}_t=K_t\hat{x}_t$.
}
\end{algorithmic} \label{alg:overall_nonUniformCosts}
\end{algorithm}

{We present a practical algorithm for the sensing-constrained \LQG control Problem~\ref{prob:LQG} (Algorithm~\ref{alg:overall_nonUniformCosts}). The algorithm follows Theorem~\ref{th:LQG_closed}: it first computes a sensing design, and then a control design, as described below.}

\myParagraph{\validated{}{Sensing design for Problem~\ref{prob:LQG}}}
\validated{}{Theorem~\ref{th:LQG_closed} implies an optimal sensor design for Problem~\ref{prob:LQG} can be computed by solving eq.~\eqref{eq:opt_sensors}.  To this end, Algorithm~\ref{alg:overall_nonUniformCosts}  first computes $\Theta_1, \Theta_2, \ldots, \Theta_T$ (Algorithm~\ref{alg:overall_nonUniformCosts}'s line~1).
Next, since eq.~\eqref{eq:opt_sensors} is inapproximable (Theorem~\ref{th:hardness}), 
Algorithm~\ref{alg:overall_nonUniformCosts} calls a greedy algorithm (Algorithm~\ref{alg:greedy_nonUniformCosts})
to compute a solution to eq.~\eqref{eq:opt_sensors} (Algorithm~\ref{alg:overall_nonUniformCosts}'s line~2).} 

Algorithm~\ref{alg:greedy_nonUniformCosts} computes a solution to eq.~\eqref{eq:opt_sensors} as follows: 
first, Algorithm~\ref{alg:greedy_nonUniformCosts} creates two candidate active sensor sets $\algSone$ and $\algStwo$ (lines~\ref{line:initializeSone_nonUniformCosts}-\ref{line:initializeStwo_nonUniformCosts}), of which only one will be selected as the solution to eq.~\eqref{eq:opt_sensors} (line~\ref{line:pick_best_set}). 
{In more detail, Algorithm~\ref{alg:greedy_nonUniformCosts}'s line~\ref{line:initializeSone_nonUniformCosts} lets~$\algSone$ be composed of a single sensor, namely the sensor~$i \in \calV$ that achieves the smallest value of the objective function in eq.~\eqref{eq:opt_sensors} 
and has smaller cost than the budget $b$ ($\sensorCost(i)\leq \sensorBudget$).
}
Then, Algorithm~\ref{alg:greedy_nonUniformCosts}'s line~\ref{line:initializeStwo_nonUniformCosts} initializes~$\algStwo$ with the empty set, and after the construction of~$\algStwo$ in Algorithm~\ref{alg:greedy_nonUniformCosts}'s lines~\ref{line:while_nonUniformCosts}--\ref{line:end_if_condition_less}, Algorithm~\ref{alg:greedy_nonUniformCosts}'s line~\ref{line:pick_best_set} computes which of $\algSone$ and $\algStwo$ achieves the smallest value for the objective function in eq.~\eqref{eq:opt_sensors}, and returns this set as the
solution to \mbox{eq.~\eqref{eq:opt_sensors}}.


\begin{algorithm}[t]
\caption{Sensing design for Problem~\ref{prob:LQG}.}
\begin{algorithmic}[1]
\REQUIRE
 Horizon $T$; system in eq.~\eqref{eq:system}; covariance $\initialCovariance$; LQG cost matrices $Q_t$ and $R_t$ in eq.~\eqref{eq:finiteLQGcost}; sensors in eq.~\eqref{eq:sensors}; sensor budget $\sensorBudget$; sensor cost~$\sensorCost(i)$, for all  $i \in \calV$. 
\ENSURE Sensor set $\algS$.
\STATE{$\algSone \leftarrow \arg\min_{i \in \calV, \sensorCost(i)\leq \sensorBudget}\sum_{t=1}^T\text{tr}[\Theta_t\Sigma\att{t}{t}(\{i\})]$;}\label{line:initializeSone_nonUniformCosts}
\STATE{$\algStwo\leftarrow\emptyset$; \quad $\calV' \leftarrow \calV$;} \label{line:initializeStwo_nonUniformCosts}
\WHILE {$\calV'\neq \emptyset$ and $\sensorCost(\algStwo)\leq \sensorBudget$}   \label{line:while_nonUniformCosts}
	\FORALL {$a\in \calV'$} \label{line:startFor1_nonUniformCosts}
	\STATE{$\algStwoAlpha\leftarrow \algStwo\cup \{a\}$; \quad $\initialCovariance(\algStwoAlpha)\leftarrow \initialCovariance$;}\label{line:initialize_covariance}
	\FORALL {$t = 1,\ldots, T$}
	\vspace{0.3mm}
	\STATE{\hspace*{-3.25mm}$\Sigma\att{t+1}{t}(\algStwoAlpha)\leftarrow A_{t} \Sigma\att{t}{t}(\algStwoAlpha) A_{t}\tran + W_{t}$;}
	\STATE{\hspace*{-3.25mm}$\Sigma\att{t}{t}(\algStwoAlpha)\leftarrow$}
	\STATE{\hspace*{-2.5mm}$[ 
	\Sigma\att{t}{t-1}(\algStwoAlpha)\inv + C_t(\algStwoAlpha)\tran V_t(\algStwoAlpha)\inv C_t(\algStwoAlpha)]\inv$;} 
	\ENDFOR\label{line:endtFor1_nonUniformCosts}
	\STATE{$\text{gain}_a \leftarrow \sum_{t=1}^T\text{tr}\{\Theta_t[\Sigma\att{t}{t}(\algStwo)-\Sigma\att{t}{t}(\algStwoAlpha)]\}$;} \label{line:cost_nonUniformCosts}
	\ENDFOR \label{line:endFor1_nonUniformCosts}
\STATE{$s \leftarrow \arg\max_{a \in \calV'}[ {\text{gain}_a}/{c(a)}]$;} \label{line:best_a_nonUniformCosts}
\STATE{$\algStwo\leftarrow \algStwo \cup \{s\}$;} \label{line:add_a_nonUniformCosts}
\STATE{$\calV' \leftarrow \calV'\setminus \{s\}$;}\label{line:remove_s}
\ENDWHILE \label{line:end_while_nonUniformCosts}
\IF {$\sensorCost(\algStwo)> \sensorBudget$}\label{line:if_condition_less}
\STATE{$\algStwo\leftarrow \algStwo\setminus\{s\}$;}
\ENDIF \label{line:end_if_condition_less}
\STATE{$\algS \leftarrow \arg\min_{\calS\in\{\algSone,\algStwo\}} \sum_{t=1}^T\text{tr}[\Theta_t\Sigma\att{t}{t}(\calS)]$.}\label{line:pick_best_set}
\end{algorithmic} \label{alg:greedy_nonUniformCosts}
\end{algorithm}

Specifically, Algorithm~\ref{alg:greedy_nonUniformCosts}'s lines~\ref{line:while_nonUniformCosts}--\ref{line:end_if_condition_less} construct $\algStwo$ as follows: at each iteration of the ``while loop'' (lines~\ref{line:while_nonUniformCosts}-\ref{line:end_while_nonUniformCosts})
a sensor is greedily added to~$\algStwo$, as long as $\algStwo$'s  cost does not exceed $\sensorBudget$.
Particularly, for each remaining sensor $a$ in $\calV\setminus \algStwo$,
 the ``for loop''  (lines~\ref{line:startFor1_nonUniformCosts}-\ref{line:endFor1_nonUniformCosts}) 
computes first
the estimation covariance resulting by adding $a$ in $\algStwo$, and then the marginal gain in the objective function in eq.~\eqref{eq:opt_sensors} (line~\ref{line:cost_nonUniformCosts}). 
Afterwards, the sensor inducing the largest marginal gain (normalized by the sensor's cost) is selected (line~\ref{line:best_a_nonUniformCosts}),
and is added in $\algStwo$ (line~\ref{line:add_a_nonUniformCosts}).  Finally, the ``if'' in lines~\ref{line:if_condition_less}-\ref{line:end_if_condition_less} ensure $\algStwo$ has cost at most $\sensorBudget$, by removing last sensor added in~$\algStwo$ if necessary.

\validated{}{\myParagraph{\validated{}{Control design for Problem~\ref{prob:LQG}}}
Theorem~\ref{th:LQG_closed} implies that given a sensor set, the controls for Problem~\ref{prob:LQG} can be computed according to the eq.~\eqref{eq:opt_control}.  To this end, Algorithm~\ref{alg:overall_nonUniformCosts} first computes $K_1,K_2,\ldots,K_T$ (line~3), and then, at each time $t=1,2,\ldots, T$, the Kalman estimate of the current state $x_t$ (line~4), and the corresponding control (line~5).}

\subsection{Co-design   \validated{}{algorithms}  for Problem~\ref{prob:minCostLQG}}
\label{sec:designAlgorithmLQGconst}

{This section presents a practical algorithm for Problem~\ref{prob:minCostLQG}  (Algorithm~\ref{algLQGconstr:overall_nonUniformCosts}).  Since the algorithm shares steps with Algorithm~\ref{alg:overall_nonUniformCosts}, we focus on the different ones.}

{Particularly, as Algorithm~\ref{alg:overall_nonUniformCosts} calls Algorithm~\ref{alg:greedy_nonUniformCosts} to solve eq.~\eqref{eq:opt_sensors}, similarly, Algorithm~\ref{algLQGconstr:overall_nonUniformCosts} calls Algorithm~\ref{algLQGconstr:greedy_nonUniformCosts} to solve eq.~\eqref{eq:opt_sensors_minimum}.   Algorithm~\ref{algLQGconstr:greedy_nonUniformCosts} is similar to Algorithm~\ref{alg:greedy_nonUniformCosts}, with the difference that Algorithm~\ref{algLQGconstr:greedy_nonUniformCosts} selects sensors until the upper bound $\bar{\kappa}$ in eq.~\eqref{eq:opt_sensors_minimum} is met (Algorithm~\ref{algLQGconstr:greedy_nonUniformCosts}'s line~3), whereas Algorithm~\ref{alg:greedy_nonUniformCosts} selects sensors up to the point the cost budget $b$ is violated (Algorithm~\ref{alg:greedy_nonUniformCosts}'s line~3).}


\section{Performance guarantees for \LQG co-design}
\label{sec:guarantees}

{We now quantify the suboptimality and running time of Algorithms~\ref{alg:overall_nonUniformCosts} and Algorithms~\ref{algLQGconstr:overall_nonUniformCosts}. Particularly, we prove both algorithms enjoy  {per-instance} suboptimality bounds,\footnote{Instead of constant suboptimality bounds across \textit{all} instances, which is impossible due to Theorem~\ref{th:hardness}.} and run in quadratic time. 
To this end, we present a notion of supermodularity ratio (Definition~\ref{def:super_ratio}), which we use to prove the suboptimality bounds. We~then establish connections between the ratio and 
control-theoretic quantities (Theorem~\ref{th:submod_ratio}), and conclude that the algorithms' suboptimality bounds are non-vanishing under control-theoretic conditions encountered \mbox{in most real-world systems (Theorem~\ref{th:freq}).}}

\subsection{Supermodularity ratio}\label{sec:submodularity}
 
To present the definition of \emph{supermodularity ratio}, we start \validated{with}{by defining}  \textit{monotonicity} and \textit{supermodularity}.

\begin{mydef}[Monotonicity~{\cite{nemhauser78analysis}}]\label{def:mon}
Consider any \validated{finite ground}{finite} set~$\mathcal{V}$.  The set function $f:2^\calV\mapsto \mathbb{R}$ is 
\emph{non-increasing} if and only if for any sets $\mathcal{A}\subseteq \validated{\mathcal{A}'}{\mathcal{B}}\subseteq\calV$, it \validated{is}{holds} $f(\mathcal{A})\geq f(\validated{\mathcal{A}'}{\mathcal{B}})$.
\end{mydef}


\begin{algorithm}[t]
\caption{\mbox{Joint Sensing and Control design \hspace{-.25mm}for \hspace{-.25mm}Problem~\ref{prob:minCostLQG}.}}
\begin{algorithmic}[1]
\REQUIRE  
Horizon $T$; system in eq.~\eqref{eq:system}; covariance $\initialCovariance$; LQG cost matrices $Q_t$ and $R_t$ in eq.~\eqref{eq:finiteLQGcost}; \LQG cost bound $\kappa$l sensors in eq.~\eqref{eq:sensors}; sensor cost~$\sensorCost(i)$, for all  $i \in \calV$. 
\ENSURE Active sensors $\widehat{\mathcal{S}}$, and controls $\hat{u}_1, \hat{u}_2,\ldots,\hat{u}_T$.
\STATE{Compute $\Theta_1,\Theta_2,\ldots,\Theta_T$ using eq.~\eqref{eq:control_riccati}.}
\STATE{Return $\widehat{\calS}$ returned by Algorithm~\ref{algLQGconstr:greedy_nonUniformCosts}, which finds a solution to the optimization problem in eq.~\eqref{eq:minCostForBoundedLQG} \label{line:sensingDesign_minimal};
 }
\STATE{Compute  $K_1,K_2,\ldots,K_T$ using eq.~\eqref{eq:control_riccati}.}
\STATE{At each $t=1\ldots,T$, compute the Kalman estimate of~$x_t$:
	$$\hat{x}_t\triangleq\mathbb{E}[x_t|y\of{1}{\algS},y\of{2}{\algS},\ldots,y\of{t}{\algS}];$$
}
\STATE{At each $t=1,\ldots,T$, return $\hat{u}_t=K_t\hat{x}_t$.
}
\end{algorithmic} \label{algLQGconstr:overall_nonUniformCosts}
\end{algorithm}

\begin{mydef}[Supermodularity~{\cite[Proposition 2.1]{nemhauser78analysis}}]\label{def:sub}
Consider any \validated{finite ground}{finite}  set $\calV$.  The set function $f:2^\calV\mapsto \mathbb{R}$ is \emph{supermodular} if and only if
for any sets $\mathcal{A}\subseteq \validated{\mathcal{A}'}{\mathcal{B}}\subseteq\calV$, and any element $\elem\in \calV$, it \validated{is}{holds}  
$f(\mathcal{A})-\!f(\mathcal{A}\cup \{\elem\})\geq f(\validated{\mathcal{A}'}{\mathcal{B}})-\!f(\validated{\mathcal{A}'}{\mathcal{B}}\cup \{\elem\})$.
\end{mydef}
That is, $f$ is supermodular if and only if it satisfies a diminishing returns property: for any $\elem\in \mathcal{V}$, the  drop $f(\mathcal{A})-f(\mathcal{A}\cup \{\elem\})$ diminishes as the set $\mathcal{A}$ grows.

\begin{mydef}[Supermodularity ratio~\cite{lehmann2006combinatorial}]\label{def:super_ratio}
Consider any \validated{finite ground}{finite}  set~$\mathcal{V}$, and a non-increasing set \mbox{function $f:2^\calV\mapsto \mathbb{R}$}.  We define the~\emph{supermodularity ratio} of $f$ as
\begin{equation*}
\gamma_f \triangleq \min_{\mathcal{A} \subseteq \calB \subseteq\mathcal{V}, \elem \in \mathcal{V}\setminus\calB} \frac{f(\mathcal{A})-f(\mathcal{A}\cup \{\elem\})}{f(\mathcal{B})-f(\mathcal{B}\cup \{\elem\})}.
\end{equation*}
\end{mydef}
The supermodularity ratio $\gamma_f$ measures how far $f$ is from being supermodular.  Particularly, $\gamma_f$ takes values in $[0,1]$, and
if $\gamma_f= 1$, then
$ f(\mathcal{A})-f(\mathcal{A}\cup \{\elem\})\geq f(\mathcal{B})-f(\mathcal{B}\cup \{\elem\})$, i.e., $f$ is supermodular.  Whereas, 
if $\validated{}{0<}\gamma_f < 1$, then
$ f(\mathcal{A})- f(\mathcal{A}\cup \{\elem\})\geq \textstyle\gamma_f \left[f(\mathcal{B})-f(\mathcal{B}\cup \{\elem\})\right]$,
i.e.,  
$\gamma_f$ captures how much ones needs to discount $f(\mathcal{B})-f(\mathcal{B}\cup \{\elem\})$, 
such that $f(\mathcal{A})-f(\mathcal{A}\cup \{\elem\})$ is at least  $f(\mathcal{B})-f(\mathcal{B}\cup \{\elem\})$.  {In this case, $f$ is called \textit{approximately (or weakly) supermodular}~\cite{krause2010submodular}.}  


\begin{algorithm}[t]
\caption{Sensing design for Problem~\ref{prob:minCostLQG}.}
\begin{algorithmic}[1]
\REQUIRE Horizon $T$; system in eq.~\eqref{eq:system}; covariance $\initialCovariance$; LQG cost matrices $Q_t$ and $R_t$ in eq.~\eqref{eq:finiteLQGcost}; \LQG cost bound $\kappa$l sensors in eq.~\eqref{eq:sensors}; sensor cost~$\sensorCost(i)$, for all  $i \in \calV$. 
\ENSURE Active sensors $\algS$.
\STATE \validated{}{$\bar{\kappa} \leftarrow \kappa-\trace{\initialCovariance N_1}-\sum_{t=1}^T\trace{W_tS_t}$}\label{lineLQGconstr:kappabar}
\STATE{$\algS\leftarrow\emptyset$; \quad $\calV' \leftarrow \calV$;} \label{lineLQGconstr:initializeStwo_nonUniformCosts}\
\WHILE {$\calV'\neq \emptyset$ and $\validated{}{\sum_{t=1}^T}\text{tr}[\Theta_t\Sigma\att{t}{t}(\algS)] > \validated{}{\bar{\kappa}}$}   \label{lineLQGconstr:while_nonUniformCosts}
	\FORALL {$a\in \calV'$} \label{lineLQGconstr:startFor1_nonUniformCosts}
	\STATE{$\algSAlpha\leftarrow \algS\cup \{a\}$; \quad $\initialCovariance(\algSAlpha)\leftarrow \initialCovariance$;}\label{lineLQGconstr:initialize_covariance}
	\FORALL {$t = 1,\ldots, T$}
	\vspace{0.3mm}
		\STATE{\hspace*{0mm}$\Sigma\att{t+1}{t}(\algSAlpha)\leftarrow A_{t} \Sigma\att{t}{t}(\algSAlpha) A_{t}\tran + W_{t}$;}
	\STATE{\hspace*{0mm}$\Sigma\att{t}{t}(\algSAlpha)\leftarrow$}
	\STATE{\hspace*{0mm}\;\;\;$[ 
	\Sigma\att{t}{t-1}(\algSAlpha)\inv + C_t(\algSAlpha)\tran V_t(\algSAlpha)\inv C_t(\algSAlpha)]\inv$;} 
	\ENDFOR\label{lineLQGconstr:endtFor1_nonUniformCosts}
	\STATE{$\text{gain}_a \leftarrow \sum_{t=1}^T\text{tr}\{\Theta_t[\Sigma\att{t}{t}(\algS)-\Sigma\att{t}{t}(\algSAlpha)]\}$;} \label{lineLQGconstr:cost_nonUniformCosts}
	\ENDFOR \label{lineLQGconstr:endFor1_nonUniformCosts}
\STATE{$s \leftarrow \arg\max_{a \in \calV'}[ {\text{gain}_a}/{c(a)}]$;} \label{lineLQGconstr:best_a_nonUniformCosts}
\STATE{$\algS\leftarrow \algS \cup \{s\}$;} \label{lineLQGconstr:add_a_nonUniformCosts}
\STATE{$\calV' \leftarrow \calV'\setminus \{s\}$;}\label{lineLQGconstr:remove_s}
\ENDWHILE \label{lineLQGconstr:end_while_nonUniformCosts}
\end{algorithmic} \label{algLQGconstr:greedy_nonUniformCosts}
\end{algorithm}

\subsection{Performance analysis for Algorithm~\ref{alg:overall_nonUniformCosts}}\label{sec:performanceGuarantees}

We quantify Algorithm~\ref{alg:overall_nonUniformCosts}'s running time and suboptimality, using the notion of supermodularity ratio.  We use the notation:
\begin{itemize}
	\item $g(\calS)$ is the optimal value of $h[\calS, u\of{1:T}{\calS}]$ across all $u\of{1:T}{\calS}$,  given any $\calS$:
	\begin{equation}
	\textstyle\hspace{-0.6em}g(\calS)\triangleq\min_{\scriptsize 
		u\of{1:T}{\calS}}h[\calS, u\of{1:T}{\calS}], \label{eq:gS}
	\end{equation}
	\item $h^\star\triangleq\min_{\scriptsize 
		\calS \subseteq \calV, 
		u\of{1:T}{\calS}} h[\calS, u\of{1:T}{\calS}]\validated{:}{, 
		\;\; \subjectTo}\;\;\sensorCost(\calS)\leq \sensorBudget$, i.e., the \validated{}{optimal value} of Problem~\ref{prob:LQG};
	\item  $
	\sensorBudget^\star\triangleq\min_{\scriptsize 
		\calS \subseteq \calV, 
		u\of{1:T}{\calS}} \sensorCost(\calS)\validated{:}{, 
		\;\; \subjectTo}\;\; h[\calS, u\of{1:T}{\calS}] \leq \kappa$, i.e., the \validated{}{optimal} value of Problem~\ref{prob:minCostLQG}.
\end{itemize}

\begin{mytheorem}[Performance of Algorithm~\ref{alg:overall_nonUniformCosts}]\label{th:approx_bound} 
Algorithm~\ref{alg:overall_nonUniformCosts} returns a sensor set $\algS$ and control policies $u\of{1:T}{\algS}$ such that
\begin{align}
\begin{split}
&\frac{h[\emptyset, u\of{1:T}{\emptyset}]-h[\algS,u_{1:T}(\algS)]}{h[\emptyset, u\of{1:T}{\emptyset}]-h^\star}\geq\\
&\qquad\qquad \max\left[\frac{\gamma_g}{2}\left(1-e^{-\gamma_g}\right),1-e^{-\gamma_g\sensorCost(\algS)/\sensorBudget} \right],
\end{split} \label{ineq:approx_bound}
\end{align}
where $\gamma_g$ is the supermodularity ratio of
$g(\calS)$ in eq.~\eqref{eq:gS}.

Moreover, Algorithm~\ref{alg:overall_nonUniformCosts} runs in $O(|\calV|^2Tn^{2.4})$ time.
\end{mytheorem}

In ineq.~\eqref{ineq:approx_bound}, $h[\emptyset, u\of{1:T}{\emptyset}]-h[\algS,u_{1:T}(\algS)]$ 
quantifies the gain from selecting $\algS$, and ineq.~\eqref{ineq:approx_bound}'s right-hand-side
guarantees the gain is close to the optimal
$h[\emptyset, u\of{1:T}{\emptyset}]-h^\star$.\footnote{{Even if no sensors are active, observe $h[\emptyset, u\of{1:T}{\emptyset}]$ is well defined and finite, since it is the \LQG cost over a \textit{finite} horizon $T$.}} 
Specifically, when either of the bounds in ineq.~\eqref{ineq:approx_bound}'s right-hand-side is $1$, then the algorithm returns an optimal solution.  

For comparison,  in Fig.~\ref{fig:bounds} we plot the bounds for  $\sensorCost(\algS)/\sensorBudget\in \{2/5, 1, 2\}$ and all $\gamma_g \in [0,1]$.
We observe that $1-e^{-\gamma_g\sensorCost(\algS)/\sensorBudget}$ dominates $\gamma_g/2\left(1-e^{-\gamma_g}\right)$ for $\sensorCost(\algS)/\sensorBudget>2/5$.
Moreover, as $\sensorCost(\algS)/\sensorBudget$ and $\gamma_g$ increase, then $1-e^{-\gamma_g\sensorCost(\algS)/\sensorBudget}$ tends to $1$, in which case, Algorithm~\ref{alg:overall_nonUniformCosts} returns an optimal solution.  

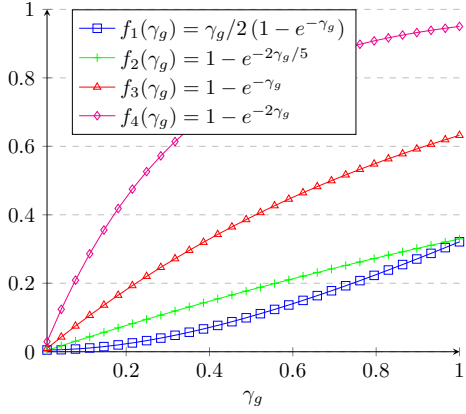
\begin{figure}[t]
\begin{center}
\begin{tikzpicture}[scale=0.8]
\begin{axis}[
    axis lines = left,
    xlabel = $\gamma_g$,
    ymajorgrids=true,
    grid style=dashed,
    legend style={at={(0.75,1)}},
    ymin=0, ymax=1,
]
\addplot [
    domain=0.01:1, 
    samples=30, 
    color=blue,
    mark=square,
    ]
    {x/2*(1-exp(-x))+0.005};
\addlegendentry{$f_1(\gamma_g)=\gamma_g/2\left(1-e^{-\gamma_g}\right)$}

\addplot [
    domain=0.01:1, 
    samples=30, 
    color=green,
    mark=+,
    ]   
{1-exp(-x/2.5)};
\addlegendentry{$\hspace*{-17pt}f_2(\gamma_g)=1-e^{-2\gamma_g/5}$}

\addplot [
    domain=0.01:1, 
    samples=30, 
    color=red,
    mark=triangle,
    ]   
{1-exp(-x)};
\addlegendentry{$\hspace*{-29pt}f_3(\gamma_g)=1-e^{-\gamma_g}$}

\addplot [
    domain=0.01:1, 
    samples=30, 
    color=magenta,
    mark=diamond,
    ]   
{1-exp(-3*x)};
\addlegendentry{$\hspace*{-25pt}f_4(\gamma_g)=1-e^{-2\gamma_g}$}
\end{axis}
\end{tikzpicture}
\caption{Plot of $f_i(\gamma_g)$, where $i=1,2,3,4$, for increasing values of $\gamma_g$ (each $f_i$ is defined in the figure's legend). By Definition~\ref{def:super_ratio} of $\gamma_g$, $\gamma_g$ takes values  between $0$ and $1$. 
}\label{fig:bounds}
\end{center}
\end{figure}

\begin{myremark}[Novelty of algorithm and bounds]\label{rem:novel1}
{Algorithm~\ref{alg:overall_nonUniformCosts} is the first scalable algorithm for Problem~\ref{prob:LQG}. Notably, although Algorithm~\ref{alg:greedy_nonUniformCosts} (used in Algorithm~\ref{alg:overall_nonUniformCosts}) is the same as the Algorithm~1 in~\cite{krause2005note},  the latter was introduced for \emph{exactly supermodular} optimization, instead of \emph{approximately supermodular optimization}, which is the  optimization framework in this paper.  Therefore, one of our contributions with Theorem~\ref{th:approx_bound} is to prove Algorithm~\ref{alg:greedy_nonUniformCosts} maintains suboptimality bounds even for approximately supermodular optimization.
The novel bounds in Theorem~\ref{th:approx_bound} also improve upon the previously known~\cite{krause2005note,leskovec2007cost} for exactly supermodular optimization: particularly, our bounds can become $1-1/e$ for supermodular optimization (the closer $c(\algS)/b$ is to 1), tightening the known $1/2(1-1/e)$~\cite{krause2005note,leskovec2007cost}. Noticeably, $1-1/e$ is the best possible bound in polynomial time for submodular optimization subject to \emph{cardinality} constraints~\cite{nemhauser78analysis}, instead of the general \emph{cost} constraints in this paper. That way, our analysis equates the approximation difficulty of \emph{cost} and \emph{cardinality} constrained optimization for the first time (among all algorithms with at most quadratic running time in the number of available elements in $\calV$, i.e., those in~\cite{krause2005note,leskovec2007cost,nemhauser78analysis}, and ours).}
\end{myremark}

{All in all, Theorem~\ref{th:approx_bound} guarantees that Algorithm~\ref{alg:overall_nonUniformCosts} achieves a close-to-optimal solution for Problem~\ref{prob:LQG}, whenever $\gamma_g > 0$.} 
In Section~\ref{sec:non-zero_ratio} we present conditions such that
$\gamma_g > 0$.

Finally, Theorem~\ref{th:approx_bound} also quantifies the scalability of  
Algorithm~\ref{alg:overall_nonUniformCosts}.  Particularly, Algorithm~\ref{alg:overall_nonUniformCosts}'s running time $O(|\calV|^2Tn^{2.4})$ is in the worst-case quadratic in the number of available sensors $\calV$ (when all must be chosen active), and linear in the Kalman filter's running time: specifically, the multiplier $Tn^{2.4}$ is due to the complexity of computing all $\Sigma\att{t}{t}$ for $t=1,2,\ldots,T$~\cite[Appendix~E]{bertsekas2005dynamic}.

\subsection{Performance analysis for Algorithm~\ref{algLQGconstr:overall_nonUniformCosts}}\label{sec:performanceGuaranteesLQGconstr}

\begin{mytheorem}[Performance of Algorithm~\ref{algLQGconstr:overall_nonUniformCosts}]\label{th:performance_LQGconstr} 
Consider Algorithm~\ref{algLQGconstr:overall_nonUniformCosts} returns a sensor set $\algS$ and control policies $u\of{1:T}{\algS}$.  Let $s_l$ be the last sensor added to $\algS$.  Then,
\begin{align}
&\hspace{-0.3cm} h[\algS,u_{1:T}(\algS)]\leq \kappa; \label{ineq:approx_boundLQGconstr}\\
&\hspace{-0.3cm}  \sensorCost(\algS)\leq 
\sensorCost(s_l)+ \frac{1}{\gamma_g}\log\left(\frac{h[\emptyset, u\of{1:T}{\emptyset}]-\kappa}{h[{\algS}_{l-1}, u\of{1:T}{{\algS}_{l-1}}]-\kappa}\right)\sensorBudget^\star\!\!,\label{ineq:cost_approx_bound}
\end{align}
where $\algS_{l-1}\triangleq\algS\setminus \{s_l\}$.

Additionally, Algorithm~\ref{algLQGconstr:overall_nonUniformCosts} runs in $O(|\calV|^2Tn^{2.4})$ time.
\end{mytheorem}

\begin{myremark}[Novelty of algorithm and bound]\label{rem:novel2}
{Algorithm~\ref{algLQGconstr:overall_nonUniformCosts} is the first scalable algorithm for Problem~\ref{prob:minCostLQG}.  Importantly, Algorithm~\ref{algLQGconstr:greedy_nonUniformCosts}, used in Algorithm~\ref{algLQGconstr:overall_nonUniformCosts}, is the first scalable algorithm with suboptimality guarantees for the problem of \emph{minimal cost set selection} where a bound to an \emph{approximately} supermodular $g$ must be met.  Particularly, 
Algorithm~\ref{algLQGconstr:greedy_nonUniformCosts}, generalizes previous algorithms \cite{wolsey1982analysis} that  focus instead on  \emph{minimal cardinality} set selection subject to bounds on an \emph{exactly} supermodular function $g$ (in which case, $\gamma_g=1$).  Notably, for $\gamma_g=1$, 
ineq.~\eqref{ineq:cost_approx_bound}'s bound recovers the \mbox{guarantee established  in~\cite[Theorem~1]{wolsey1982analysis}.}}
\end{myremark}

All in all, ineq.~\eqref{ineq:approx_boundLQGconstr} implies Algorithm~\ref{algLQGconstr:overall_nonUniformCosts} returns a  solution to Problem~\ref{prob:minCostLQG} with the prescribed  \LQG performance.  And parallel to~ineq.~\eqref{ineq:approx_bound}, ineq.~\eqref{ineq:cost_approx_bound} implies for $\gamma_g>0$ that  Algorithm~\ref{algLQGconstr:overall_nonUniformCosts} achieves a close-to-optimal sensor cost.

\subsection{Conditions for $\gamma_g>0$}\label{sec:non-zero_ratio}

We provide control-theoretic conditions such that $\gamma_g$ is non-zero, in which case both Algorithm~\ref {alg:overall_nonUniformCosts} and Algorithm~\ref{algLQGconstr:overall_nonUniformCosts} guarantee a close-to-optimal performance.  Particularly, we first prove that  if $\sum_{t=1}^T\Theta_t\succ 0$, then $\gamma_g$ is non-zero.  {Afterwards, 
we show the condition holds true in all problem instances one typically encounters in the real-world.  Specifically, we prove  $\sum_{t=1}^T\Theta_t\succ 0$ holds whenever zero control would result in a suboptimal behavior for the system; that is, we prove $\sum_{t=1}^T\Theta_t\succ 0$ holds in all systems where \LQG control improves system performance.}

\begin{mytheorem}[Non-zero computable bound for the supermodularity ratio $\gamma_g$]\label{th:submod_ratio}
For any sensor $i \in \calV$, let $\bar{C}_{i,t} \triangleq V_{i,t}^{-1/2}{C}_{i,t}$ be the 
\validated{normalized}{\emph{whitened}} measurement matrix.
If the strict inequality $\sum_{t=1}^{T}\Theta_t\succ 0$ holds, then $\gamma_g\neq 0$. Additionally, 
if we assume  $\trace{\bar{C}_{i,t}\bar{C}_{i,t}\tran}=1$, and $\text{tr}[\Sigma\att{t}{t}(\emptyset)]\leq \lambda_\max^2[\Sigma\att{t}{t}(\emptyset)]$, then 
\begin{align}\label{ineq:sub_ratio_bound}
\begin{split}
\gamma_g\geq &\frac{\lambda_\min(\sum_{t=1}^T \Theta_t) }{\lambda_\max(\sum_{t=1}^T \Theta_t)}\frac{ \min_{t\in\{1,2,\ldots,T\}}\lambda_\min^2[\Sigma\att{t}{t}(\calV)] }{\max_{t\in\{1,2,\ldots,T\}}\lambda_\max^2[\Sigma\att{t}{t}(\emptyset)]}\\
&\;\;\;\dfrac{1+\min_{i\in\calV, t\in\until{T}}\lambda_\min[\bar{C}_{i} \Sigma\att{t}{t}(\calV) \bar{C}_{i}\tran]}{2+\max_{i\in\calV, t\in\until{T}}\lambda_\max[\bar{C}_{i} \Sigma\att{t}{t}(\emptyset) \bar{C}_{i}\tran]
}.
 \end{split}
\end{align}
\end{mytheorem}

Ineq.~\eqref{ineq:sub_ratio_bound} suggests 
ways~$\gamma_g$ can increase, and, correspondingly, the bounds for Algorithm~\ref{alg:overall_nonUniformCosts} and of Algorithm~\ref{algLQGconstr:overall_nonUniformCosts} can improve:
when $\lambda_\min(\sum_{t=1}^T \Theta_t) /\lambda_\max(\sum_{t=1}^T \Theta_t)$ increases to $1$, then the right-hand-side in ineq.~\eqref{ineq:sub_ratio_bound} increases. {Therefore, since each $\Theta_t$ weight the states depending on their relevance for control purposes (Remark~\ref{rmk:interpretation}), the right-hand-side in ineq.~\eqref{ineq:sub_ratio_bound} increases when all the directions in the state space become equally important for control purposes. 
} \validated{To~see~this, for~example that}{Indeed, in the extreme case where} $\lambda_\max( \Theta_t)=\lambda_\min(\Theta_t)=\lambda$, the objective function in eq.~\eqref{eq:opt_sensors} becomes
\begin{align*}
\sum_{t=1}^{T}\text{tr}[\Theta_t\Sigma\att{t}{t}(\calS)]&=
\validated{}{\lambda \sum_{t=1}^{T}\text{tr}[\Sigma\att{t}{t}(\calS)]},
\end{align*}
\validated{Overall, 
it is {easier} for Algorithm~\ref{alg:overall_nonUniformCosts} to approximate a solution to Problem~\ref{alg:overall_nonUniformCosts} as the cost function in eq.~\eqref{eq:opt_sensors} becomes the cost function in the standard sensor selection problems where one minimizes the total estimation covariance as in eq.~\eqref{eq:interpretationRemark_kalman}.}{which matches the cost function in the classical sensor selection 
where all states are equally important (per eq.~\eqref{eq:interpretationRemark_kalman}). 
}


Theorem~\ref{th:submod_ratio} states $\gamma_g$ is non-zero whenever $\sum_{t=1}^{T} \Theta_t\succ 0$. 
To provide insight on the type of control problems for which this result holds,
 next we translate the technical condition $\sum_{t=1}^{T} \Theta_t\succ 0$ into an equivalent control-theoretic condition.

\begin{mytheorem}[\validated{Control-level condition for near-optimal sensor selection}{Control-theoretic condition for near-optimal co-design}] \label{th:freq}
Consider the 
{(noiseless, perfect state-information)} 
\LQG problem where at any $t=1,2,\ldots,T$, the state $x_t$ is known to each controller $u_t$ and the process noise $w_t$ is zero, i.e., the \validated{optimization}{optimal control} problem
\begin{equation}\label{pr:perfect_state}
\!\!\textstyle\min_{u_{1:T}}\sum_{t=1}^{T}\left.[\|x\at{t+1}\|^2_{Q_t} +\|u_{t}(x_t)\|^2_{R_t}]\right|_{\Sigma\att{t}{t}=W_t=0}.
\end{equation}

Let $A_t$  be invertible for all $t=1,2,\ldots,T$; the strict inequality $\sum_{t=1}^{T} \Theta_t\succ 0$ holds if and only if for all non-zero initial conditions~$x_1$, \validated{}{the all-zeroes control policy $u^{\circ}_{1:T}\triangleq(0,0,\ldots,0)$ is not an optimal solution to eq.~\eqref{pr:perfect_state}:}
\begin{align*}
u^{\circ}_{1:T} \!\notin \textstyle\arg\min_{u_{1:T}}\sum_{t=1}^{T}\!\left.[\|x\at{t+1}\|^2_{Q_t} +\|u_{t}(x_t)\|^2_{R_t}]\right|_{\Sigma\att{t}{t}=W_t=0}.
\end{align*}
\end{mytheorem}

{
Theorem~\ref{th:freq} suggests $\sum_{t=1}^{T} \Theta_t\succ 0$ 
holds if and only if for any non-zero initial condition~$x_1$ the all-zeroes control policy $u^{\circ}_{1:T}=(0,0,\ldots,0)$ is suboptimal for the noiseless, perfect state-information LQG problem.  Intuitively, this encompasses most practical control design problems 
where a zero controller would result in a suboptimal behavior of the system (\LQG control design itself would be unnecessary in the case where a zero controller, i.e., no control action, can \mbox{already attain the desired system performance}).
} 

Overall, Algorithm~\ref{alg:overall_nonUniformCosts} and Algorithm~\ref{algLQGconstr:overall_nonUniformCosts} are the first scalable algorithms for Problem~\ref{prob:LQG} and  Problem~\ref{prob:minCostLQG}, respectively, and 
they achieve non-vanishing per-instance performance guarantees.


\section{Numerical Evaluations}\label{sec:exp}

We consider two applications for the \LQG control and sensing co-design framework:
\emph{formation control}  and 
\emph{autonomous navigation}. 
We present a Monte Carlo analysis for both, which demonstrates: 
(i) the proposed sensor selection strategy is near-optimal; particularly, the resulting \LQG cost  matches the optimal selection in all instances for which the optimal could be computed via a brute-force approach;
(ii) a more naive selection which attempts to minimize the state estimation covariance~\cite{jawaid2015submodularity}
(rather than the \LQG cost) has degraded \LQG performance, often comparable to a random selection;
(iii) in the considered instances, a clever selection of a small subset of sensors can ensure an \LQG cost that is 
close to the one obtained by using all available sensors, hence providing an effective alternative for control under 
sensing constraints.

\newcommand{\mpw}{4.5cm}
\begin{figure}[t]
\hspace{-4mm}
\begin{minipage}{\textwidth}
\begin{tabular}{cc}%
\begin{minipage}{3.5cm}%
\centering
\includegraphics[width=.965\columnwidth]{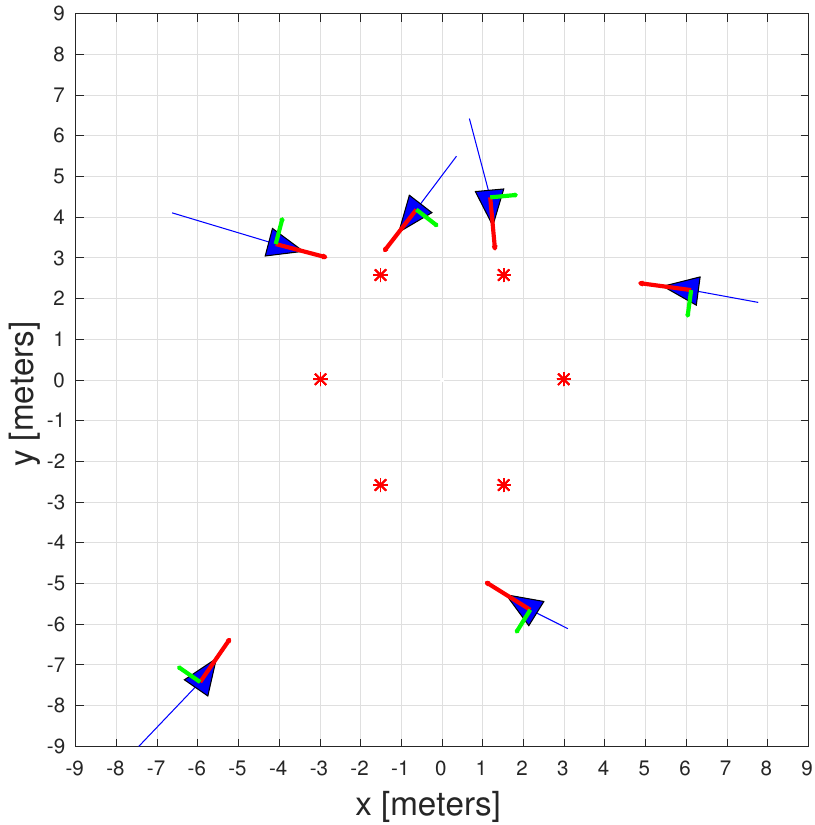} \\
(a) formation control 
\end{minipage}
& \hspace{-5mm}
\begin{minipage}{5.5cm}%
\centering%
\includegraphics[width=1.05\columnwidth,trim=0cm 1cm 0cm 2cm,clip]{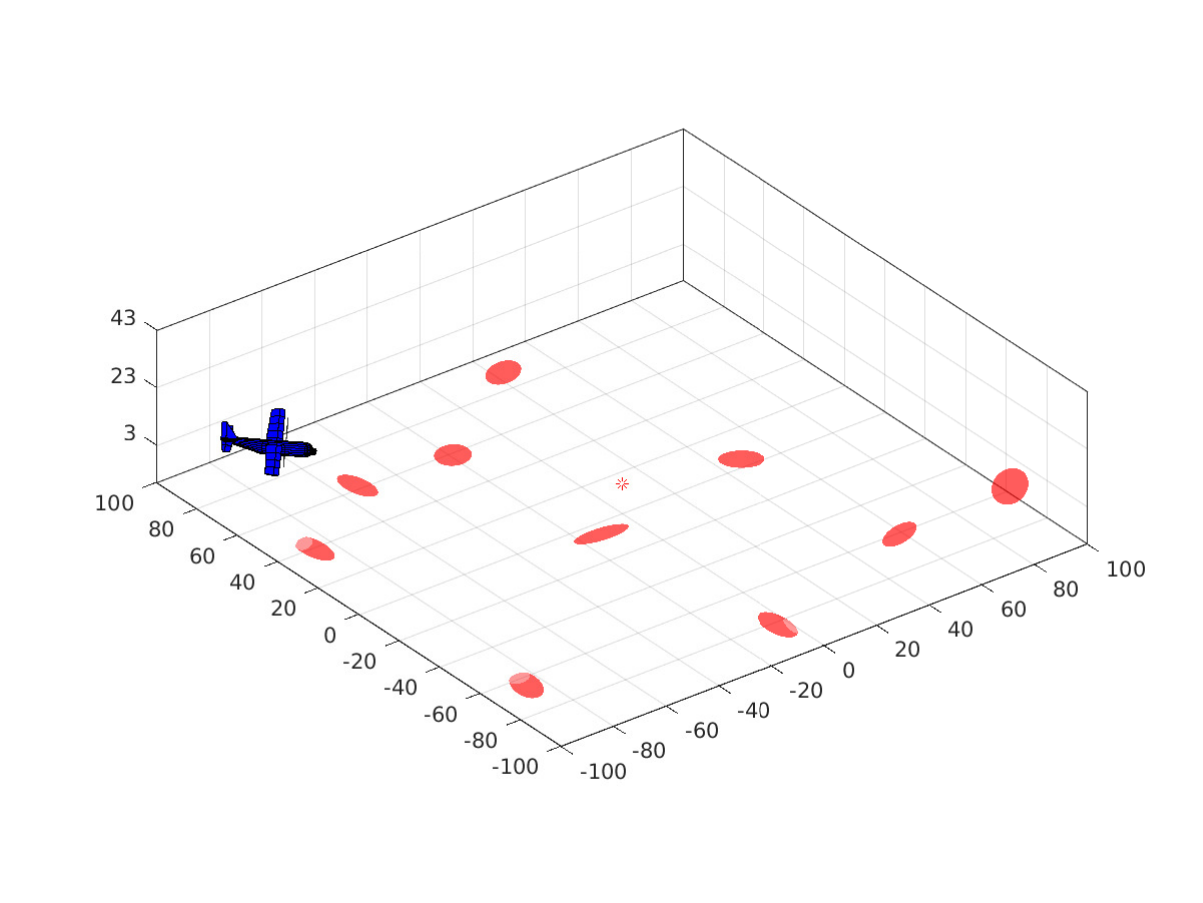} \\
(b) robot navigation
\end{minipage}
\end{tabular}
\end{minipage}%
\vspace{-1mm}
\caption{\label{fig:applications}
Applications of the \LQG control and sensing co-design framework. 
}
\end{figure}

\subsection{Sensing-constrained formation control}\label{sec:exp-formationControl}

\myParagraph{Simulation setup}  The application scenario is illustrated in \myFigure{fig:applications}(a).
A team of $\nrRobots$ agents (blue triangles) moves in 2D.
At $t=1$, the agents are randomly deployed in a $10\rm{m} \times 10\rm{m}$ square. 
Their objective is to reach a target formation shape 
 (red stars); in \myFigure{fig:applications}(a) the desired
 formation has an hexagonal shape, while in general for a formation of $\nrRobots$, the desired 
 formation is an equilateral polygon with $\nrRobots$ vertices. 
 Each robot is modeled as a double-integrator, with state $x_i = [p_i \; v_i]\tran \in \Real{4}$  
 ($p_i$ is agent $i$'s position, and $v_i$ its velocity), and can control its acceleration 
 $u_i \in \Real{2}$. The process noise is a diagonal matrix $W = \diag{[1e^{-2}, \; 1e^{-2}, \; 1e^{-4}, \; 1e^{-4}]}$. 
 Each robot $i$ is equipped with a GPS, which measures the agent position $p_i$ with a covariance 
 $V_{gps,i} = 2\cdot\eye_2$. Moreover, the agents are equipped with lidars allowing 
 each agent~$i$ to measure the relative position of another agent~$j$ with covariance 
 $V_{lidar,ij} = 0.1\cdot\eye_2$. 
 The agents have limited on-board resources, hence they want to activate only $k$ sensors.

For our tests, we consider two setups. In the first, named \emph{homogeneous formation control}, 
the \LQG weight matrix $Q$ is a block diagonal matrix with $4\times 4$ blocks, 
and each block $i$ chosen as $Q_i = 0.1 \cdot\eye_4$;
since each block of~$Q$ weights equally the tracking error of a robot, in the homogeneous case
the tracking error of all agents is equally important.
In~the second setup, named \emph{heterogeneous formation control},  $Q$ 
is chose as above, except for one of the agents, say robot 1, for which we choose 
$Q_1 = 10 \cdot\eye_4$; this setup models the case in which 
each agent has a different role or importance, hence one weights differently the 
tracking error of the agents.  In~both cases the matrix $R$ is chosen to be the identity matrix. 
The simulation is carried on over $T$ time steps, and $T$ is also chosen as \LQG horizon.
Results are averaged over 100 Monte Carlo runs: at each run we randomize the 
initial estimation covariance {$\initialCovariance$}.

\myParagraph{Compared techniques}
We compare five techniques. All techniques use an \LQG-based estimator and controller, and they 
only differ by the selections of the  active sensors.
The~first approach is the optimal sensor selection, denoted as \toptimal, which 
attains the minimum in eq.~\eqref{eq:opt_sensors}, and which we compute by enumerating all 
possible subsets.
The second approach is a pseudo-random sensor selection, denoted as \trandom, 
 which selects all the GPS measurements and a random subset of the lidar measurements.
The third approach, denoted as \tlogdet, selects sensors so to minimize the average $\logdet$ of the estimation covariance 
over the horizon; this approach resembles~\cite{jawaid2015submodularity} and is agnostic to the control task.
The fourth approach is the proposed sensor selection strategy (Algorithm~\ref{alg:greedy_nonUniformCosts}), and is denoted as \tslqg.
Finally, we also report the \LQG performance when all sensors are selected.
This approach is denoted as \tallSensors.

\myParagraph{Results}
The results of the numerical analysis are reported in \myFigure{fig:formationControlStats}.
When not specified otherwise, we consider a formation of $\nrRobots = 4$ agents, 
which can only use a total of $k=6$ sensors, and a control horizon $T=20$. 
\myFigure{fig:formationControlStats}(a) shows the \LQG cost for the homogeneous case and for increasing horizon. 
We note that, in all tested instance, the proposed approach \tslqg matches the optimal selection \toptimal, and both 
approaches are relatively close to \tallSensors, which selects all the 
 available sensors. On the other hand, \tlogdet leads to worse tracking 
 performance, and is often close to \trandom.
These considerations are confirmed by the heterogeneous setup, in \myFigure{fig:formationControlStats}(b).
In this case, the separation between our proposed approach and \tlogdet becomes even larger;
 the intuition is that the heterogeneous case rewards differently the tracking errors at different agents, 
 hence while \tlogdet attempts to equally reduce the estimation error across the formation, the proposed approach 
 \tslqg selects sensors in a task-oriented fashion, since the matrices $\Theta_t$ for all $\allT$ in the cost function in eq.~\eqref{eq:opt_sensors}
 incorporate the \LQG weight matrices.

\myFigure{fig:formationControlStats}(c) shows the \LQG cost attained for increasing 
number of selected sensors $k$ and for the homogeneous case.   
For increasing number of sensors all techniques converge to \tallSensors (since the entire ground set is selected).
\myFigure{fig:formationControlStats}(d) shows the same statistics for the heterogeneous case. 
Now, \tslqg matches \tallSensors earlier, starting at $k = 7$;
intuitively, in the heterogeneous case, adding more sensors may have 
marginal impact on the \LQG cost (e.g., if the cost rewards a small tracking error for robot 1, it may be 
of little value to take a lidar measurement between robot 3 and 4). This further stresses the importance of the 
proposed framework as a parsimonious way to control a system with minimal resources.

\myFigure{fig:formationControlStats}(e) 
and \myFigure{fig:formationControlStats}(f) show the \LQG cost attained by the compared techniques for increasing 
number of agents.
\toptimal quickly becomes intractable to compute, hence we omit values beyond $\nrRobots = 4$.
In both figures, the separation among the techniques increases with the number of agents,
since the set of available sensors quickly increases with $n$. In the heterogeneous case  
\tslqg remains relatively close to \tallSensors, implying that for the purpose of \LQG control, using a cleverly selected
 small subset of  sensors still ensures excellent tracking performance.


\newcommand{\resultsFolderFormationControl}{code/results-formationControl-gpsAndRandom-100-randInitCov-newLQG}
\newcommand{\myhspace}{\hspace{-2mm}}

\renewcommand{\mpw}{4.5cm}
\begin{figure}[t]
\myhspace\myhspace
\begin{minipage}{\textwidth}
\begin{tabular}{cc}%
\myhspace
\begin{minipage}{\mpw}%
\centering
\includegraphics[width=1.05\columnwidth]{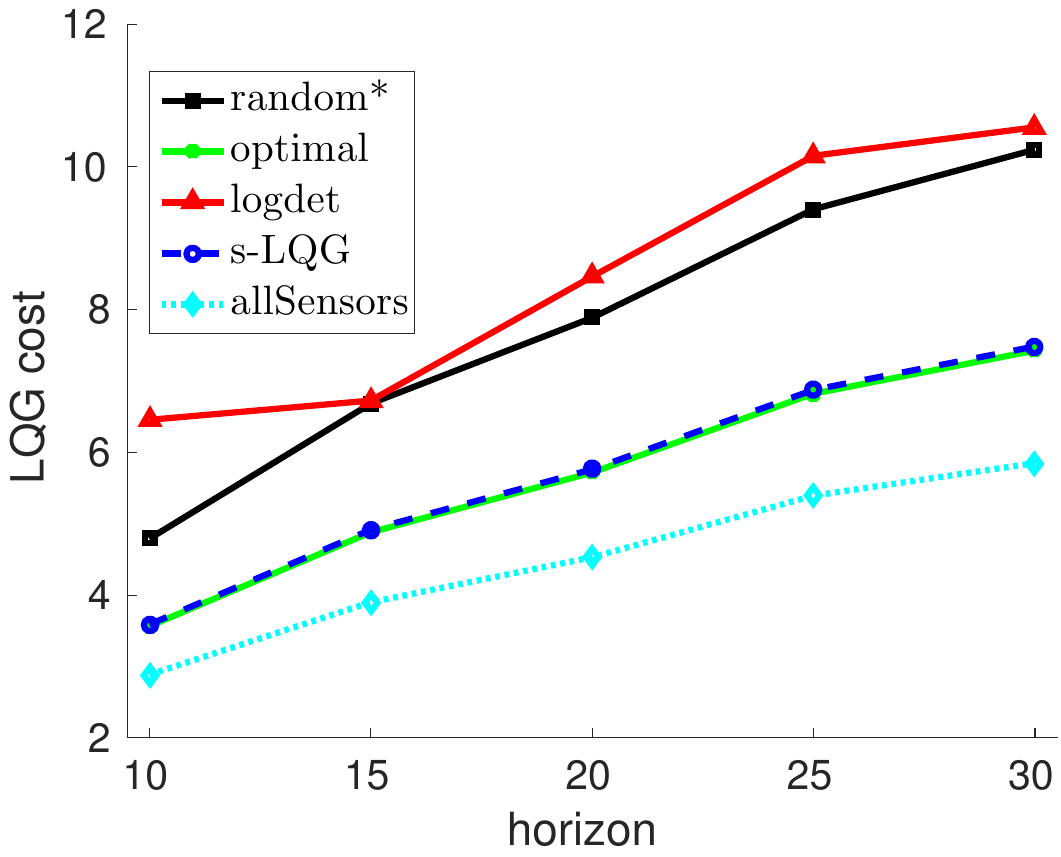} \\
(a) homogeneous 
\end{minipage}
& \myhspace
\begin{minipage}{\mpw}%
\centering%
\includegraphics[width=1.05\columnwidth]{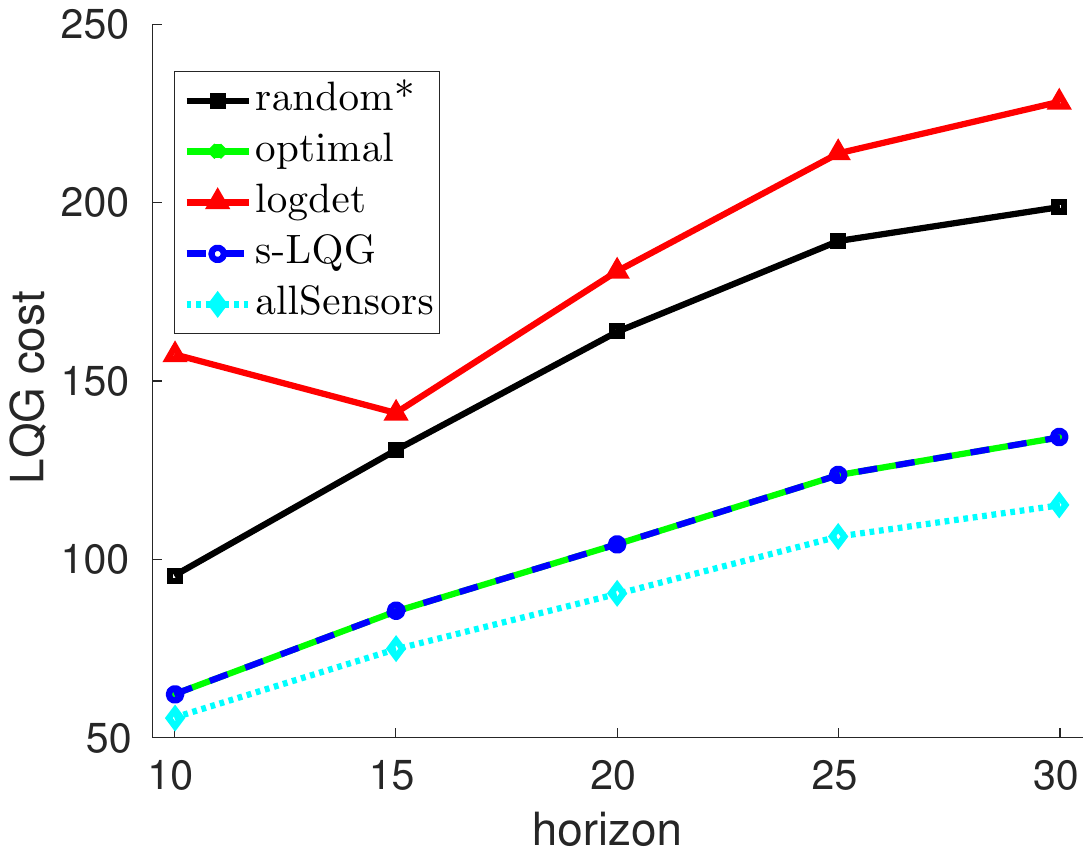} \\
(b) heterogeneous 
\end{minipage}
\\
\myhspace
\begin{minipage}{\mpw}%
\centering
\includegraphics[width=1.05\columnwidth]{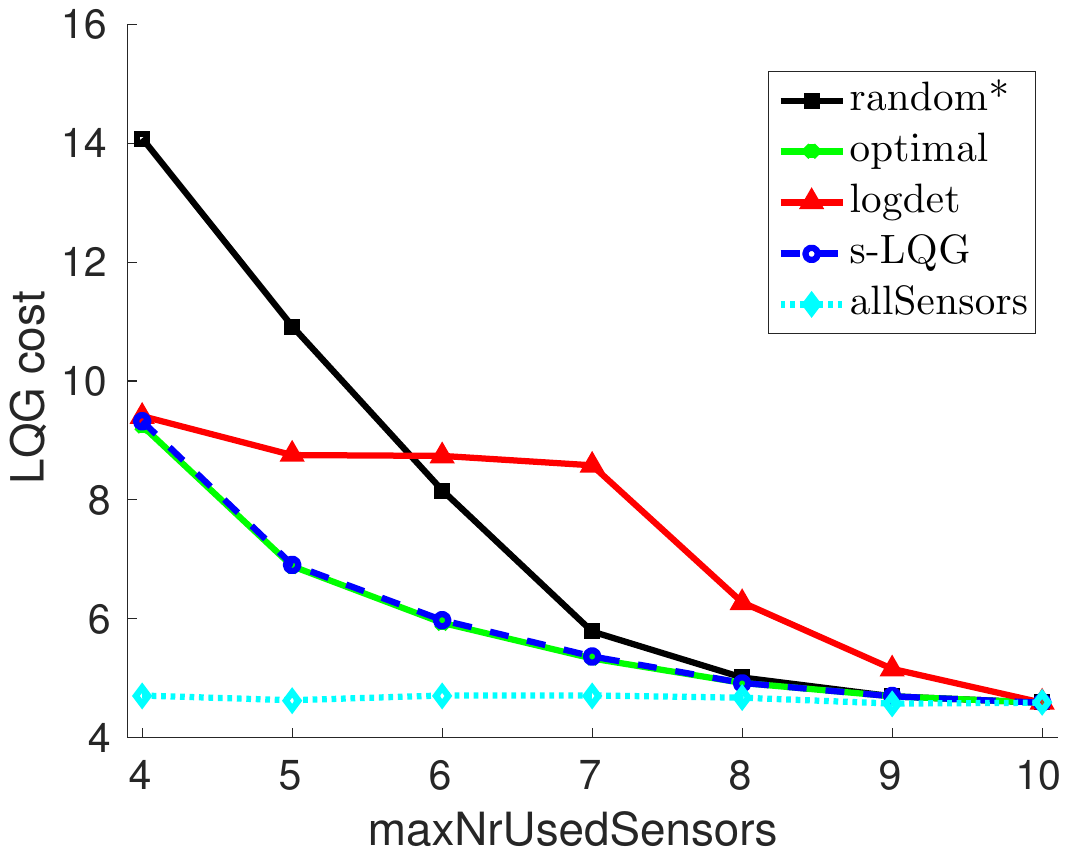} \\
(c) homogeneous 
\end{minipage}
& \myhspace
\begin{minipage}{\mpw}%
\centering%
\includegraphics[width=1.05\columnwidth]{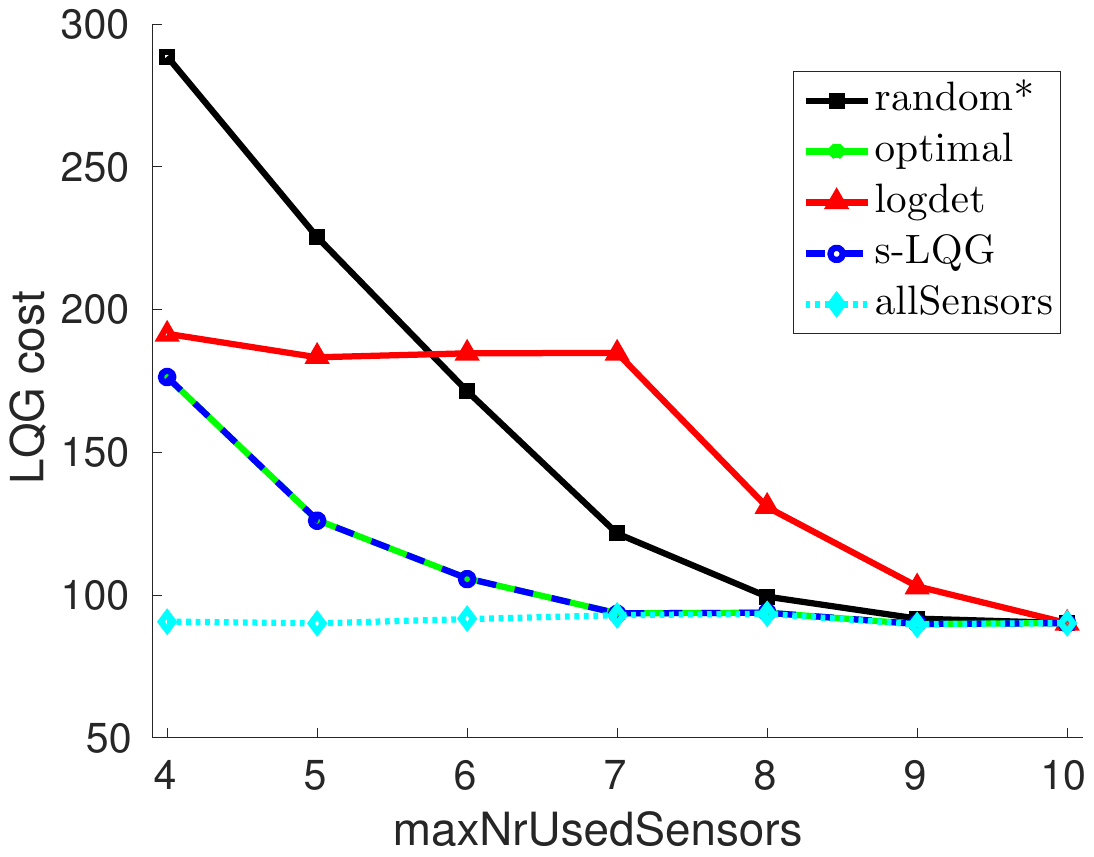} \\
(d) heterogeneous 
\end{minipage}
\\
\myhspace
\begin{minipage}{\mpw}%
\centering
\includegraphics[width=1.05\columnwidth]{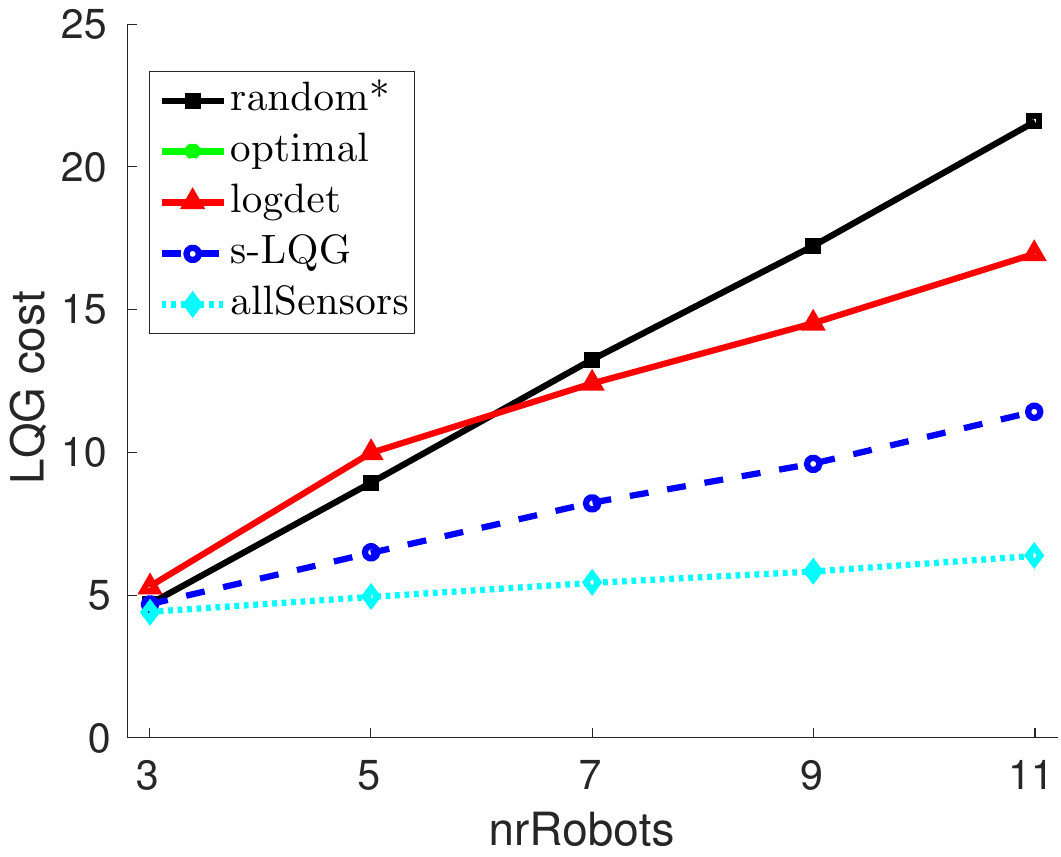} \\
(e) homogeneous 
\end{minipage}
& \myhspace
\begin{minipage}{\mpw}%
\centering%
\includegraphics[width=1.05\columnwidth]{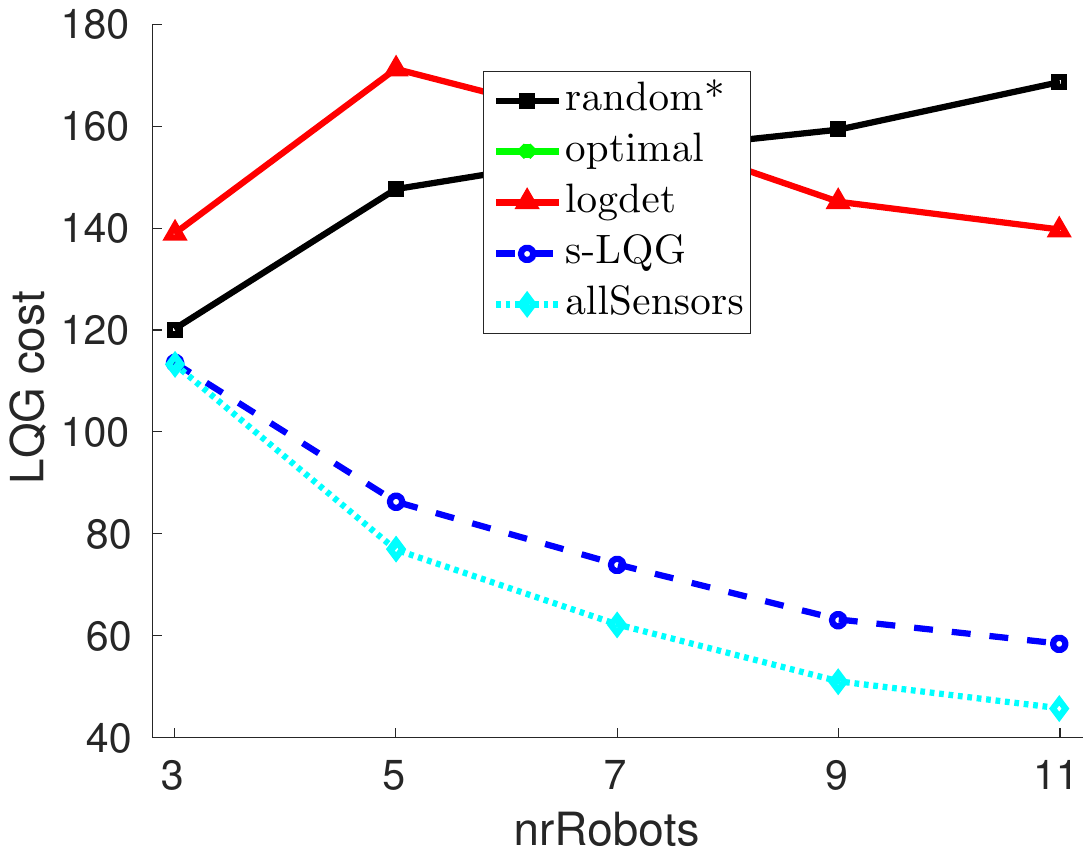} \\
(f) heterogeneous 
\end{minipage}
\end{tabular}
\end{minipage}%
\vspace{-1mm}
\caption{\label{fig:formationControlStats}
\LQG cost for increasing (a)-(b) control horizon $T$, (c)-(d) number of used 
sensors $k$ (all sensors are considered to have sensor-cost~1), and (e)-(f) number of agents $n$.  Statistics are reported for the homogeneous 
formation control setup (left column), and the heterogeneous setup (right column).
}\vspace{-5mm}
\end{figure}

\subsection{Resource-constrained robot navigation}\label{sec:exp-flyingRobot}

\myParagraph{Simulation setup}  The second application scenario is illustrated in \myFigure{fig:applications}(b).
An unmanned aerial robot (UAV) moves in a 3D space, starting from a
randomly selected location. 
The objective of the UAV is to land, and specifically,  to reach 
 $[0,\;0,\;0]$ with zero velocity. 
 The UAV is modeled as a double-integrator, with state $x = [p \; v]\tran \in \Real{6}$  
 ($p$ is the position, while $v$ its velocity), and can control its acceleration 
 $u \in \Real{3}$.  The process noise is $W = \eye_6$. 
 The UAV is equipped with multiple sensors. It has an on-board GPS, measuring the 
 UAV position  $p$ with a covariance $2 \cdot\eye_3$, 
and an altimeter, measuring only the last component of $p$ (altitude) with standard deviation $0.5\rm{m}$. 
Moreover, the UAV 
can use a stereo camera to measure the relative position of $\ell$ landmarks on the ground;
 we assume the location of each landmark to be known 
 approximately, and we associate to each landmark an uncertainty covariance
(red ellipsoids in \myFigure{fig:applications}(b)), which is randomly generated at the 
beginning of each run.
 The UAV has limited on-board resources, hence it wants to use only a few of sensing modalities.
For instance, the resource-constraints may be due to the power consumption of the GPS and the altimeter,
or may be due to computational constraints that prevent to run multiple object-detection algorithms to detect all landmarks on the ground.  We consider two sensing-constrained scenarios: (i)~all sensors to have the same cost (equal to $1$), in which case, the UAV can activate at most $k$ sensors; (ii)~the sensors to have heterogeneous costs: particularly, the  GPS's cost  is set equal to $3$; the altimeter's cost is set equal to 2; and each landmark's cost is set equal to $1$.

We use $Q = \diag{[1e^{-3},\; 1e^{-3},\;10,\; 1e^{-3},\; 1e^{-3},\; 10]}$ 
and $R = \eye_3$. The structure of $Q$ reflects the fact that during landing 
we are particularly interested in controlling the vertical direction and the vertical velocity 
(entries with larger weight in $Q$), while we are less interested in controlling accurately the 
horizontal position and velocity (assuming a sufficiently large landing site).
In the following, we present results averaged over 100 Monte Carlo runs: in each run, we~randomize the 
covariances describing the landmark position uncertainty.

\myParagraph{Compared techniques}
We consider the five techniques discussed in the previous section.

\myParagraph{Results}
The results of our numerical analysis are reported in \myFigure{fig:FlyingRobotStats} for the case where all sensors have the same sensor-cost, and in \myFigure{fig:heteregenousCost} for the case where sensors have different costs.
When not specified otherwise, we consider a total of $k=3$ sensors to be selected, and a control horizon $T=20$. 


\newcommand{\resultsFolderFlyingRobot}{code/results-flyingRobot-gpsAndRandom-100-randInitCov-newLQG}
\renewcommand{\myhspace}{\hspace{-2mm}}

\renewcommand{\mpw}{4.5cm}
\begin{figure}[t]
\myhspace\hspace*{-1mm}
\begin{minipage}{\textwidth}
\begin{tabular}{cc}%
\myhspace
\begin{minipage}{\mpw}%
\centering
\includegraphics[width=1\columnwidth]{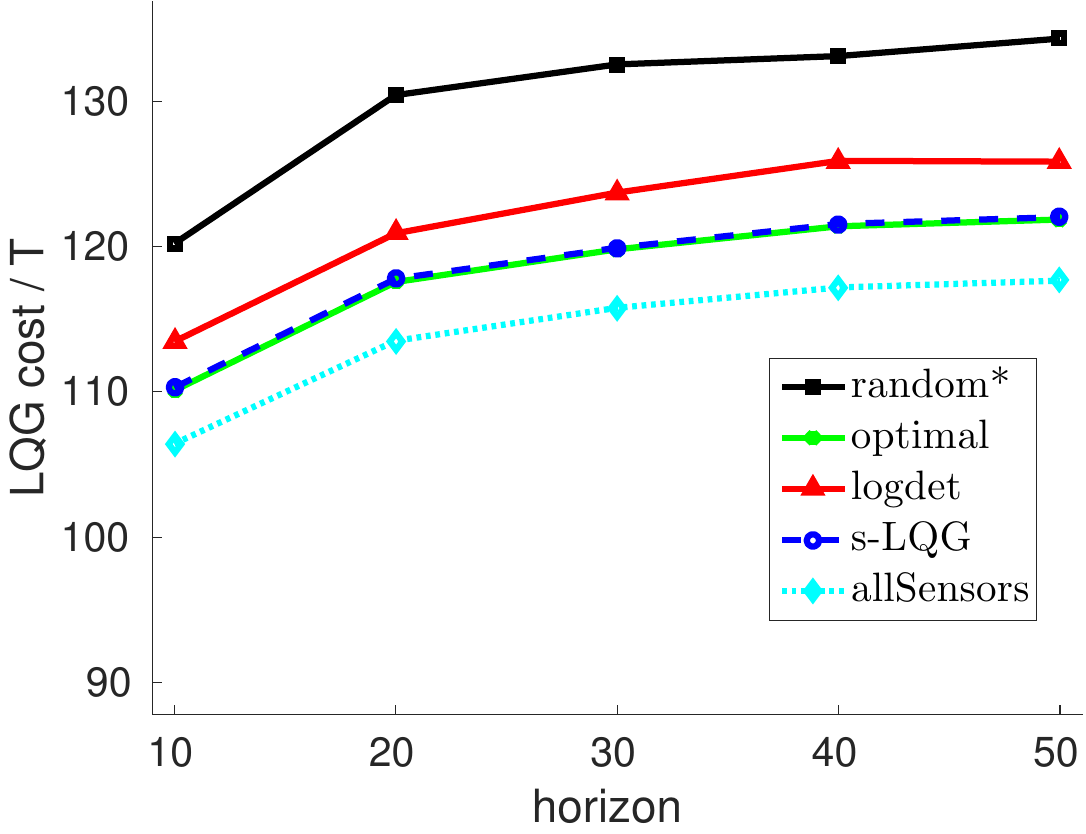} \\
(a) heterogeneous 
\end{minipage}
& \myhspace
\begin{minipage}{\mpw}%
\centering%
\includegraphics[width=1\columnwidth]{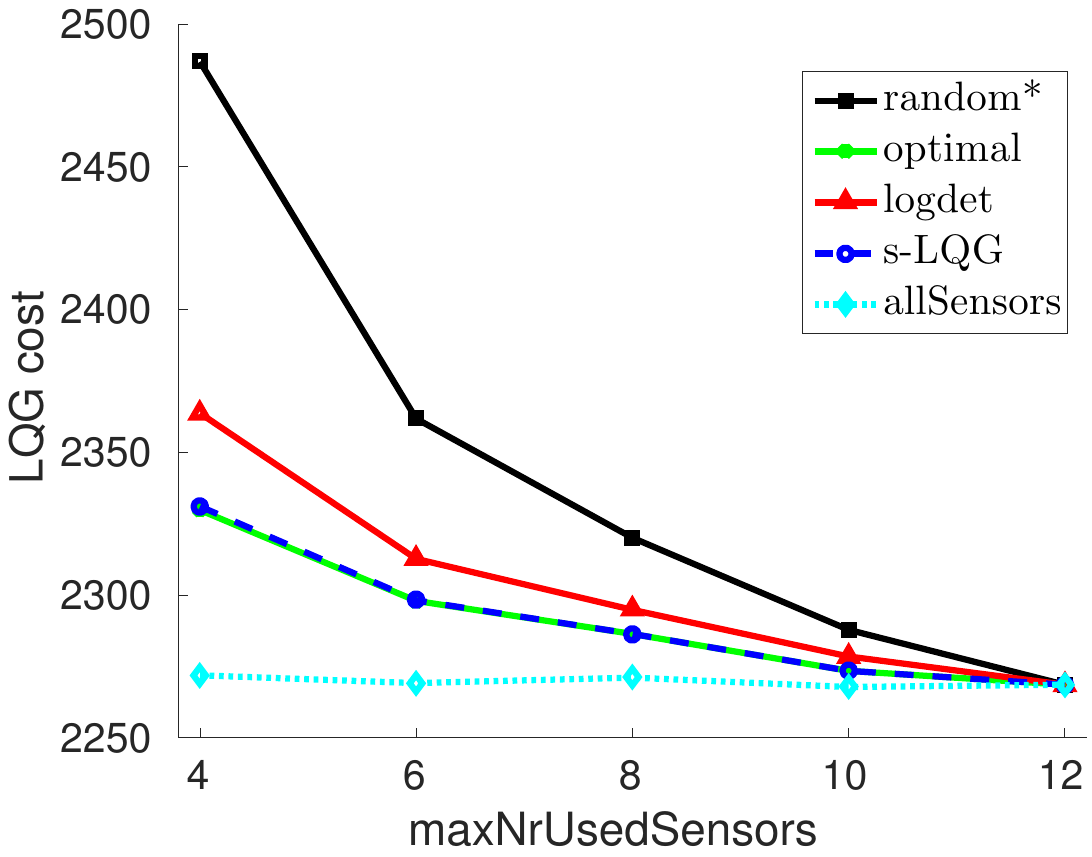} \\
(b) heterogeneous 
\end{minipage}
\end{tabular}
\end{minipage}%
\vspace{-1mm}
\caption{\label{fig:FlyingRobotStats}
\LQG cost for increasing (a) horizon $T$\!, and (b) number of used 
sensors~$k$ (all sensors are considered to have cost~1).  
}\vspace{-5mm}
\end{figure}

In \myFigure{fig:FlyingRobotStats}(a) 
we plot the \LQG cost normalized by the horizon, which makes more visible 
the differences among the techniques. Similarly to the formation control example, \tslqg matches the optimal selection \toptimal, 
while \tlogdet and \trandom have suboptimal performance. 
\myFigure{fig:FlyingRobotStats}(b) shows the \LQG cost attained by the compared techniques for increasing 
number of selected sensors $k$. All techniques converge to \tallSensors for increasing $k$, but in the 
regime in which few sensors are used \tslqg still outperforms alternative sensor selection schemes, and matches \toptimal.

\myFigure{fig:heteregenousCost} shows the \LQG cost attained by the compared techniques for increasing control horizon and various sensor cost budgets~$b$.    
Similarly to \myFigure{fig:FlyingRobotStats}, \tslqg has the same performance as \toptimal, 
whereas \tlogdet and \trandom have suboptimal performance. Notably, for $b=15$ all sensors can be chosen; for this reason in \myFigure{fig:heteregenousCost}(d) all compared techniques (but the random) have the same performance.

\section{Concluding Remarks}\label{sec:con}

{We addressed {an} \LQG control and sensing co-design problem, where one jointly designs control and sensing policies under resource constraints.
The problem is central in modern IoT and IoBT control applications, ranging from large-scale networked systems to miniaturized robotic networks.
{Motivated by the inapproximability of the problem}, we provided the first scalable algorithms with {per-instance} suboptimality bounds.  Importantly, the bounds are non-vanishing under general control-theoretic conditions, encountered in most real-world systems.
To this end, we also extended 
the literature on supermodular optimization: {by providing scalable algorithms for optimizing approximately supermodular functions subject to heterogeneous cost constraints; and by providing novel suboptimality bounds that improve the known bounds even for exactly supermodular optimization. }

The paper opens several avenues for future research. 
First, {the development of distributed implementations of the proposed algorithms would offer computational speedups.} 
Second, other co-design problems are interesting to be explored, such as the co-design of control-sensing-actuation. 
Third, while we provide bounds on an approximate sensor design against  optimal design, 
one could provide bounds against
the case where all sensors are used~\cite{siam2018ecc}.
Finally, in adversarial or failure-prone scenarios, one must account for sensor failures; 
to this end, one could leverage recent results on \emph{robust combinatorial optimization}~\cite{tzoumas2017resilient}.
}


\renewcommand{\myhspace}{\hspace{-1mm}}

\renewcommand{\mpw}{4.5cm}
\begin{figure}[t]
\myhspace\myhspace
\begin{minipage}{\textwidth}
\begin{tabular}{cc}%
\myhspace
\begin{minipage}{\mpw}%
\centering
\includegraphics[width=1.05\columnwidth]{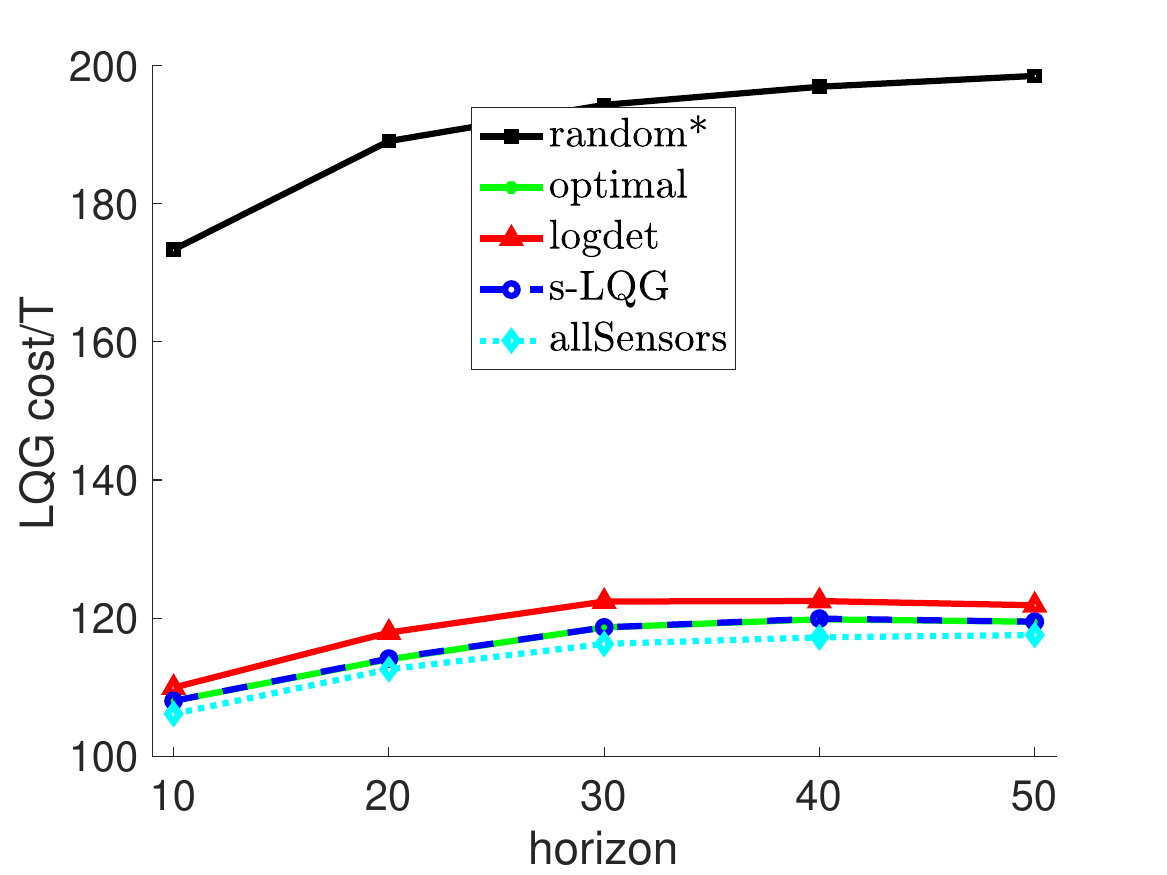} \\
(a)  budget $b=6$
\end{minipage}
& \myhspace\myhspace
\begin{minipage}{\mpw}%
\centering%
\includegraphics[width=1.05\columnwidth]{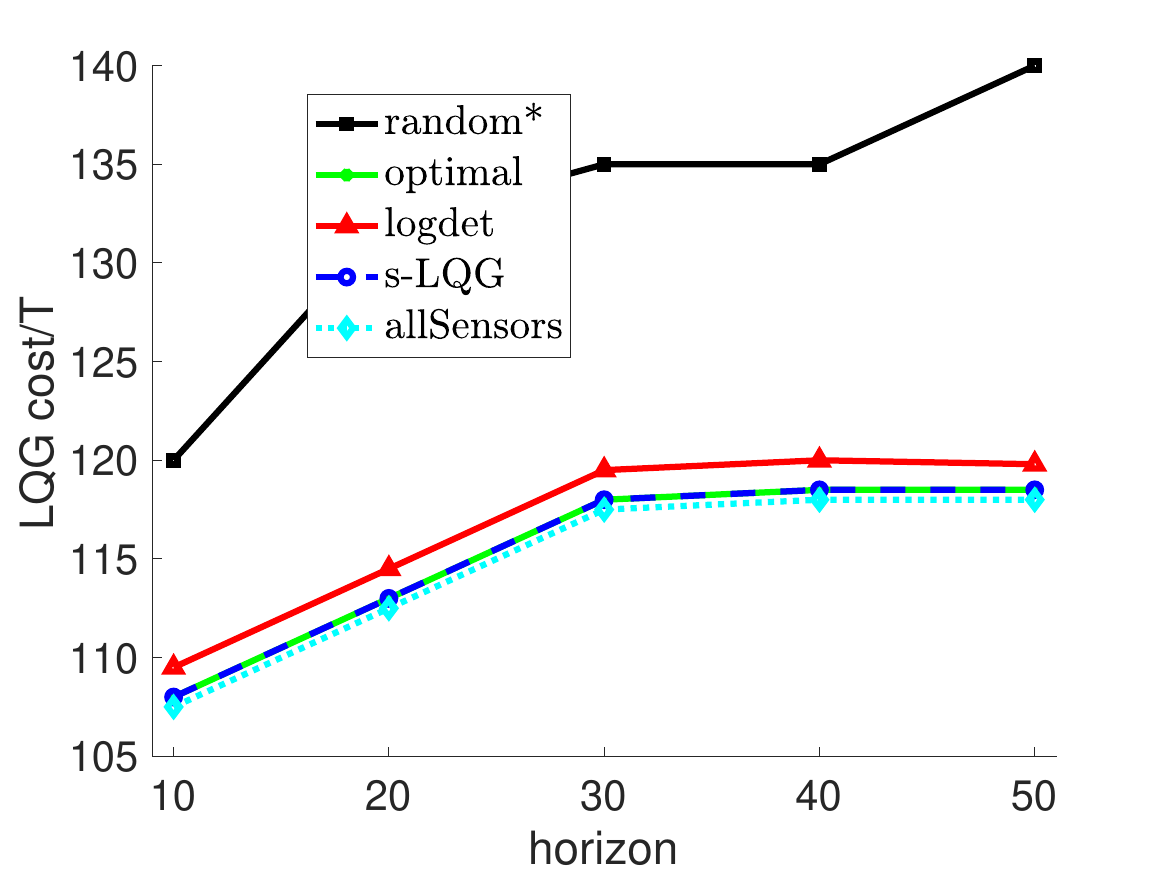} \\
(b)  budget $b=8$ 
\end{minipage}
\\
\myhspace
\begin{minipage}{\mpw}%
\centering
\includegraphics[width=1.05\columnwidth]{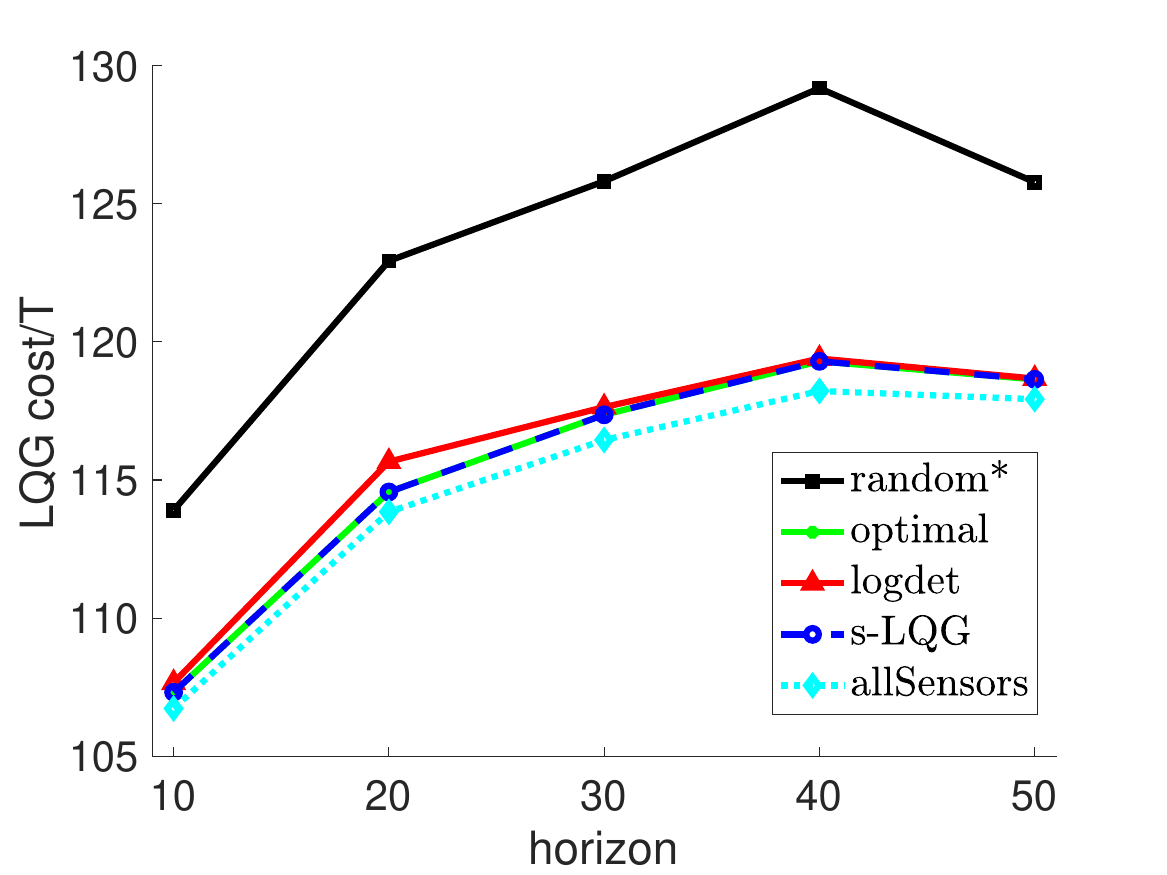} \\
(c)  budget $b=10$ 
\end{minipage}
& \myhspace\myhspace
\begin{minipage}{\mpw}%
\centering%
\includegraphics[width=1.05\columnwidth]{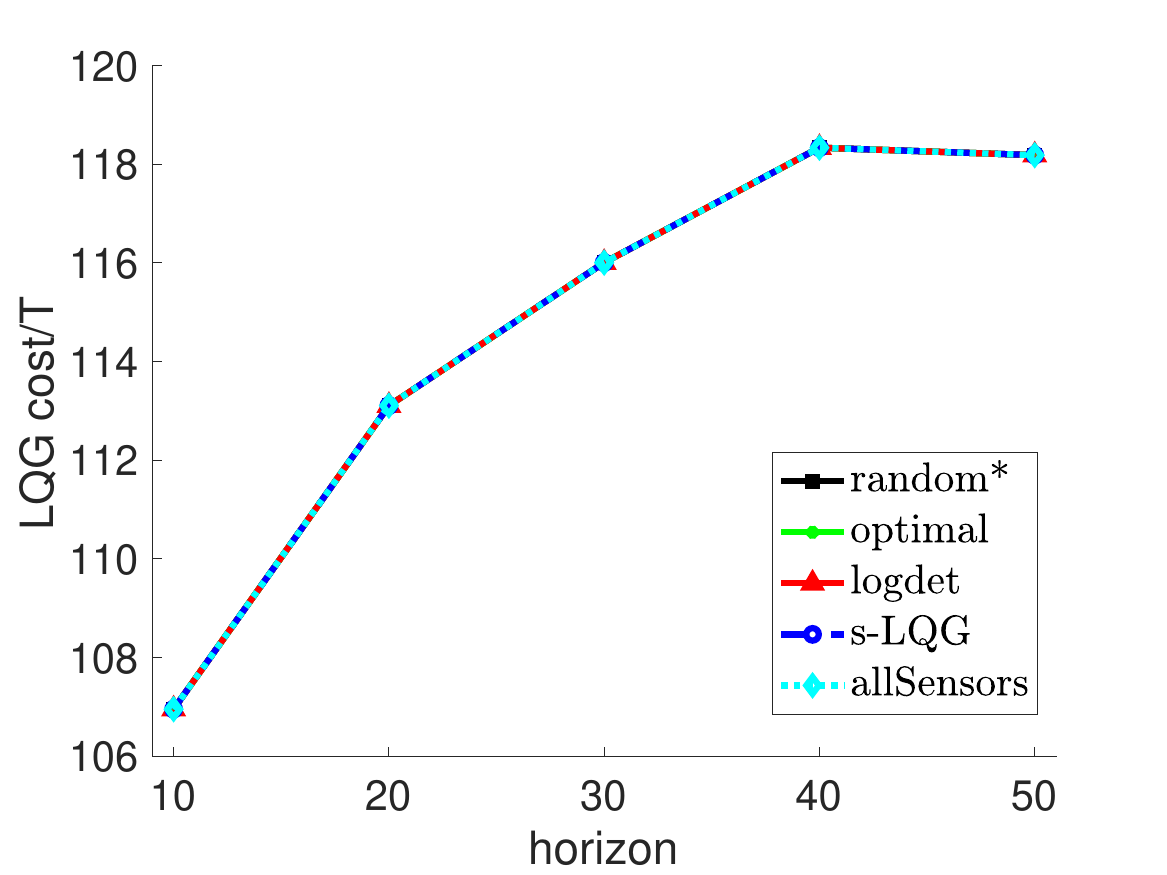} \\
(d)  budget $b=15$
\end{minipage}
\end{tabular}
\end{minipage}
\vspace{-1mm}
\caption{\label{fig:heteregenousCost}
\LQG cost for increasing horizon $T$ and for various sensing budgets~$b$. 
}\vspace{-5mm}
\end{figure}


\appendices 



\section*{Appendix A: Preliminary facts}


\begin{mylemma}[\hspace{-0.05mm}{\cite[Proposition 8.5.5]{bernstein2005matrix}}]\label{lem:inverse}
Consider two positive definite matrices $A_1$ and $A_2$. If $A_1\preceq A_2$, then $A_2\inv\preceq A_1\inv\!$.
\end{mylemma}

\begin{mylemma}[Trace inequality~{\cite[Proposition~8.4.13]{bernstein2005matrix}}]\label{lem:trace_low_bound_lambda_min}
Consider a symmetric $A$, and a positive semi-definite $B$. 
Then,
\begin{equation*}
\lambda_\min(A)\trace{B}\leq \trace{AB}\leq \lambda_\max(A)\trace{B}.
\end{equation*}
\end{mylemma}

\begin{mylemma}[Woodbury identity~{\cite[Corollary 2.8.8]{bernstein2005matrix}}]\label{lem:woodbury}
Consider $A$, $C$, $U$ and $V$ such that $A$, $C$, and $A+UCV$ are invertible. Then,
\begin{equation*}
(A+UCV)\inv=A\inv-A\inv U(C\inv+VA\inv U)\inv VA\inv\!.
\end{equation*}
\end{mylemma}

\begin{mylemma}[\hspace{-0.05mm}{\cite[Proposition 8.5.12]{bernstein2005matrix}}]\label{lem:traceAB_mon}
Consider two symmetric matrices $A_1$ and $A_2$, and a positive semi-definite matrix $B$. If~$A_1\preceq A_2$, then $\trace{A_1B}\leq \trace{A_2B}$.  
\end{mylemma}

\begin{mylemma}[\hspace{-0.05mm}{\cite[Appendix~E]{bertsekas2005dynamic}}]\label{lem:covariance_riccati}
For any sensors $\calS$, $\Sigma\att{t}{t}(\calS)$ is the solution of the Kalman filtering recursion
\beq\label{eq:covariance_riccati}
\begin{array}{rcl}
\Sigma\att{t}{t}(\calS) \!\!&\!\!=\!\!&\!\![ 
\Sigma\att{t}{t-1}(\calS)\inv + C_t(\calS)\tran V_t(\calS)\inv C_t(\calS)]\inv\!,\\
\Sigma\att{t+1}{t}(\calS) \!\!&\!\!=\!\!&\!\! A_{t} \Sigma\att{t}{t}(\calS) A_{t}\tran + W_{t},
\end{array}
\eeq
with boundary \validated{condition the}{condition} $\initialCovariance(\calS)=\initialCovariance$. 
\end{mylemma}

\begin{mylemma}[\hspace{-0.05mm}{\cite[Lemma~6]{Tzoumas18acc-scLGQG}}]\label{prop:one-step_monotonicity} Consider two sensor sets $\calS_1,\calS_2 \subseteq \calV$. {If~$\calS_1\subseteq \calS_2$, then $\Sigma\att{1}{1}(\calS_1)\succeq \Sigma\att{1}{1}(\calS_2)$}.  
\end{mylemma}

\begin{mylemma}[\hspace{-0.05mm}{\cite[Corollary~1]{Tzoumas18acc-scLGQG}}]\label{cor:from_t_to_t}
Let $\Sigma\att{t}{t}$ be defined as in eq.~\eqref{eq:covariance_riccati} with boundary \validated{condition the}{condition} $\Sigma\att{1}{1}$; similarly, let $\bar{\Sigma}\att{t}{t}$ be defined as in eq.~\eqref{eq:covariance_riccati} with boundary \validated{condition the}{condition} $\bar{\Sigma}\att{1}{1}$.
If~$\Sigma\att{t}{t}\preceq\bar{\Sigma}\att{t}{t}$, then $\Sigma\att{t+i}{t+i}\preceq\bar{\Sigma}\att{t+i}{t+i}$ for any positive integer $i$.  
\end{mylemma}

\begin{mylemma}[\hspace{-0.05mm}{\cite[Corollary~2]{Tzoumas18acc-scLGQG}}]\label{cor:from_t_to_t+1}
Let $\Sigma\att{t}{t}$ be defined as in eq.~\eqref{eq:covariance_riccati} with boundary \validated{condition the}{condition} $\Sigma\att{1}{1}$; similarly, let $\bar{\Sigma}\att{t}{t}$ be defined as in eq.~\eqref{eq:covariance_riccati} with boundary \validated{condition the}{condition} $\bar{\Sigma}\att{1}{1}$.
If~$\Sigma\att{t}{t}\preceq\bar{\Sigma}\att{t}{t}$, then $\Sigma\att{t+i}{t+i-1}\preceq\bar{\Sigma}\att{t+i}{t+i-1}$ for any positive integer $i$.  
\end{mylemma}

\begin{mylemma}\label{lem:minimum_of_products}
Consider positive real numbers $a$, $b$, $\gamma$, $a_1,a_2,\ldots,a_n$ such that $\sum_{i=1}^{n}a_i=a$.  Then, $$f(a_1,a_2,\ldots,a_n)=1-\prod_{i=1}^{n}\left(1-\gamma\frac{a_i}{b}\right)$$
has its minimum at $a_1=a_2=\ldots=a_n=a/n$, and
$$f(a/n,a/n,\ldots,a/n)=1-\left(1-\validated{\frac{a\gamma/b}{n}}{\frac{a\gamma}{b n}}\right)^n\geq 1-e^{-a\gamma/b}\!.$$
\end{mylemma}
\begin{proof}[Proof of Lemma~\ref{lem:minimum_of_products}]
The proof is obtained using the method of Lagrange multipliers, and is omitted (for a complete proof, see~\cite[Proof of Lemma~9]{tzoumas2018lqg}) 
\end{proof}

\begin{mylemma}[Monotonicity of cost function in eq.~\eqref{eq:opt_sensors} {\cite[Proposition~2]{Tzoumas18acc-scLGQG}}]\label{prop:monotonicity}
Consider $\sum_{t=1}^{T}\trace{\Theta_t \Sigma\att{t}{t}(\calS)}$  in eq.~\eqref{eq:opt_sensors}.  If $\calS_1\subseteq \calS_2 $, then $\sum_{t=1}^{T}\trace{\Theta_t \Sigma\att{t}{t}(\calS_1)}\geq \sum_{t=1}^{T}\trace{\Theta_t \Sigma\att{t}{t}(\calS_2)}$.  
\end{mylemma}

\renewcommand*\thesubsection{\thesection.\arabic{subsection}}

\section*{Appendix B: Proof of Theorem~\ref{th:LQG_closed}}

\noindent \textit{B.1.~Proof of part~(1) of Theorem~\ref{th:LQG_closed}}

\begin{mylemma}\label{lem:LQG_closed}
Consider any $\calS$, and let $\validated{u^\star_{1:T,\calS}}{u^\star_{1:T}(\calS)}$  be the vector of control policies $(K_1\hat{x}\of{1}{\calS}, K_2\hat{x}\of{2}{\calS}, \ldots, K_T\hat{x}\of{T}{\calS})$. 
Then $u^\star_{1:T}(\calS)$ is an optimal control policy:
\beq
u^\star_{1:T}(\calS) \in \argmin_{u_{1:T}(\calS)}h[\calS, u\of{1:T}{\calS}],\label{eq:lem:opt_control}
\eeq
and, particularly, $u^\star_{1:T}(\calS) $ attains an \LQG cost equal to:
\begin{equation}\label{eq:lem:opt_sensors}
h[\calS, u^\star_{1:T}(\calS) ] = 
\mathbb{E}(\|x_1\|_{N_1})+ \sum_{t=1}^T\left\{\text{tr}[\Theta_t\Sigma\att{t}{t}(\calS)]+\trace{W_tS_t}\right\}.
\end{equation}
\end{mylemma}
\begin{proof}[Proof of Lemma~\ref{lem:LQG_closed}] The proof follows Lemma~1's proof in~\cite{tanaka2015sdp}, and can also be found in~\cite[Appendix B]{tzoumas2018lqg}.
\end{proof}

\begin{proof}[Proof of part~(1) of Theorem~\ref{th:LQG_closed}]
Eq.~\eqref{eq:opt_sensors} is a direct consequence of eq.~\eqref{eq:lem:opt_sensors}, since the value of Problem~\ref{prob:LQG} is equal to $\min_{\calS\subseteq \calV, \sensorCost(\calS)\leq \sensorBudget} h[\calS, u^\star_{1:T}(\calS)]$, and both $\mathbb{E}(\|x_1\|_{N_1})=\trace{\initialCovariance N_1}$ and $\sum_{t=1}^T\trace{W_tS_t}$ are independent of $\calS$.  Finally, eq.~\eqref{eq:opt_control} directly follows from eq.~\eqref{eq:lem:opt_control}.  

\end{proof}

\bigskip

\noindent \textit{B.1.~Proof of part~(2) of Theorem~\ref{th:LQG_closed}}

\begin{mylemma}\label{lem:proof_part_2_th_perf} $\optS$, and $\optUoneT$ are a solution to Problem~\ref{prob:minCostLQG} if and only if they are a solution to
\begin{equation}
\label{eq:equiv_minCostForBoundedLQG}
\min_{\scriptsize \calS \subseteq \calV, u\of{1:T}{\calS}} \sensorCost(\calS)\validated{:}{, 
 \;\; \subjectTo}\;\;
\min_{\scriptsize u\of{1:T}{{\calS}}  } h\left[{\calS},u\of{1:T}{{\calS}}\right] \leq \kappa.
\end{equation}

\end{mylemma}
\begin{proof}[Proof of Lemma~\ref{lem:proof_part_2_th_perf}]
We prove the lemma \validated{, using the method of contradiction}{by contradiction}. 
Particularly, let $\optS$ and {$\optUoneT$} be a solution to Problem~\ref{prob:minCostLQG}, and assume 
{by contradiction} that they are not to eq.~\eqref{eq:equiv_minCostForBoundedLQG},
{which instead has solution $\widehat{\calS}$ and $\hatUoneT$. 
By optimality of $\widehat{\calS}$ and $\hatUoneT$ (and suboptimality of $\optS$ and $\optUoneT$) in eq.~\eqref{eq:equiv_minCostForBoundedLQG}, it follows 
$\sensorCost(\widehat{\calS})<\sensorCost(\optS)$.} In addition, $g(\widehat{\calS})\leq \kappa$, since $(\widehat{\calS},\validated{\widehat{u}_1, \widehat{u}_2, \ldots, \widehat{u}_T}{\hatUoneT})$ must be feasible for  eq.~\eqref{eq:equiv_minCostForBoundedLQG}.  However, the latter implies $h\left(\widehat{\calS},\validated{}{\hatUoneT}\right)\leq \kappa$. 
Therefore, $(\widehat{\calS},\validated{\widehat{u}_1, \widehat{u}_2, \ldots, \widehat{u}_T}{\hatUoneT})$ is feasible for Problem~\ref{prob:minCostLQG} \validated{.  But since $\sensorCost(\widehat{\calS})<\sensorCost(\optS)$, while $\optS$ is assumed to be a solution to Problem~\ref{prob:minCostLQG}, that is, it must be $\sensorCost(\optS)\leq \sensorCost(\widehat{\calS})$, we get a contraction.
}{and has a better objective value with respect to the optimal solution $(\optS,\optUoneT)$ (we already observed $\sensorCost(\widehat{\calS})<\sensorCost(\optS)$), leading to contradiction.}

For the other direction, now let  $\optS$ and  $\validated{u_1^\star, u_2^\star, \ldots, u_T^\star}{\optUoneT}$ be a solution to eq.~\eqref{eq:equiv_minCostForBoundedLQG}, and assume that they are not to Problem~\ref{prob:minCostLQG}, 
{which instead has solution $(\widehat{\calS},\hatUoneT)$. 
By optimality of $(\widehat{\calS},\hatUoneT)$ (and suboptimality of $\optS$ and $\optUoneT$) in Problem~\ref{prob:minCostLQG}, it follows 
$\sensorCost(\widehat{\calS})<\sensorCost(\optS)$.}  
  In addition, $h\left(\widehat{\calS},\validated{}{\hatUoneT}\right)\leq \kappa$, since $(\widehat{\calS},\validated{}{\hatUoneT})$ must be feasible for Problem~\ref{prob:minCostLQG}, and, as a result, $g(\widehat{\calS})\leq \kappa$.
Therefore, $(\widehat{\calS},\validated{\widehat{u}_1, \widehat{u}_2, \ldots, \widehat{u}_T}{\hatUoneT})$ is feasible for eq.~\eqref{eq:equiv_minCostForBoundedLQG} {and has a better objective value with respect to the optimal solution $(\optS,\optUoneT)$ (we already observed $\sensorCost(\widehat{\calS})<\sensorCost(\optS)$), leading to contradiction.}
\end{proof}

\begin{proof}[Proof of  part~(2) of \validated{Theorem~\ref{th:approx_bound}}{Theorem~\ref{th:LQG_closed}}]
The proof follows from Lemma~\ref{lem:LQG_closed} and Lemma~\ref{lem:proof_part_2_th_perf}.  
\omitted{Particularly, similarly to the proof of Theorem~\ref{th:LQG_closed}'s part~(1), Lemma~\ref{lem:LQG_closed}, along with eq.~\eqref{eq:lem:opt_sensors} and the fact that $\mathbb{E}(\|x_1\|_{N_1})=\trace{\initialCovariance N_1}$, implies that if the sensor set ${\calS}^\star$ and the controllers $\validated{u_1^\star, u_2^\star, \ldots, u_T^\star}{\optUoneT}$ are a solution to the optimization problem in eq.~\eqref{eq:equiv_minCostForBoundedLQG}, then ${\calS}^\star$ and the controllers $\validated{u_1^\star, u_2^\star, \ldots, u_T^\star}{\optUoneT}$ can be computed in cascade as follows: 
\begin{align}
{\calS}^\star &\in \argmin_{\scriptsize \calS \subseteq \calV} \sensorCost(\calS)\validated{:}{, 
 \;\; \subjectTo}\;\; \text{tr}[\Theta_t\Sigma\att{t}{t}({\calS})]\leq \nonumber\\
&\qquad\qquad\qquad\;\;\;\;\; \kappa-\trace{\initialCovariance N_1}-\sum_{t=1}^T\trace{W_tS_t}, \label{eq:opt_sensors_minimum_lem} \\
{u}_t^\star&=K_t\hat{x}\of{t}{{\calS}^\star}, \quad t=1,\ldots,T.\label{eq:opt_control_minimum_lem}
\end{align}
In addition, Lemma~\ref{lem:proof_part_2_th_perf} implies that $(\optS, \validated{u_1^\star,u_2^\star,$ $\ldots,u_T^\star}{\optUoneT})$ is a solution to Problem~\ref{prob:minCostLQG}. As a result, eqs.~\eqref{eq:opt_sensors_minimum}-\eqref{eq:opt_control_minimum} hold true.
}
\end{proof}

\section*{Appendix C: Proof of Theorem~\ref{th:hardness}}

Consider a problem instance for Problem~\ref{prob:LQG} and Problem~\ref{prob:minCostLQG}, where $T=1$, and $A_1=B_1=C_1=Q_1=R_1=I$.  Then, $\Theta_1=\eye/2$, and, as a result, the objective function in eq.~\eqref{eq:opt_sensors} becomes $1/2\text{tr}[\Sigma\att{1}{1}(\calS)]$.  Now, choosing $\Sigma\att{1}{1}(\calS)$ to be the steady state Kalman filtering matrix defined in \cite[Theorem~2]{ye2018complexity}, as well as, $c(\calS)$, $b$ be as in \cite[Theorem~2]{ye2018complexity}, makes eq.~\eqref{eq:opt_sensors} and the optimization problem in \cite{ye2018complexity} equivalent.  But, the latter is inapproximable in polynomial time \cite[Theorem~2]{ye2018complexity} (namely, unless \emph{NP}$=$\emph{P}, there is no polynomial time algorithm that guarantees a constant suboptimality bound).  Therefore, eq.~\eqref{eq:opt_sensors} is too, and due to Theorem~\ref{th:LQG_closed} both Problem~\ref{prob:LQG} and Problem~\ref{prob:minCostLQG} as well.

\section*{Appendix D: Proof of Theorem~\ref{th:approx_bound}}

For any~$\calS$, let $f(\calS)\triangleq\sum_{t=1}^T\text{tr}[\Theta_t\Sigma\att{t}{t}(\calS)]$ be the objective function in eq.~\eqref{eq:opt_sensors}, $\optS$ be a solution in eq.~\eqref{eq:opt_sensors}, and $\optBudget\triangleq\sensorCost(\optS)$.  Let $\algStwo$ be the set Algorithm~\ref{alg:greedy_nonUniformCosts} constructs by the end of line~\ref{line:end_if_condition_less}; let $\calG\triangleq\algStwo$. 
Let $s_i$ be the $i$-th element added in $\calG$ during the $i$-th iteration of Algorithm~\ref{alg:greedy_nonUniformCosts}'s ``while loop'' (lines~\ref{line:while_nonUniformCosts}-\ref{line:end_while_nonUniformCosts}).  Let  $\calG_i\triangleq\{s_1,s_2,\ldots, s_i\}$.
 Finally, consider Algorithm~\ref{alg:greedy_nonUniformCosts}'s ``while loop''  terminates after $l+1$ iterations.  

Algorithm~\ref{alg:greedy_nonUniformCosts}'s ``while loop'' terminates: (i)~when $\calV'=\emptyset$, that is, when all available sensors in $\calV$ can been chosen by Algorithm~\ref{alg:greedy_nonUniformCosts} as active while satisfying the budget constraint $\sensorBudget$; and (ii)~when $\sensorCost(\calG_{l+1})>\sensorBudget$, that is, when the addition of $s_{l+1}$ in~$\calG_l$ makes the cost of $\calG_{l+1}$ to  {{violate} the budget $\sensorBudget$.} Henceforth, we focus on the second scenario, which implies $s_{l+1}$ will be removed by the ``if'' statement in Algorithm~\ref{alg:greedy_nonUniformCosts}'s lines~\ref{line:if_condition_less}--\ref{line:end_if_condition_less} and, as a result, $\calG_l=\algStwo$.

\begin{mylemma}[Generalization of {\cite[Lemma~2]{krause2005note}}]\label{lem:gen_lemma_2}
For $i=1,2,\ldots, l+1$, it holds
$$f(\calG_{i-1})-f(\validated{G}{\calG}_{i})\geq \frac{\gamma_f\sensorCost(s_i)}{\optBudget} (f(\validated{G}{\calG}_{i-1})-f(\optS)).$$
\end{mylemma}
\begin{proof}[Proof of Lemma~\ref{lem:gen_lemma_2}]
Due to the monotonicity of the cost function $f$ in eq.~\eqref{eq:opt_sensors} (Lemma~\ref{prop:monotonicity}), 
\begin{align*}
f(\calG_{i-1})-f(\optS)&\leq f(\calG_{i-1})-f(\optS \cup \calG_{i-1})\\
&=f(\calG_{i-1})-f[(\optS\setminus \calG_{i-1}) \cup \calG_{i-1}].
\end{align*}
Let $\{z_1, z_2, \ldots, z_m\}\triangleq\optS\setminus \calG_{i-1}$, and  also let 
$$d_j\triangleq f(\calG_{i-1}\cup \{z_1, z_2, \ldots, z_{j-1}\})-f(\calG_{i-1}\cup \{z_1, z_2, \ldots, z_j\}),$$
for $j=1,2,\ldots,m$.  Then, $f(\calG_{i-1})-f(\optS)\leq \sum_{j=1}^{m}d_j$. Now,

\begin{align*}
\frac{d_j}{\sensorCost(z_j)}\leq \frac{f(\calG_{i-1})-f(\calG_{i-1}\cup \{z_j\})}{\gamma_f \sensorCost(z_j)}
\leq \frac{f(\calG_{i-1})-f(\calG_{i})}{\gamma_f\sensorCost(s_i)},
\end{align*}
where the first inequality holds due to the Definition~\ref{def:super_ratio} of $\gamma_f$, and the second  due to the greedy rule (Algorithm~\ref{alg:greedy_nonUniformCosts}'s line~\ref{line:best_a_nonUniformCosts}) and the definitions of $\calG_i$, and $s_i$.  Since $\sum_{j=1}^{m}\sensorCost(z_j)\leq \optBudget$,
\belowdisplayskip =-12pt$$f(\calG_{i-1})-f(\optS)\leq \sum_{j=1}^{m}d_j\leq \optBudget\frac{f(\calG_{i-1})-f(\calG_{i})}{\gamma_f\sensorCost(s_i)}.$$
\end{proof}

\begin{mylemma}[Adapation of {\cite[Lemma~3]{krause2005note}}]\label{lem:gen_lemma_3}
For $i=1,2,\ldots, l+1$, 
$$f(\emptyset)-f(G_{i})\geq \left[1-\prod_{j=1}^{i}\left(1-\frac{\gamma_f\sensorCost(s_j)}{\optBudget}\right)\right] [f(\emptyset)-f(\optS)].$$
\end{mylemma}
\begin{proof}[Proof of Lemma~\ref{lem:gen_lemma_3}]
We complete the proof inductively. For $i=1$, we need to prove $f(\emptyset)-f(G_{1})\geq \gamma_f\sensorCost(s_1)/\optBudget [f(\emptyset)-f(\optS)]$, which follows from Lemma~\ref{lem:gen_lemma_2} for $i=1$.  Then, 
for $i>1$,
\begin{align*}
f(\emptyset)-f(\calG_{i})&=f(\emptyset)-f(\calG_{i-1})+[f(\calG_{i-1})-f(\calG_{i})]\\
&\geq f(\emptyset)-f(\calG_{i-1})+\\
&\;\;\;\;\frac{\gamma_f\sensorCost(s_i)}{\optBudget} (f(G_{i-1})-f(\optS))\\
&=\left(1-\frac{\gamma_f\sensorCost(s_i)}{\optBudget}\right)[f(\emptyset)-f(\calG_{i-1}])+\\
&\;\;\;\;\frac{\gamma_f\sensorCost(s_i)}{\optBudget}[f(\emptyset)-f(\optS)]\\
&\geq \left(1-\frac{\gamma_f\sensorCost(s_i)}{\optBudget}\right) \left[1-\prod_{j=1}^{i-1}\left(1-\frac{\gamma_f\sensorCost(s_j)}{\optBudget}\right)\right]\\
&\;\;\;\; [f(\emptyset)-f(\optS)]+\frac{\gamma_f\sensorCost(s_i)}{\optBudget}[f(\emptyset)-f(\optS)]\\
&= \left[1-\prod_{j=1}^{i}\left(1-\frac{\gamma_f\sensorCost(s_j)}{\optBudget}\right)\right] [f(\emptyset)-f(\optS)],
\end{align*}
\validated{using}{where we used} Lemma~\ref{lem:gen_lemma_2} for the first inequality, and the induction hypothesis for the second.
\end{proof}

\begin{proof}[Proof of part (1) of Theorem~\ref{th:approx_bound}] 
To prove Algorithm~\ref{alg:overall_nonUniformCosts}'s approximation bound $\gamma_g/2\left(1-e^{-\gamma_g}\right)$, we let $\sensorBudget'\triangleq\validated{\sum_{j=1}^{l+1}=\sensorCost(s_i)}{\sum_{j=1}^{l+1}\sensorCost(s_j)}$. Then,
\begin{align}
f(\emptyset)-f(\calG_{l+1})&\geq\left[1-\prod_{j=1}^{l+1}\left(
1-\frac{\gamma_f\sensorCost(s_j)}{\optBudget}\right)\right] [f(\emptyset)-f(\optS)]\nonumber\\
& \geq  \left(1-e^{-\gamma_f\sensorBudget'/\optBudget}\right) [f(\emptyset)-f(\optS)],\nonumber\\
& \geq  \left(1-e^{-\gamma_f}\right) [f(\emptyset)-f(\optS)],\label{ineq:aux_1}
\end{align}
where the first inequality follows from Lemma~\ref{lem:gen_lemma_3}, the second from Lemma~\ref{lem:minimum_of_products}, and ineq.~\eqref{ineq:aux_1} from that $\sensorBudget'/\optBudget\geq 1$ and, as a result, $e^{-\gamma_f\sensorBudget'/\optBudget}\leq e^{-\gamma_f}$\!, that is, $1-e^{-\gamma_f\sensorBudget'/\optBudget}\geq 1-e^{-\gamma_f}$\!.

Also,  $f(\emptyset)-f(\algSone)\geq \gamma_f [f(\calG_{l})-f(\calG_{l+1})]$ due to the Definition~\ref{def:super_ratio} of $\gamma_g$ and, as a result,
\begin{align}
&\!\!\!\gamma_f [f(\emptyset)-f(\calG_{l+1})]\nonumber\\
&\leq f(\emptyset)-f(\algSone)+\gamma_f [f(\emptyset)-f(\calG_{l})]\nonumber\\
&\leq 2\max\left\{f(\emptyset)-f(\algSone), \gamma_f [f(\emptyset)-f(\calG_{l})]\right\}. \label{ineq:aux_2}
\end{align}

Substituting ineq.~\eqref{ineq:aux_1} in ineq.~\eqref{ineq:aux_2}, and rearranging, gives
\begin{align*}
&\!\!\!
\max\left\{f(\emptyset)-f(\algSone), \gamma_f [f(\emptyset)-f(\calG_{l})]\right\} \\
&\geq \frac{\gamma_f}{2}\left(1-e^{-\gamma_f}\right) [f(\emptyset)-f(\optS)]
, \end{align*}
which implies (since $\gamma_f$ takes values in $[0,1]$)
\begin{align}
&\!\!\!
\max\left[f(\emptyset)-f(\algSone), f(\emptyset)-f(\calG_{l})\right]\nonumber \\
&\geq \frac{\gamma_f}{2}\left(1-e^{-\gamma_f}\right) [f(\emptyset)-f(\optS)].\label{ineq:aux_3}
 \end{align}

Finally, the bound $\gamma_g/2\left(1-e^{-\gamma_g}\right)$ follows from ineq.~\eqref{ineq:aux_3} as the combination of the following three observations: 
i)  $\calG_l=\algStwo$, and, as a result, $f(\emptyset)-f(\calG_l) = f(\emptyset)-f(\algStwo)$.
ii) Algorithm~\ref{alg:greedy_nonUniformCosts} returns  $\algS$ such at 
$\algS\in \arg \max_{\calS\in \{\algSone,\algStwo\}}\left[f(\emptyset)-f(\calS)\right]$ and, as a result, the previous observation, along with ineq.~\eqref{ineq:aux_3}, gives:
\begin{equation}\label{ineq:aux_4}
f(\emptyset)-f(\algS)
\geq \frac{\gamma_f}{2}\left(1-e^{-\gamma_f}\right) [f(\emptyset)-f(\optS)].
 \end{equation}
iii) Finally, Lemma~\ref{lem:LQG_closed} implies that for any  $\calS,\calS'$, $g(\calS)=f(\calS)+\mathbb{E}(\|x_1\|_{N_1})+\sum_{t=1}^{T}\trace{W_tS_t}$, where $\mathbb{E}(\|x_1\|_{N_1})+\sum_{t=1}^{T}\trace{W_tS_t}$ is independent of~$\calS$.  As a result, for any  $\calS,\calS'\subseteq \calV$, then $f(\calS)-f(\calS')=g(\calS)-g(\calS')$, which implies $\gamma_f=\gamma_g$ due to  Definition~\ref{def:super_ratio}.  In addition, Lemma~\ref{lem:LQG_closed} implies  for any $\calS\subseteq \calV$ that $g(\calS)=h[\calS, u\of{1:T}{\validated{\algS}{\calS}}]$ and $g^\star=g(\optS)$.  Thereby,  for any $\calS$ that $f(\emptyset)-f(\calS)=g(\emptyset)-g(\calS)=h[\emptyset, u\of{1:T}{\emptyset}]-h[\calS, u\of{1:T}{\validated{\algS}{\calS}}]$ and $f(\emptyset)-f(\optS)=g(\emptyset)-g(\optS)=h[\emptyset, u\of{1:T}{\emptyset}]-g^\star$\!\!.  Overall, ineq.~\eqref{ineq:aux_4} is written as
\begin{align*}
&h[\emptyset, u\of{1:T}{\emptyset}]-h[\algS, u\of{1:T}{\validated{\algS}{\algS}}]
\geq \\
&\qquad\qquad\qquad\frac{\gamma_f}{2}\left(1-e^{-\gamma_f}\right) \left\{h[\emptyset, u\of{1:T}{\emptyset}]-g^\star\right\}
, \end{align*}
which implies the bound $\gamma_g/2\left(1-e^{-\gamma_g}\right)$.

It remains to prove $1-e^{-\gamma_g\sensorCost(\algS)/\sensorBudget}$:
\begin{align}
f(\emptyset)-f(\calG_{l})&\geq\left[1-\prod_{j=1}^{l}\left(
1-\frac{\gamma_f\sensorCost(s_j)}{\optBudget}\right)\right] [f(\emptyset)-f(\calG_l)]\nonumber\\
& \geq  \left(1-e^{-\gamma_f\sensorCost(\calG_l)/\optBudget}\right) [f(\emptyset)-f(\optS)],\nonumber\\
& \geq  \left(1-e^{-\gamma_f\sensorCost(\calG_l)/\sensorBudget}\right) [f(\emptyset)-f(\optS)],\label{ineq:aux_1111}
\end{align}
where the first inequality follows from Lemma~\ref{lem:gen_lemma_3}, the second from Lemma~\ref{lem:minimum_of_products}, and ineq.~\eqref{ineq:aux_1111} from that $\sensorCost(\calG_l)/\optBudget\geq \sensorCost(\calG_l)/\sensorBudget$, since $\optBudget\leq \sensorBudget$, which implies $e^{-\gamma_f\sensorCost(\calG_l)/\optBudget}\leq e^{-\gamma_f\sensorCost(\calG_l)/\sensorBudget}$, i.e., $1-e^{-\gamma_f\sensorBudget'/\optBudget}\geq 1-e^{-\gamma_f\sensorCost(\calG_l)/\sensorBudget}$.  The proof is completed using the observations (i)-(iii) above for $\gamma_g/2\left(1-e^{-\gamma_g}\right)$.
\end{proof}

\begin{proof}[{Proof of part (2) of Theorem~\ref{th:approx_bound}}]	
The proof is parallel to that of Theorem~2 in~\cite{tzoumas2016near}.
\end{proof}
\section*{Appendix E: Proof of Theorem~\ref{th:performance_LQGconstr}}

We consider the notation in Appendix~D.
Also, let $\optS$ be a solution to Problem~\ref{prob:minCostLQG}, and $\optBudget=\sensorCost(\optS)$.
Consider the computation of the set $\algS$ in Algorithm~\ref{algLQGconstr:greedy_nonUniformCosts}, and let $\calG\triangleq\algS$ be the returned one.  Let $s_i$ be the $i$-th element added in $\calG$ during the $i$-th iteration of Algorithm~\ref{algLQGconstr:greedy_nonUniformCosts}'s ``while loop.''  Finally,  let  $\calG_i\triangleq\{s_1,s_2,\ldots, s_i\}$.


\begin{mylemma}[Adaptation of Lemma~\ref{lem:gen_lemma_2}]\label{lemLQGconstr:gen_lemma_2}
For $i=1,2,\ldots, |\calG|$,
$$f(\calG_{i-1})-f(G_{i})\geq \frac{\gamma_f\sensorCost(s_i)}{\optBudget} (f(G_{i-1})-f(\optS)).$$
\end{mylemma}
\begin{proof}
	The proof is parallel to Lemma~\ref{lem:gen_lemma_2}'s proof.
\omitted{Due to the monotonicity of the cost function $f$ in eq.~\eqref{eq:opt_sensors} (Lemma~\ref{prop:monotonicity}),
\begin{align*}
f(\calG_{i-1})-f(\optS)&\leq f(\calG_{i-1})-f(\optS \cup \calG_{i-1})\\
&=f(\calG_{i-1})-f[(\optS\setminus \calG_{i-1}) \cup \calG_{i-1}].
\end{align*}

Let $\{z_1, z_2, \ldots, z_m\}\triangleq\optS\setminus \calG_{i-1}$, and  also let 
$$d_j\triangleq f(\calG_{i-1}\cup \{z_1, z_2, \ldots, z_{j-1}\})-f(\calG_{i-1}\cup \{z_1, z_2, \ldots, z_j\}),$$
for $j=1,2,\ldots,m$.  Then, $f(\calG_{i-1})-f(\optS)\leq \sum_{j=1}^{m}d_j$.

Notice that
\begin{align*}
\frac{d_j}{\sensorCost(z_j)}\leq \frac{f(\calG_{i-1})-f(\calG_{i-1}\cup \{z_j\})}{\gamma_f \sensorCost(z_j)}
\leq \frac{f(\calG_{i-1})-f(\calG_{i})}{\gamma_f\sensorCost(s_i)},
\end{align*}
where the first inequality holds due to the Definition~\ref{def:super_ratio} of the supermodularity ratio $\gamma_f$, and the second inequality holds due to the greedy rule (Algorithm~\ref{algLQGconstr:greedy_nonUniformCosts}'s line~\ref{lineLQGconstr:best_a_nonUniformCosts}) and the definitions of $\calG_i$ and $s_i$.  Since $\sum_{j=1}^{m}\sensorCost(z_j)\leq \optBudget$, it holds that
\belowdisplayskip =-12pt$$f(\calG_{i-1})-f(\optS)\leq \sum_{j=1}^{m}d_j\leq \optBudget\frac{f(\calG_{i-1})-f(\calG_{i})}{\gamma_f\sensorCost(s_i)}.$$}
\end{proof}

\begin{mylemma}[Adaptation of Lemma~\ref{lem:gen_lemma_3}]\label{lemLQGconstr:gen_lemma_3}
For $i=1,2,\ldots, |\calG|$,
$$f(\emptyset)-f(G_{i})\geq \left[1-\prod_{j=1}^{i}\left(1-\frac{\gamma_f\sensorCost(s_j)}{\optBudget}\right)\right] [f(\emptyset)-f(\optS)].$$
\end{mylemma}
\begin{proof}
		The proof is parallel to Lemma~\ref{lem:gen_lemma_3}'s proof.
\omitted{We complete the proof inductively.  Particularly, for $i=1$, we need to prove $f(\emptyset)-f(G_{1})\geq \gamma_f\sensorCost(s_1)/\optBudget [f(\emptyset)-f(\optS)]$, which follows from Lemma~\ref{lemLQGconstr:gen_lemma_2} for $i=1$.  Then, 
we have for $i>1$:
\begin{align*}
f(\emptyset)-f(\calG_{i})&=f(\emptyset)-f(\calG_{i-1})+[f(\calG_{i-1})-f(\calG_{i})]\\
&\geq f(\emptyset)-f(\calG_{i-1})+\\
&\;\;\;\;\frac{\gamma_f\sensorCost(s_i)}{\optBudget} (f(G_{i-1})-f(\optS))\\
&=\left(1-\frac{\gamma_f\sensorCost(s_i)}{\optBudget}\right)[f(\emptyset)-f(\calG_{i-1}])+\\
&\;\;\;\;\frac{\gamma_f\sensorCost(s_i)}{\optBudget}[f(\emptyset)-f(\optS)]\\
&\geq \left(1-\frac{\gamma_f\sensorCost(s_i)}{\optBudget}\right) \left[1-\prod_{j=1}^{i-1}\left(1-\frac{\gamma_f\sensorCost(s_j)}{\optBudget}\right)\right]\\
&\;\;\;\; [f(\emptyset)-f(\optS)]+\frac{\gamma_f\sensorCost(s_i)}{\optBudget}[f(\emptyset)-f(\optS)]\\
&= \left[1-\prod_{j=1}^{i}\left(1-\frac{\gamma_f\sensorCost(s_j)}{\optBudget}\right)\right] [f(\emptyset)-f(\optS)],
\end{align*}
using Lemma~\ref{lemLQGconstr:gen_lemma_2} for the first inequality and the induction hypothesis for the second inequality.
}
\end{proof}

\begin{proof}[Proof of part (1) of Theorem~\ref{th:performance_LQGconstr}] We first {observe}  ineq.~\eqref{ineq:approx_boundLQGconstr} holds since Algorithm~\ref{algLQGconstr:overall_nonUniformCosts} returns $\algS$ once $h[\algS, u\of{1:T}{\algS}]\leq \kappa$ is satisfied.  

It remains to prove ineq.~\eqref{ineq:cost_approx_bound}. Let $l\triangleq |\calG|$; then, $\calG_l=\calG$, by the definition of $\calG_i$, and from Lemma~\ref{lem:gen_lemma_3} for $i=l-1$,
\begin{align}
f(\emptyset)-f(\calG_{l-1})&\geq\left[1-\prod_{j=1}^{l-1}\left(
1-\frac{\gamma_f\sensorCost(s_j)}{\optBudget}\right)\right] [f(\emptyset)-f(\optS)]\nonumber\\
& \geq  \left(1-e^{-\gamma_f\sensorCost(\calG_{l-1})/\optBudget}\right) [f(\emptyset)-f(\optS)],\label{ineq:aux_11111}
\end{align}
where ineq.~\eqref{ineq:aux_11111} follows from Lemma~\ref{lem:minimum_of_products}.  Moreover, Lemma~\ref{lem:LQG_closed} implies that for any  $\calS,\calS'$, it is $g(\calS)=f(\calS)+\mathbb{E}(\|x_1\|_{N_1})+\sum_{t=1}^{T}\trace{W_tS_t}$, where   $\mathbb{E}(\|x_1\|_{N_1})+\sum_{t=1}^{T}\trace{W_tS_t}$ is independent of~$\calS$, and, as a result, $f(\calS)-f(\calS')=g(\calS)-g(\calS')$, which implies $\gamma_f=\gamma_g$.  Moreover, Lemma~\ref{lem:LQG_closed} implies for any $\calS\subseteq$ that $g(\calS)=h[\calS, u\of{1:T}{\validated{\algS}{\calS}}]$, and, as a result, $f(\emptyset)-f(\calG_{l-1})=h[\emptyset, u\of{1:T}{\emptyset}]-h[\calG_{l-1}, u\of{1:T}{\calG_{l-1}}]$ and $f(\emptyset)-f(\optS)=h[\emptyset, u\of{1:T}{\emptyset}]-h[\optS, u\of{1:T}{\optS}]$.  In sum, ineq.~\eqref{ineq:aux_11111} is the same as the inequality 
\begin{align*}
& h[\emptyset, u\of{1:T}{\emptyset}]-h[\calG_{l-1}, u\of{1:T}{\calG_{l-1}}]\geq \\ &\;\;\;\left(1-e^{-\gamma_g\sensorCost(\calG_{l-1})/\optBudget}\right) \left\{h[\emptyset, u\of{1:T}{\emptyset}]-h[\optS, u\of{1:T}{\optS}]\right\},
\end{align*}
which, by letting $\temp\triangleq 1-e^{-\gamma_g\sensorCost(\calG_{l-1})/\optBudget}$ and  rearranging, gives
\begin{align}
h[\calG_{l-1}, u\of{1:T}{\calG_{l-1}}]&\leq (1-\temp)h[\emptyset,  u\of{1:T}{\emptyset}]+\temp h[\optS\!\!, u\of{1:T}{\optS}]\nonumber\\
&\leq (1-\temp)h[\emptyset,  u\of{1:T}{\emptyset}]+\temp\kappa,
\label{ineq:aux_14}
\end{align}
where the second inequality holds because $\optS$ is a solution to Problem~\ref{prob:minCostLQG} and, as result, $h[\optS\!\!, u\of{1:T}{\optS}]\leq \kappa$.  Now, we recall Algorithm~\ref{algLQGconstr:greedy_nonUniformCosts} returns $\calG=\calG_l$ when for $i=l$ it is the first time $h[\calG_i, u\of{1:T}{\calG_i}]\leq \kappa$. Therefore, $h[\calG_{l-1}, u\of{1:T}{\calG_{l-1}}]>\kappa$ and, as a result, \validated{for the}{there exists} $\epsilon>0$ such that $h[\calG_{l-1}, u\of{1:T}{\calG_{l-1}}]=(1+\epsilon)\kappa$, \validated{}{and} ineq.~\eqref{ineq:aux_14} gives
\begin{align*}
&(1+\epsilon)\kappa\leq (1-\temp)h[\emptyset,  u\of{1:T}{\emptyset}]+\temp\kappa\Rightarrow\nonumber\\
&\epsilon\kappa\leq (1-\temp)h[\emptyset,  u\of{1:T}{\emptyset}]-(1-\temp)\kappa \Rightarrow\nonumber\\
&\epsilon\kappa\leq (1-\temp)\{h[\emptyset,  u\of{1:T}{\emptyset}]-\kappa\} \Rightarrow\nonumber\\
&\epsilon\kappa\leq e^{-\gamma_g\sensorCost(\calG_{l-1})/\optBudget}\{h[\emptyset,  u\of{1:T}{\emptyset}]-\kappa\} \Rightarrow\nonumber\\
&\log\left(\frac{\epsilon\kappa}{h[\emptyset,  u\of{1:T}{\emptyset}]-\kappa}\right)\leq -\gamma_g\sensorCost(\calG_{l-1})/\optBudget \Rightarrow\nonumber\\
&\sensorCost(\calG_{l-1})\leq\frac{1}{\gamma_g} \log\left(\frac{h[\emptyset,  u\of{1:T}{\emptyset}]-\kappa}{\epsilon\kappa}\right)\optBudget \Rightarrow\nonumber\\
&\sensorCost(\calG)\leq c(s_l)+\frac{1}{\gamma_g} \log\left(\frac{h[\emptyset,  u\of{1:T}{\emptyset}]-\kappa}{\epsilon\kappa}\right)\optBudget\!\!,
\end{align*}
where the latter holds since $\calG=\calG_{l-1}\cup \{s_l\}$, due to the definitions of $\calG$, $\calG_{l-1}$, and $s_l$, and since $\sensorCost(\calG)=\sensorCost(\calG_{l-1})+\sensorCost(s_l)$.  Finally, since the definition of $\epsilon$ implies $\epsilon\kappa=h[\calG_{l-1}, u\of{1:T}{\calG_{l-1}}]-\kappa$, and the definition of $\calG$ is $\calG=\algS$, the proof of ineq.~\eqref{ineq:approx_boundLQGconstr} is complete.
\end{proof}

\begin{proof}[Proof of part (2) of Theorem~\ref{th:performance_LQGconstr}]
The proof is similar to the proof of part (2) of Theorem~\ref{th:approx_bound}.
\end{proof}

\section*{Appendix F: Proof of Theorem~\ref{th:submod_ratio}}

We complete the proof by first deriving a lower bound for the numerator of $\gamma_g$, and then, by deriving an upper bound for the denominator $\gamma_g$.  We use the following notation: $c\triangleq\mathbb{E}(x_1\tran N_1 x_1)+\sum_{t=1}^T\trace{W_tS_t}$,
and for any  $\calS$, and time $\allT$, $f_t(\calS)\triangleq\trace{\Theta_t \Sigma\att{t}{t}(\calS)}$.  Then, $g(\calS)=c+\sum_{t=1}^T f_t(\calS),$
due to eq.~\eqref{eq:lem:opt_sensors} in Lemma~\ref{lem:LQG_closed}.

\setcounter{paragraph}{0}
\paragraph{Lower bound for the numerator of  $\gamma_g$} The numerator of $\gamma_g$ has the form
$\sum_{t=1}^{T}[f_t(\calS)-f_t(\calS\cup \{v\})]$,
for some  $\calS$, and $v \in \calV$. We now lower bound each $f_t(\calS)-f_t(\calS\cup \{v\})$: from~eq.~\eqref{eq:covariance_riccati} in Lemma~\ref{lem:covariance_riccati}, observe
\begin{equation*}
\textstyle\Sigma\att{t}{t}(\calS\cup \{v\})=[\Sigma\att{t}{t-1}\inv(\calS \cup \{v\})+\sum_{i \in \calS \cup \{v\}} \bar{C}_{i,t}\tran \bar{C}_{i,t}]\inv\!.
\end{equation*}
Define $\Omega_t=\Sigma\att{t}{t-1}\inv(\calS)+\sum_{i \in \calS}^T \bar{C}_{i,t}\tran \bar{C}_{i,t}$, and $\bar{\Omega}_{t}=\Sigma\att{t}{t-1}\inv(\calS\cup \{v\})+\sum_{i \in \calS}^T \bar{C}_{i,t}\tran \bar{C}_{i,t}$; using Lemma~\ref{lem:woodbury}, 
\begin{align*}
f_t(\calS\cup \{v\})&=\trace{\Theta_t\bar{\Omega}_{t}\inv}-\\
&\hspace{0.4em}\trace{\Theta_t\bar{\Omega}_{t}\inv \bar{C}_{v,t}\tran (I+\bar{C}_{v,t} \bar{\Omega}_{t}\inv \bar{C}_{v,t}\tran)\inv \bar{C}_{v,t} \bar{\Omega}_{t}\inv}.
\end{align*}
Therefore, for any time $t \in \until{T}$,
\begin{align}
& f_t(\calS)-f_t(\calS\cup \{v\})=
\nonumber\\
&\trace{\Theta_t\Omega_{t}\inv}-\trace{\Theta_t\bar{\Omega}_{t}\inv}+\nonumber\\
&\trace{\Theta_t\bar{\Omega}_{t}\inv \bar{C}_{v,t}\tran (I+\bar{C}_{v,t} \bar{\Omega}_{t}\inv \bar{C}_{v,t}\tran)\inv \bar{C}_{v,t} \bar{\Omega}_{t}\inv}\geq\nonumber\\
& \trace{\Theta_t\bar{\Omega}_{t}\inv \bar{C}_{v,t}\tran (I+\bar{C}_{v,t} \bar{\Omega}_{t}\inv \bar{C}_{v,t}\tran)\inv \bar{C}_{v,t} \bar{\Omega}_{t}\inv},\label{ineq:super_ratio_aux_1}
\end{align}
where ineq.~\eqref{ineq:super_ratio_aux_1} holds because $\trace{\Theta_t\Omega_{t}\inv}\geq \trace{\Theta_t\bar{\Omega}_{t}\inv}$.  In~particular, $\trace{\Theta_t\Omega_{t}\inv}\geq \trace{\Theta_t\bar{\Omega}_{t}\inv}$ is implied as follows:  
Lemma~\ref{prop:one-step_monotonicity} implies $\Sigma\att{1}{1}(\calS)\succeq\Sigma\att{1}{1}(\calS\cup\{v\})$.  Then, Corollary~\ref{cor:from_t_to_t+1} implies $ \Sigma\att{t}{t-1}(\calS)\succeq \Sigma\att{t}{t-1}(\calS\cup \{v\})$, and as a result, Lemma~\ref{lem:inverse} implies $\Sigma\att{t}{t-1}(\calS)\inv\preceq\Sigma\att{t}{t-1}(\calS\cup\{v\})\inv\!$. Now, $\Sigma\att{t}{t-1}(\calS)\inv\preceq\Sigma\att{t}{t-1}(\calS\cup\{v\})\inv$ and the definition of $\Omega_t$ and~of~$\bar{\Omega}_t$ imply $\Omega_{t}\preceq \bar{\Omega}_{t}$.  Next, Lemma~\ref{lem:inverse} implies  $\Omega_{t}\inv\succeq \bar{\Omega}_{t}\inv$\!\!.  As a result, since also $\Theta_t$ is a symmetric matrix, Lem- \mbox{ma~\ref{lem:traceAB_mon} gives the desired inequality $\trace{\Theta_t\Omega_{t}\inv}\geq \trace{\Theta_t\bar{\Omega}_{t}\inv}$.} 

Continuing from the ineq.~\eqref{ineq:super_ratio_aux_1},
\begin{align}
& f_t(\calS)-f_t(\calS\cup \{v\})\geq\nonumber \\
&\trace{\bar{C}_{v,t} \bar{\Omega}_{t}\inv \Theta_t\bar{\Omega}_{t}\inv \bar{C}_{v,t}\tran (I+\bar{C}_{v,t} \bar{\Omega}_{t}\inv \bar{C}_{v,t}\tran)\inv }\geq \nonumber\\
&\lambda_\min((I+\bar{C}_{v,t} \bar{\Omega}_{t}\inv \bar{C}_{v,t}\tran)\inv)\trace{\bar{C}_{v,t} \bar{\Omega}_{t}\inv \Theta_t\bar{\Omega}_{t}\inv \bar{C}_{v,t}\tran },\label{ineq:super_ratio_aux_2}
\end{align}
where ineq.~\eqref{ineq:super_ratio_aux_2} holds due to Lemma~\ref{lem:trace_low_bound_lambda_min}.  From ineq.~\eqref{ineq:super_ratio_aux_2},
\hspace*{-3mm}\begin{align}
& f_t(\calS)-f_t(\calS\cup \{v\})\geq\nonumber \\
&= \lambda_\max\inv(I+\bar{C}_{v,t} \bar{\Omega}_{t}\inv \bar{C}_{v,t}\tran)\trace{\bar{C}_{v,t} \bar{\Omega}_{t}\inv \Theta_t\bar{\Omega}_{t}\inv \bar{C}_{v,t}\tran }\nonumber\\
&\geq \lambda_\max\inv(I+\bar{C}_{v,t} \Sigma\att{t}{t}(\emptyset) \bar{C}_{v,t}\tran)\trace{\bar{C}_{v,t} \bar{\Omega}_{t}\inv \Theta_t\bar{\Omega}_{t}\inv \bar{C}_{v,t}\tran }\nonumber\\
&= \lambda_\max\inv(I+\bar{C}_{v,t} \Sigma\att{t}{t}(\emptyset) \bar{C}_{v,t}\tran)\trace{\Theta_t \bar{\Omega}_{t}\inv \bar{C}_{v,t}\tran \bar{C}_{v,t}\bar{\Omega}_{t}\inv},\label{ineq:super_ratio_aux_3}
\end{align}
where we used $\bar{\Omega}_{t}\inv\preceq\Sigma\att{t}{t}(\emptyset)$, which holds since: $\bar{\Omega}_t$ implies  $\bar{\Omega}_{t}\succeq \Sigma\inv\att{t}{t-1}(\calS\cup \{v\})$, and as a result, from Lemma~\ref{lem:inverse} $\bar{\Omega}_{t}\inv \preceq \Sigma\att{t}{t-1}(\calS\cup \{v\})$.  In addition, Corollary~\ref{cor:from_t_to_t+1} and $\Sigma\att{1}{1}(\calS\cup \{v\})\preceq \Sigma\att{1}{1}(\emptyset)$, which holds due to Lemma~\ref{prop:one-step_monotonicity}, imply $\Sigma\att{t}{t-1}(\calS\cup \{v\})\preceq \Sigma\att{t}{t-1}(\emptyset)$.  Finally, from eq.~\eqref{eq:covariance_riccati} in Lemma~\ref{lem:covariance_riccati}, $\Sigma\att{t}{t-1}(\emptyset)=\Sigma\att{t}{t}(\emptyset)$. Overall,  $\bar{\Omega}_{t}\inv\preceq\Sigma\att{t}{t}(\emptyset)$.

Consider $t' \in \until{T}$ such that $\bar{\Omega}_{t'}\inv \bar{C}_{v,t'}\tran \bar{C}_{v,t'}\bar{\Omega}_{t'}\inv  \preceq  \bar{\Omega}_{t}\inv \bar{C}_{v,t}\tran \bar{C}_{v,t} \bar{\Omega}_{t}\inv$, for any $t=1,\ldots,T$. Also, let $\Phi\triangleq\bar{\Omega}_{t'}\inv \bar{C}_{v,t'}\tran \bar{C}_{v,t'} \bar{\Omega}_{t'}\inv$, and $l\triangleq\min_{t\in\until{T},v \in \calV}\lambda_\max\inv(I+\bar{C}_{v,t} \Sigma\att{t}{t}(\emptyset) \bar{C}_{v,t}\tran)$. 
Summing ineq.~\eqref{ineq:super_ratio_aux_3} across all $t \in \until{T}$, and using Lemmata~\ref{lem:trace_low_bound_lambda_min} and~\ref{lem:traceAB_mon},
\begin{align*}
\textstyle g(\calS)-g(\calS\cup\{v\})&\textstyle\geq  l\sum_{t=1}^T \trace{\Theta_t\bar{\Omega}_{t}\inv \bar{C}_{v,t}\tran \bar{C}_{v,t} \bar{\Omega}_{t}\inv }\\
\textstyle& \textstyle\geq l \lambda_\min\left(\sum_{t=1}^T \Theta_t \right)\trace{\Phi}>0,
\end{align*}
which is non-zero because $\sum_{t=1}^T \Theta_t\succ 0$ and $\Phi$ is a non-zero positive semi-definite matrix. 

Finally, we lower bound $\trace{\Phi}$, using Lemma~\ref{lem:trace_low_bound_lambda_min}:
\begin{align}
\textstyle\trace{\Phi}&=\trace{\bar{\Omega}_{t'}\inv \bar{C}_{v,t'}\tran \bar{C}_{v,t'} \bar{\Omega}_{t'}\inv }\nonumber\\
\textstyle&\geq \lambda_\min(\bar{\Omega}_{t'}^{-2}) \trace{\bar{C}_{v,t'}\tran\bar{C}_{v,t'}}\nonumber\\
\textstyle&\geq \lambda_\min^2(\Sigma\att{t'}{t'}(\calV)) \trace{\bar{C}_{v,t'}\tran\bar{C}_{v,t'}},\label{ineq:super_ratio_aux_10}
\end{align}
where ineq.~\eqref{ineq:super_ratio_aux_10} holds because $\bar{\Omega}_{t'}^{-1}\succeq \Sigma\att{t'}{t'}(\calV)$.  Particularly,  $\bar{\Omega}_{t'}^{-1}\succeq \Sigma\att{t'}{t'}(\calS\cup \{v\})$ is derived by applying Lemma~\ref{lem:inverse} to $\bar{\Omega}_{t'}\preceq \bar{\Omega}_{t'}+\bar{C}_{v,t}\tran\bar{C}_{v,t}\tran=\Sigma\inv\att{t'}{t'}(\calS\cup\{v\})$, where the equality holds by the definition of $\bar{\Omega}_{t'}$.  In~addition, due to Lemma~\ref{prop:one-step_monotonicity}, $\Sigma\att{1}{1}(\calS\cup \{v\})\succeq \Sigma\att{1}{1}(\calV)$, and as a result, from Corollary~\ref{cor:from_t_to_t}, $\Sigma\att{t'}{t'}(\calS\cup \{v\})\succeq \Sigma\att{t'}{t'}(\calV)$.  Overall, $\bar{\Omega}_{t'}^{-1}\succeq \Sigma\att{t'}{t'}(\calV)$ holds.

\paragraph{Upper bound for the denominator of $\gamma_g$} The proof follows similar ideas as above, and is omitted (for a complete proof, see~\cite[Proof of Theorem~\ref{th:submod_ratio}]{tzoumas2018lqg}).

\section*{Appendix G: Proof of Theorem~\ref{th:freq}}
\label{sec:contribution}


\begin{mylemma}[System-level condition for near-optimal co-design]\label{prop:iff_for_all_zero_control}
Let $N_1$ be defined as in eq.~\eqref{eq:control_riccati}.  The control policy $u^{\circ}_{1:T}\triangleq(0,0,\ldots, 0)$ is suboptimal for the \LQG problem in eq.~\eqref{pr:perfect_state} for all non-zero initial conditions~$x_1$
if and only if 
\beq\label{ineq:iff_for_all_zero_control}
\textstyle \sum_{t=1}^{T} A_1\tran \cdots A_t\tran Q_t A_t\cdots A_1\succ N_1.
\eeq
\end{mylemma}
\begin{proof}[Proof of Lemma~\ref{prop:iff_for_all_zero_control}]
For any $x_1$, eq.~\eqref{eq:lem:opt_sensors} in Lemma~\ref{lem:LQG_closed} implies for eq.~\eqref{pr:perfect_state}:
\begin{equation}\label{eq:it_always_aux_1}
\textstyle\min_{u_{1:T}}\sum_{t=1}^{T}\left.[\|x\at{t+1}\|^2_{Q_t} +\|u_{t}(x_t)\|^2_{R_t}]\right|_{\Sigma\att{t}{t}=W_t=0}=x_1\tran N_1 x_1,
\end{equation}
since $\mathbb{E}(\|x_1\|^2_{N_1})=x_1\tran N_1x_1$, because $x_1$ is known ($\Sigma\att{1}{1}=0$), and $\Sigma\att{t}{t}$ and $W_t$ are zero.
In addition, for $u_{1:T}=(0,0,\ldots,0)$, the objective function in eq.~\eqref{pr:perfect_state} is 
\beal\label{eq:it_always_aux_2}
\textstyle&\sum_{t=1}^{T}\left.[\|x\at{t+1}\|^2_{Q_t} +\|u_{t}(x_t)\|^2_{R_t}]\right|_{\Sigma\att{t}{t}=W_t=0}\\
\textstyle&= \sum_{t=1}^{T} x\at{t+1}\tran Q_t x\at{t+1}\\ 
\textstyle&=x_1\tran\sum_{t=1}^{T} A_1\tran A_2\tran \cdots A_t\tran Q_t A_tA_{t-1}\cdots A_1 x_1,
\eeal
since $x_{t+1}=A_tA_{t-1}\cdots A_1x_1$ when all $u_1, \ldots, u_T$ are zero.  

From eqs.~\eqref{eq:it_always_aux_1} and~\eqref{eq:it_always_aux_2}, we have that $x_1\tran N_1 x_1 < x_1\tran\sum_{t=1}^{T} A_1\tran A_2\tran \cdots A_t\tran Q_t A_tA_{t-1}\cdots A_1 x_1$ holds for any non-zero $x_1$ if and only if
$N_1 \prec \sum_{t=1}^{T} A_1\tran \cdots A_t\tran Q_t A_tA_{t-1}\cdots A_1.$
\end{proof}

\begin{mylemma}\label{lem:theta_formula}
$\Theta_t=A_t\tran S_t A_t+Q_{t-1}-S_{t-1}$, for $t=1,\ldots,T$.
\end{mylemma}
\begin{proof}[Proof of Lemma~\ref{lem:theta_formula}]
Using the Woobury identity in Lemma~\ref{lem:woodbury}, and the notation in eq.~\eqref{eq:control_riccati},
$N_t = A_t\tran (S_t\inv+B_tR_t\inv B_t\tran)\inv A_t 
=A_t\tran S_t A_t-\Theta_t.$
The latter, gives $\Theta_t=A_t\tran S_t A_t-N_t$.  In addition, from eq.~\eqref{eq:control_riccati},  $-N_t=Q_{t-1}-S_{t-1}$, since $S_t=Q_t+N_{t+1}$.
\end{proof}

\begin{mylemma}\label{lem:observability_condition}
$\sum_{t=1}^{T} A_1\tran A_2\tran \cdots A_t\tran Q_t A_tA_{t-1}\cdots A_1\succ N_1$ if and only if 
$\sum_{t=1}^{T}A\tran_{1}A\tran_2\cdots A\tran_{t-1} \Theta_t A_{t-1}A_{t-2}\cdots A_1 \succ 0.$

\end{mylemma}
\begin{proof}[Proof]
For $i= t-1, \ldots, 1$, we pre- and post-multiply the identity in Lemma~\ref{lem:theta_formula} with $A_i\tran$ and $A_i$, respectively:
\belowdisplayskip=-12pt\beal\nonumber
&\Theta_t=A_t\tran S_t A_t+Q_{t-1}-S_{t-1}\Rightarrow\\
& A_{t-1}\tran\Theta_t A_{t-1}=A_{t-1}\tran A_t\tran S_t A_t A_{t-1}+A_{t-1}\tran Q_{t-1}A_{t-1}-\\
& \qquad A_{t-1}\tran S_{t-1}A_{t-1}\Rightarrow\\
& A_{t-1}\tran\Theta_t A_{t-1}=A_{t-1}\tran A_t\tran S_t A_t A_{t-1}+A_{t-1}\tran Q_{t-1}A_{t-1}-\\
& \qquad \Theta_{t-1}+Q_{t-2}-S_{t-2}\Rightarrow\\
& \Theta_{t-1}+ A_{t-1}\tran\Theta_t A_{t-1}=A_{t-1}\tran A_t\tran S_t A_t A_{t-1}+\\
&\qquad A_{t-1}\tran Q_{t-1}A_{t-1}+Q_{t-2}-S_{t-2}\Rightarrow\\
&\ldots\Rightarrow\\
&\qquad \ldots+ A_2\tran Q_2 A_2+ Q_{1}-S_{1}\Rightarrow\\
& \Theta_1+ A_1\tran\Theta_2 A_1+\ldots+A_1\tran \cdots A_{t-1}\tran\Theta_t A_{t-1}\cdots A_1=\\
& A_1\tran \cdots A_t\tran S_t A_t \cdots A_1+A_1\tran \cdots A_{t-1}\tran Q_{t-1}A_{t-1}\cdots A_1+\\
&\qquad\ldots+ A_1\tran Q_1 A_1-N_1\Rightarrow\\
&\sum_{t=1}^{T}A\tran_{1}\cdots A\tran_{t-1} \Theta_t A_{t-1}\cdots A_1 =\\
&\qquad\sum_{t=1}^{T} A_1\tran \cdots A_t\tran Q_t A_t\cdots A_1-N_1.\qedhere
\eeal
\end{proof}

\begin{mylemma}\label{lem:sum_theta}
Consider for any $\allT$ that $A_t$ is invertible.
$\sum_{t=1}^{T}A\tran_{1}A\tran_2\cdots A\tran_{t-1} \Theta_t A_{t-1}A_{t-2}\cdots A_1\succ 0$ if and only if  
$\sum_{t=1}^{T} \Theta_t \succ 0.$
\end{mylemma}
\begin{proof}[Proof of Lemma~\ref{lem:sum_theta}] Let $U_t= A_{t-1}A_{t-2}\cdots A_1$.

We first prove that for any non-zero vector $z$, if it is $\sum_{t=1}^{T}A\tran_{1}A\tran_2\cdots A\tran_{t-1} \Theta_t A_{t-1}A_{t-2}\cdots A_1\succ 0$, then $\sum_{t=1}^{T} z\tran\Theta_t z> 0$.  Particularly, since $U_t$ is invertible, ---because for any $t\in \{1,2,\ldots, T\}$, $A_t$ is,---
\beal\label{eq:lem_sum_theta_aux_1}
\sum_{t=1}^{T} z\tran\Theta_t z&=\sum_{t=1}^{T} z\tran U_t^{-\top}U_t\tran\Theta_t U_t U_t\inv z\\
&=\sum_{t=1}^T\trace{\phi_t \phi_t\tran U_t\tran\Theta_t U_t},
\eeal
where we let $\phi_t\triangleq U_t\inv z$.  Consider a time $t'$ such that for any time $t\in \{1,2\ldots, T\}$, $\phi_{t'} \phi_{t'}\tran\preceq \phi_t \phi_t\tran$.  From eq.~\eqref{eq:lem_sum_theta_aux_1}, using Lemmata~\ref{lem:trace_low_bound_lambda_min} and~\ref{lem:traceAB_mon},
\begin{align*}
\textstyle\sum_{t=1}^{T} z\tran\Theta_t z&\textstyle\geq\sum_{t=1}^T\trace{\phi_{t'} \phi_{t'}\tran U_t\tran\Theta_t U_t}\\
\textstyle&\textstyle=\|\phi_{t'}\|_2^2 \lambda_\min(\sum_{t=1}^TU_t\tran\Theta_t U_t)>0.
\end{align*}

We finally prove that for any non-zero vector $z$, if  $\sum_{t=1}^{T} \Theta_t \succ 0$, then $\sum_{t=1}^{T}z A\tran_{1}\cdots A\tran_{t-1} \Theta_t A_{t-1}\cdots A_1z\succ 0$:
\begin{align}\label{eq:lem_sum_theta_aux_2}
\textstyle\sum_{t=1}^{T} z\tran U_t\tran\Theta_t U_t z&\textstyle=\sum_{t=1}^T\trace{ \xi_{t}\tran \Theta_t \xi_{t}},
\end{align}
where we let $\xi_t\triangleq U_tz$.  Consider time $t'$ such that for any time $t\in \{1,\ldots, T\}$, $\xi_{t'} \xi_{t'}\tran\preceq \xi_t \xi_t\tran$.  From eq.~\eqref{eq:lem_sum_theta_aux_1},
\belowdisplayskip =-12pt\begin{align*}
\textstyle\sum_{t=1}^T\trace{ \xi_{t}\tran \Theta_t \xi_{t}}&\textstyle\geq\trace{\xi_{t'} \xi_{t'}\tran \sum_{t=1}^T\Theta_t }\\
\textstyle&\textstyle=\|\xi_{t'}\|_2^2 \lambda_\min(\sum_{t=1}^T\Theta_t )>0.
\end{align*}
\end{proof}

\begin{proof}[Proof of Theorem~\ref{th:freq}]
Theorem~\ref{th:freq} follows from the sequential application of Lemmata~\ref{prop:iff_for_all_zero_control},~\ref{lem:observability_condition}, and~\ref{lem:sum_theta}.
\end{proof}

\bibliographystyle{IEEEtran}
\bibliography{references}

\begin{thebibliography}{10}
\providecommand{\url}[1]{#1}
\csname url@samestyle\endcsname
\providecommand{\newblock}{\relax}
\providecommand{\bibinfo}[2]{#2}
\providecommand{\BIBentrySTDinterwordspacing}{\spaceskip=0pt\relax}
\providecommand{\BIBentryALTinterwordstretchfactor}{4}
\providecommand{\BIBentryALTinterwordspacing}{\spaceskip=\fontdimen2\font plus
\BIBentryALTinterwordstretchfactor\fontdimen3\font minus
  \fontdimen4\font\relax}
\providecommand{\BIBforeignlanguage}[2]{{%
\expandafter\ifx\csname l@#1\endcsname\relax
\typeout{** WARNING: IEEEtran.bst: No hyphenation pattern has been}%
\typeout{** loaded for the language `#1'. Using the pattern for}%
\typeout{** the default language instead.}%
\else
\language=\csname l@#1\endcsname
\fi
#2}}
\providecommand{\BIBdecl}{\relax}
\BIBdecl

\bibitem{bertsekas2005dynamic}
D.~P. Bertsekas, \emph{Dynamic programming and optimal control,
  Vol.~{I}}.\hskip 1em plus 0.5em minus 0.4em\relax Athena Scientific, 2005.

\bibitem{abdelzaher2018toward}
T.~Abdelzaher, N.~Ayanian, T.~Basar, S.~Diggavi, J.~Diesner, D.~Ganesan,
  R.~Govindan, S.~Jha, T.~Lepoint, B.~Marlin \emph{et~al.}, ``Toward an
  internet of battlefield things: A resilience perspective,'' \emph{Computer},
  vol.~51, no.~11, pp. 24--36, 2018.

\bibitem{michael2011cooperative}
N.~Michael, J.~Fink, and V.~Kumar, ``Cooperative manipulation and
  transportation with aerial robots,'' \emph{Autonomous Robots}, vol.~30,
  no.~1, pp. 73--86, 2011.

\bibitem{prorok2017impact}
A.~Prorok, M.~A. Hsieh, and V.~Kumar, ``The impact of diversity on optimal
  control policies for heterogeneous robot swarms,'' \emph{IEEE Transactions on
  Robotics}, vol.~33, no.~2, pp. 346--358.

\bibitem{gupta2006stochastic}
V.~Gupta, T.~H. Chung, B.~Hassibi, and R.~M. Murray, ``On a stochastic sensor
  selection algorithm with applications in sensor scheduling and sensor
  coverage,'' \emph{Automatica}, vol.~42, no.~2, pp. 251--260, 2006.

\bibitem{carlone2017attention}
L.~Carlone and S.~Karaman, ``Attention and anticipation in fast visual-inertial
  navigation,'' in \emph{IEEE International Conference on Robotics and
  Automation}, 2017, pp. 3886--3893.

\bibitem{iwaki2018lqg}
T.~Iwaki and K.~H. Johansson, ``Lqg control and scheduling co-design for
  wireless sensor and actuator networks,'' in \emph{IEEE International Workshop
  on Signal Processing Advances in Wireless Communications}, 2018, pp. 1--5.

\bibitem{Nair07ieee-rateConstrainedControl}
G.~Nair, F.~Fagnani, S.~Zampieri, and R.~Evans, ``Feedback control under data
  rate constraints: An overview,'' \emph{Proceedings of the IEEE}, vol.~95,
  no.~1, pp. 108--137, 2007.

\bibitem{Baillieul07ieee-networkedControl}
J.~Baillieul and P.~Antsaklis, ``Control and communication challenges in
  networked real-time systems,'' \emph{Proceedings of the IEEE}, vol.~95,
  no.~1, pp. 9--28, 2007.

\bibitem{Elia01tac-limitedInfoControl}
N.~Elia and S.~Mitter, ``Stabilization of linear systems with limited
  information,'' \emph{IEEE Trans. on Automatic Control}, vol.~46, no.~9, pp.
  1384--1400, 2001.

\bibitem{Nair04sicon-rateConstrainedControl}
G.~Nair and R.~Evans, ``Stabilizability of stochastic linear systems with
  finite feedback data rates,'' \emph{SIAM Journal on Control and
  Optimization}, vol.~43, no.~2, pp. 413--436, 2004.

\bibitem{Tatikonda04tac-limitedCommControl}
S.~Tatikonda and S.~Mitter, ``Control under communication constraints,''
  \emph{IEEE Trans. on Automatic Control}, vol.~49, no.~7, pp. 1056--1068,
  2004.

\bibitem{Borkar97-limitedCommControl}
V.~Borkar and S.~Mitter, ``{LQG} control with communication constraints,''
  \emph{Comm., Comp., Control, and Signal Processing}, pp. 365--373, 1997.

\bibitem{LeNy14tac-limitedCommControl}
J.~L. Ny and G.~Pappas, ``Differentially private filtering,'' \emph{IEEE Trans.
  on Automatic Control}, vol.~59, no.~2, pp. 341--354, 2014.

\bibitem{lin2017sparse}
F.~Lin and V.~Adetola, ``Sparse output feedback synthesis via proximal
  alternating linearization method,'' \emph{arXiv preprint:1706.08191}, 2017.

\bibitem{lin2011augmented}
F.~Lin, M.~Fardad, and M.~R. Jovanovic, ``Augmented lagrangian approach to
  design of structured optimal state feedback gains,'' \emph{IEEE Transactions
  on Automatic Control}, vol.~56, no.~12, pp. 2923--2929, 2011.

\bibitem{lin2013design}
F.~Lin, M.~Fardad, and M.~R. Jovanovi{\'c}, ``Design of optimal sparse feedback
  gains via the alternating direction method of multipliers,'' \emph{IEEE
  Transactions on Automatic Control}, vol.~58, no.~9, pp. 2426--2431, 2013.

\bibitem{zare2018proximal}
A.~Zare, H.~Mohammadi, N.~K. Dhingra, M.~R. Jovanovi{\'c}, and T.~T. Georgiou,
  ``Proximal algorithms for large-scale statistical modeling and optimal
  sensor/actuator selection,'' \emph{arXiv preprint:1807.01739}, 2018.

\bibitem{liu2017decentralized}
T.~Liu, S.~Azarm, and N.~Chopra, ``On decentralized optimization for a class of
  multisubsystem codesign problems,'' \emph{Journal of Mechanical Design}, vol.
  139, no.~12, p. 121404, 2017.

\bibitem{tanaka2015sdp}
T.~Tanaka and H.~Sandberg, ``{SDP}-based joint sensor and controller design for
  information-regularized optimal {LQG} control,'' in \emph{IEEE Conference on
  Decision and Control}, 2015, pp. 4486--4491.

\bibitem{tanaka2018lqg}
T.~Tanaka, P.~M. Esfahani, and S.~K. Mitter, ``{LQG} control with minimum
  directed information: {S}emidefinite programming approach,'' \emph{IEEE
  Trans.~on Automatic Control}, vol.~63, no.~1, pp. 37--52, 2018.

\bibitem{joshi2009sensor}
S.~Joshi and S.~Boyd, ``Sensor selection via convex optimization,'' \emph{IEEE
  Transactions on Signal Processing}, vol.~57, no.~2, pp. 451--462, 2009.

\bibitem{leny2011kalman}
J.~L. Ny, E.~Feron, and M.~A. Dahleh, ``Scheduling continuous-time kalman
  filters,'' \emph{IEEE Trans. on Aut. Control}, vol.~56, no.~6, pp.
  1381--1394, 2011.

\bibitem{jawaid2015submodularity}
S.~T. Jawaid and S.~L. Smith, ``Submodularity and greedy algorithms in sensor
  scheduling for linear dynamical systems,'' \emph{Automatica}, vol.~61, pp.
  282--288, 2015.

\bibitem{zhao2016scheduling}
Y.~Zhao, F.~Pasqualetti, and J.~Cort{\'e}s, ``Scheduling of control nodes for
  improved network controllability,'' in \emph{IEEE Conf. on Decision and
  Control}, 2016, pp. 1859--1864.

\bibitem{chamon2017mean}
L.~F. Chamon, G.~J. Pappas, and A.~Ribeiro, ``The mean square error in kalman
  filtering sensor selection is approximately supermodular,'' in \emph{IEEE
  Conf. on Decision and Control}, 2017, pp. 343--350.

\bibitem{tzoumas2015sensor}
V.~Tzoumas, A.~Jadbabaie, and G.~J. Pappas, ``Sensor placement for optimal
  {K}alman filtering,'' in \emph{Amer.~Contr.~Conf.}, 2016, pp. 191--196.

\bibitem{clark2012leader}
A.~Clark, L.~Bushnell, and R.~Poovendran, ``On leader selection for performance
  and controllability in multi-agent systems,'' in \emph{IEEE 51st IEEE
  Conference on Decision and Control}, 2012, pp. 86--93.

\bibitem{clark2017input}
A.~Clark, B.~Alomair, L.~Bushnell, and R.~Poovendran, ``Input selection for
  performance and controllability of structured linear descriptor systems,''
  \emph{SIAM Journal on Control and Optimization}, vol.~55, no.~1, pp.
  457--485, 2017.

\bibitem{liu2017minimal}
Z.~Liu, Y.~Long, A.~Clark, P.~Lee, L.~Bushnell, D.~Kirschen, and R.~Poovendran,
  ``Minimal input selection for robust control,'' in \emph{IEEE 56th Annual
  Conference on Decision and Control}, 2017, pp. 2659--2966.

\bibitem{pequito2015framework}
S.~Pequito, S.~Kar, and A.~P. Aguiar, ``A framework for structural input/output
  and control configuration selection in large-scale systems,'' \emph{IEEE
  Transactions on Automatic Control}, vol.~61, no.~2, pp. 303--318, 2015.

\bibitem{summers2016submodularity}
T.~H. Summers, F.~L. Cortesi, and J.~Lygeros, ``On submodularity and
  controllability in complex dynamical networks,'' \emph{IEEE Transactions on
  Control of Network Systems}, vol.~3, no.~1, pp. 91--101, 2016.

\bibitem{tzoumas2015minimal}
V.~Tzoumas, M.~A. Rahimian, G.~J. Pappas, and A.~Jadbabaie, ``Minimal actuator
  placement with bounds on control effort,'' \emph{IEEE Transactions on Control
  of Network Systems}, vol.~3, no.~1, pp. 67--78, 2015.

\bibitem{summers2017performance}
T.~Summers and M.~Kamgarpour, ``Performance guarantees for greedy maximization
  of non-submodular set functions in systems and control,'' \emph{arXiv
  preprint:1712.04122}, 2017.

\bibitem{summers2017performance2}
T.~Summers and J.~Ruths, ``Performance bounds for optimal feedback control in
  networks,'' in \emph{American Control Conference}, 2018, pp. 203--209.

\bibitem{nozari2017time}
E.~Nozari, F.~Pasqualetti, and J.~Cort{\'e}s, ``Time-invariant versus
  time-varying actuator scheduling in complex networks,'' in \emph{American
  Control Conference}, 2017, pp. 4995--5000.

\bibitem{taha2017time}
A.~F. Taha, N.~Gatsis, T.~Summers, and S.~Nugroho, ``Time-varying sensor and
  actuator selection for uncertain cyber-physical systems,'' \emph{IEEE
  Transactions on Control of Network Systems}, vol.~6, no.~2, pp. 750--762,
  2019.

\bibitem{clark2016submodularity}
A.~Clark, B.~Alomair, L.~Bushnell, and R.~Poovendran, \emph{Submodularity in
  dynamics and control of networked systems}.\hskip 1em plus 0.5em minus
  0.4em\relax Springer, 2017.

\bibitem{chakrabortty2011optimal}
A.~{Chakrabortty} and C.~F. {Martin}, ``Optimal measurement allocation
  algorithms for parametric model identification of power systems,'' \emph{IEEE
  Transactions on Control Systems Technology}, vol.~22, no.~5, pp. 1801--1812,
  2014.

\bibitem{moreno2015actuator}
C.~P. Moreno, H.~Pfifer, and G.~J. Balas, ``Actuator and sensor selection for
  robust control of aeroservoelastic systems,'' in \emph{American Control
  Conference}, 2015, pp. 1899--1904.

\bibitem{lim1992method}
K.~Lim, ``Method for optimal actuator and sensor placement for large flexible
  structures,'' \emph{Journal of Guidance, Control, and Dynamics}, vol.~15,
  no.~1, pp. 49--57, 1992.

\bibitem{zelazo2010graph}
D.~Zelazo, ``Graph-theoretic methods for the analysis and synthesis of
  networked dynamic systems,'' Ph.D. dissertation, University of Washington.

\bibitem{das2011submodular}
A.~Das and D.~Kempe, ``Submodular meets spectral: {G}reedy algorithms for
  subset selection, sparse approximation and dictionary selection,'' in
  \emph{Intl. Conf. on Machine Learning}, 2011, pp. 1057--1064.

\bibitem{wang2016approximation}
Z.~Wang, B.~Moran, X.~Wang, and Q.~Pan, ``Approximation for maximizing monotone
  non-decreasing set functions with a greedy method,'' \emph{Journal of
  Combinatorial Optimization}, vol.~31, no.~1, pp. 29--43, 2016.

\bibitem{sviridenko2013optimal}
M.~Sviridenko, J.~Vondr{\'a}k, and J.~Ward, ``Optimal approximation for
  submodular and supermodular optimization with bounded curvature,''
  \emph{arXiv preprint:1311.4728}, 2013.

\bibitem{sviridenko2017optimal}
------, ``Optimal approximation for submodular and supermodular optimization
  with bounded curvature,'' \emph{Mathematics of Operations Research}, vol.~42,
  no.~4, pp. 1197--1218, 2017.

\bibitem{nemhauser78analysis}
G.~Nemhauser, L.~Wolsey, and M.~Fisher, ``An analysis of approximations for
  maximizing submodular set functions -- {I},'' \emph{Mathematical
  Programming}, vol.~14, no.~1, pp. 265--294, 1978.

\bibitem{wolsey1982analysis}
L.~A. Wolsey, ``An analysis of the greedy algorithm for the submodular set
  covering problem,'' \emph{Combinatorica}, vol.~2, no.~4, pp. 385--393, 1982.

\bibitem{khuller1999budgeted}
S.~Khuller, A.~Moss, and J.~S. Naor, ``The budgeted maximum coverage problem,''
  \emph{Info.~Processing Letters}, vol.~70, no.~1, pp. 39--45, 1999.

\bibitem{sviridenko2004}
M.~Sviridenko, ``A note on maximizing a submodular set function subject to a
  knapsack constraint,'' \emph{Operations Research Letters}, vol.~32, no.~1,
  pp. 41--43, 2004.

\bibitem{krause2005note}
A.~Krause and C.~Guestrin, ``A note on the budgeted maximization of submodular
  functions,'' 2005.

\bibitem{tzoumas2018sensing}
V.~{Tzoumas}, L.~{Carlone}, G.~J. {Pappas}, and A.~{Jadbabaie},
  ``Sensing-constrained {LQG} control,'' in \emph{American Control Conference},
  2018.

\bibitem{lehmann2006combinatorial}
B.~Lehmann, D.~Lehmann, and N.~Nisan, ``Combinatorial auctions with decreasing
  marginal utilities,'' \emph{Games and Economic Behavior}, vol.~55, no.~2, pp.
  270--296, 2006.

\bibitem{chamon2016near}
L.~F. Chamon and A.~Ribeiro, ``Near-optimality of greedy set selection in the
  sampling of graph signals,'' in \emph{IEEE Global Conference on Signal and
  Information Processing}, 2016, pp. 1265--1269.

\bibitem{hashemi2019submodular}
A.~Hashemi, M.~Ghasemi, H.~Vikalo, and U.~Topcu, ``Submodular observation
  selection and information gathering for quadratic models,'' in \emph{Intl.
  Conf. on Machine Learning}, 2019, pp. 2653--2662.

\bibitem{guo2019actuator}
B.~Guo, O.~Karaca, T.~Summers, and M.~Kamgarpour, ``Actuator placement for
  optimizing network performance under controllability constraints,''
  \emph{arXiv preprint:1903.08120}, 2019.

\bibitem{feige1998}
U.~Feige, ``A threshold of $ln(n)$ for approximating set cover,'' \emph{Journal
  of the ACM}, vol.~45, no.~4, pp. 634--652, 1998.

\bibitem{sviridenko2004note}
M.~Sviridenko, ``A note on maximizing a submodular set function subject to a
  knapsack constraint,'' \emph{Operations Research Letters}, vol.~32, no.~1,
  pp. 41--43, 2004.

\bibitem{nguyen2013budgeted}
H.~Nguyen and R.~Zheng, ``On budgeted influence maximization in social
  networks,'' \emph{IEEE Journal on Selected Areas in Communications}, vol.~31,
  no.~6, pp. 1084--1094, 2013.

\bibitem{iyer2013submodular}
R.~K. Iyer and J.~A. Bilmes, ``Submodular optimization with submodular cover
  and submodular knapsack constraints,'' in \emph{Advances in Neural
  Information Processing Systems}, 2013, pp. 2436--2444.

\bibitem{zhang2016submodular}
H.~Zhang and Y.~Vorobeychik, ``Submodular optimization with routing
  constraints,'' in \emph{AAAI Conference on Artificial Intelligence}, 2016.

\bibitem{qian2017subset}
C.~Qian, J.-C. Shi, Y.~Yu, and K.~Tang, ``On subset selection with general cost
  constraints,'' in \emph{International Joint Conference on Artificial
  Intelligence}, 2017, pp. 2613--2619.

\bibitem{leskovec2007cost}
J.~Leskovec, A.~Krause, C.~Guestrin, C.~Faloutsos, C.~Faloutsos, J.~VanBriesen,
  and N.~Glance, ``Cost-effective outbreak detection in networks,'' in
  \emph{ACM SIGKDD international conference on Knowledge discovery and data
  mining}, 2007, pp. 420--429.

\bibitem{Tzoumas18acc-scLGQG}
V.~{Tzoumas}, L.~{Carlone}, G.~J. {Pappas}, and A.~{Jadbabaie},
  ``{Sensing-constrained LQG Control},'' \emph{arXiv preprint: 1709.08826},
  2017.

\bibitem{tzoumas2018lqg}
------, ``{LQG} control and sensing co-design,'' \emph{arXiv
  preprint:1802.08376}, 2018.

\bibitem{ye2018complexity}
L.~Ye, S.~Roy, and S.~Sundaram, ``On the complexity and approximability of
  optimal sensor selection for {K}alman filtering,'' in \emph{American Control
  Conference}, 2018, pp. 5049--5054.

\bibitem{krause2010submodular}
A.~Krause and V.~Cevher, ``Submodular dictionary selection for sparse
  representation,'' in \emph{International Conference on Machine Learning},
  2010.

\bibitem{siam2018ecc}
M.~{Siami} and A.~{Jadbabaie}, ``Deterministic polynomial-time actuator
  scheduling with guaranteed performance,'' in \emph{European Control
  Conference}, 2018, pp. 113--118.

\bibitem{tzoumas2017resilient}
V.~Tzoumas, K.~Gatsis, A.~Jadbabaie, and G.~J. Pappas, ``Resilient monotone
  submodular maximization,'' in \emph{IEEE Conf. on Decision and Control},
  2017.

\bibitem{bernstein2005matrix}
D.~S. Bernstein, \emph{Matrix mathematics}.\hskip 1em plus 0.5em minus
  0.4em\relax Princeton University Press, 2005.

\bibitem{tzoumas2016near}
V.~Tzoumas, A.~Jadbabaie, and G.~J. Pappas, ``Near-optimal sensor scheduling
  for batch state estimation: Complexity, algorithms, and limits,'' in
  \emph{IEEE Conference on Decision and Control}, 2016, pp. 2695--2702.

\end{thebibliography}

{
\begin{biography}[{\includegraphics[width=1in,height=1.25in,keepaspectratio]{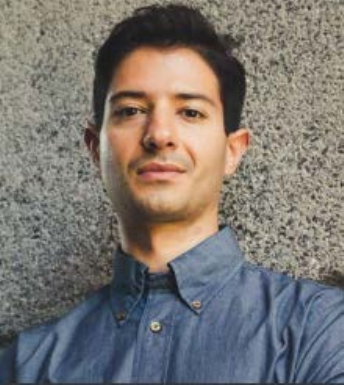}}]{Vasileios Tzoumas}  (S'12-M'18) starts as an Assistant Professor at the University of Michigan, Ann Arbor, on January 2021.  Currently, he is a research scientist at the Department of Aeronautics and Astronautics, and the Laboratory for Information and Decision Systems (LIDS) at MIT, where he previously was a post-doctoral associate. He received his Ph.D.~at the Department of Electrical and Systems Engineering, University of Pennsylvania (2018). He was a visiting Ph.D.~student at the MIT Institute for Data, Systems, and Society in 2017.  He holds a diploma in Electrical and Computer Engineering from the National Technical University of Athens (2012); a Master of Science in Electrical Engineering from the University of Pennsylvania (2016); and a Master of Arts in Statistics from the Wharton School of Business at the University of Pennsylvania (2016). His research interests include control theory, perception, learning, and combinatorial and non-convex optimization, with applications to robotics, resource-constrained cyber-physical systems,  and unmanned aerospace systems.  He aims for a provably trustworthy autonomy.  
	His work includes seminal results on provably optimal resilient combinatorial optimization, with applications to  multi-robot information gathering for resiliency against robotic failures and adversarial removals.
	Dr.~Tzoumas was a Best Student Paper Award finalist at the 56th IEEE Conference in Decision and Control (2017), and a Best Paper Award finalist in Robotic Vision at the 2020 IEEE International Conference on Robotics and Automation (ICRA). 
\end{biography}

\begin{biography}[{\includegraphics[width=1in,height=1.25in,clip,keepaspectratio]{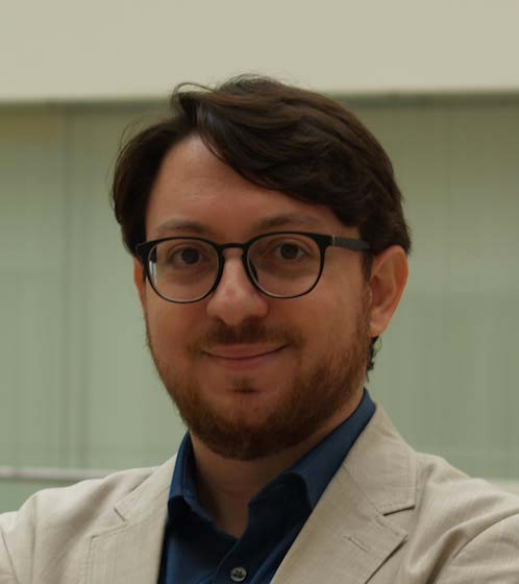}}]{Luca Carlone} is the Charles Stark Draper Assistant Professor in the Department of Aeronautics
	and Astronautics at the Massachusetts Institute of
	Technology, and a Principal Investigator in the Laboratory for Information \& Decision Systems (LIDS).
	He has obtained a B.S.~degree in mechatronics
	from the Polytechnic University of Turin, Italy, in
	2006; an S.M.~degree in mechatronics from the
	Polytechnic University of Turin, Italy, in 2008; an
	S.M.~degree in automation engineering from the
	Polytechnic University of Milan, Italy, in 2008; and
	a Ph.D.~degree in robotics also the Polytechnic University of Turin in 2012.
	He joined LIDS as a postdoctoral associate (2015) and later as a Research
	Scientist (2016), after spending two years as a postdoctoral fellow at the
	Georgia Institute of Technology (2013-2015). His research interests include
	nonlinear estimation, numerical and distributed optimization, and probabilistic
	inference, applied to sensing, perception, and decision-making in single and
	multi-robot systems. His work includes seminal results on certifiably correct
	algorithms for localization and mapping, as well as approaches for visual inertial navigation and distributed mapping. He is a recipient of the 2017
	Transactions on Robotics King-Sun Fu Memorial Best Paper Award, the
	Best Paper award at WAFR 2016, the Best Student Paper award at the 2018
	Symposium on VLSI Circuits, and was best paper finalist at RSS 2015 and ICRA 2020.
\end{biography}

\vskip -2\baselineskip plus -1fil
\begin{biography}[{\includegraphics[width=1in,height=1.25in,clip,keepaspectratio]{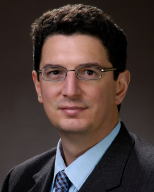}}]{George J.~Pappas} (S'90-M'91-SM'04-F'09) received the Ph.D.~degree in electrical engineering and computer sciences from the University of California, Berkeley, CA, USA, in  1998. He is currently the Joseph Moore Professor and Chair of the Department of Electrical and Systems Engineering, University of Pennsylvania, Philadelphia, PA, USA. He  also holds a secondary appointment with the Department of Computer and Information Sciences and the Department of Mechanical Engineering and Applied Mechanics. He is a Member of the GRASP Lab and the PRECISE Center. He had previously served as the Deputy Dean for Research with the School of Engineering and Applied Science. His research interests include control theory and, in particular, hybrid systems, embedded systems, cyberphysical systems, and hierarchical and distributed control systems, with applications to unmanned aerial vehicles, distributed robotics, green buildings, and biomolecular networks. Dr. Pappas has received various awards, such as the Antonio Ruberti Young Researcher Prize, the George S. Axelby Award, the Hugo Schuck Best Paper Award, the George H. Heilmeier Award, the National Science Foundation PECASE award and numerous best student papers awards.
\end{biography}

\vskip -2\baselineskip plus -1fil
\begin{biography}[{\includegraphics[width=1in,height=1.25in,clip,keepaspectratio]{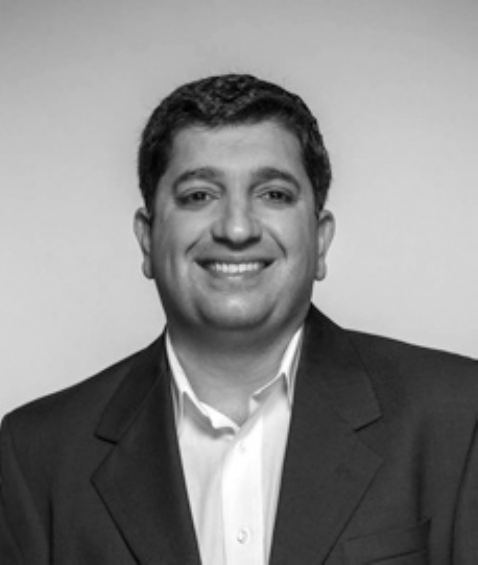}}] {Ali Jadbabaie} (S'99-M'08-SM'13-F'15) is the JR East Professor of Engineering and Associate Director of the Institute for Data, Systems and Society at MIT, where he is also on the faculty of the department of civil and environmental engineering and a principal investigator in the Laboratory for Information and Decision Systems (LIDS). He is the director of the Sociotechnical Systems Research Center, one of MIT's 13 laboratories. He received his Bachelors (with high honors) from Sharif University of Technology in Tehran, Iran, a Masters degree in electrical and computer engineering from the University of New Mexico, and his Ph.D.~in control and dynamical systems from the California Institute of Technology. He was a postdoctoral scholar at Yale University before joining the faculty at Penn in July 2002. Prior to joining MIT faculty, he was the Alfred Fitler Moore a Professor of Network Science and held secondary appointments in computer and information science and operations, information and decisions in the Wharton School. He was the inaugural editor-in-chief of IEEE Transactions on Network Science and Engineering, a new interdisciplinary journal sponsored by several IEEE societies. He is a recipient of a National Science Foundation Career Award, an Office of Naval Research Young Investigator Award, the O. Hugo Schuck Best Paper Award from the American Automatic Control Council, and the George S. Axelby Best Paper Award from the IEEE Control Systems Society. His students have been winners and finalists of student best paper awards at various ACC and CDC conferences. He is an IEEE fellow and a recipient of the Vannevar Bush Fellowship from the office of Secretary of Defense. His current research interests include the interplay of dynamic systems and networks with specific emphasis on multi-agent coordination and control, distributed optimization, network science, and network economics.
\end{biography}
}

\end{document}